\pdfoutput=1
\documentclass[11pt]{article}
\usepackage[T1]{fontenc}
\usepackage{lmodern}

\usepackage{amsmath,amsthm,amssymb}
\usepackage[usenames,dvipsnames]{xcolor}
\usepackage{enumerate}
\usepackage{graphicx}
\usepackage{cite}
\usepackage{comment}
\usepackage{oands}
\usepackage{tikz}
\usepackage{changepage}
\usepackage{bbm}
\usepackage{mathtools}
\usepackage[margin=1in]{geometry}
\usepackage[pagewise,mathlines]{lineno}
\usepackage{appendix}
\usepackage{stmaryrd}
\usepackage{multicol}
\usepackage{microtype}
\usepackage[colorinlistoftodos]{todonotes}
\usepackage{dsfont}

\usepackage[pdftitle={Permutons, meanders, and SLE-decorated Liouville quantum gravity},
  pdfauthor={Jacopo Borga, Ewain Gwynne, and Xin Sun},
colorlinks=true,linkcolor=NavyBlue,urlcolor=RoyalBlue,citecolor=PineGreen,bookmarks=true,bookmarksopen=true,bookmarksopenlevel=2,unicode=true,linktocpage]{hyperref}

\theoremstyle{plain}
\newtheorem{thm}{Theorem}[section]
\newtheorem{cor}[thm]{Corollary}
\newtheorem{lem}[thm]{Lemma}
\newtheorem{prop}[thm]{Proposition}
\newtheorem{conj}[thm]{Conjecture}

\def\@rst #1 #2other{#1}
\newcommand\MR[1]{\relax\ifhmode\unskip\spacefactor3000 \space\fi
  \MRhref{\expandafter\@rst #1 other}{#1}}
\newcommand{\MRhref}[2]{\href{http://www.ams.org/mathscinet-getitem?mr=#1}{MR#2}}

\theoremstyle{definition}
\newtheorem{defn}[thm]{Definition}
\newtheorem{remark}[thm]{Remark}

\newtheorem{prob}[thm]{Problem}

\numberwithin{equation}{section}

\newcommand{\dsb}{\begin{adjustwidth}{2.5em}{0pt}
\begin{footnotesize}}
\newcommand{\dse}{\end{footnotesize}
\end{adjustwidth}}

\newcommand{\ssb}{\begin{adjustwidth}{2.5em}{0pt}}
\newcommand{\sse}{\end{adjustwidth}}

\newcommand{\aryb}{\begin{eqnarray*}}
\newcommand{\arye}{\end{eqnarray*}}
\def\alb#1\ale{\begin{align*}#1\end{align*}}
\def\allb#1\alle{\begin{align}#1\end{align}}
\newcommand{\eqb}{\begin{equation}}
\newcommand{\eqe}{\end{equation}}
\newcommand{\eqbn}{\begin{equation*}}
\newcommand{\eqen}{\end{equation*}}

\newcommand{\BB}{\mathbbm}
\newcommand{\ol}{\overline}

\newcommand{\op}{\operatorname}

\newcommand{\frk}{\mathfrak}
\newcommand{\eqD}{\overset{d}{=}}
\newcommand{\ep}{\varepsilon}
\newcommand{\rta}{\rightarrow}

\newcommand{\wt}{\widetilde}
\newcommand{\wh}{\widehat}
\newcommand{\mcl}{\mathcal}

\newcommand{\bdy}{\partial}

\newcommand{\perm}{{\boldsymbol{\pi}}}

\newcommand{\cc}{{\mathbf{c}}}

\newcommand{\ccM}{{\mathbf{c}_{\mathrm M}}}

\DeclareMathOperator{\Perm}{Perm}
\DeclareMathOperator{\occ}{occ}
\newcommand{\pocc}{\widetilde{\occ}}

\let\originalleft\left
\let\originalright\right
\renewcommand{\left}{\mathopen{}\mathclose\bgroup\originalleft}
\renewcommand{\right}{\aftergroup\egroup\originalright}

\title{Permutons, meanders, and SLE-decorated Liouville quantum gravity}
 \date{ }
 \author{
\begin{tabular}{c} Jacopo Borga\\[-3pt]\small MIT \end{tabular}
\begin{tabular}{c} Ewain Gwynne\\[-3pt]\small University of Chicago \end{tabular} 
\begin{tabular}{c} Xin Sun\\[-3pt]\small Peking University \end{tabular} 
}

\begin{document}

\maketitle

\begin{abstract}
We study a class of random permutons which can be constructed from a pair of space-filling Schramm-Loewner evolution (SLE) curves on a Liouville quantum gravity (LQG) surface. This class includes the skew Brownian permutons introduced by Borga (2021), which describe the scaling limit of various types of random pattern-avoiding permutations. Another interesting permuton in our class is the meandric permuton, which corresponds to two independent SLE$_8$ curves on a $\gamma$-LQG surface with $\gamma  = \sqrt{\frac13 \left( 17 - \sqrt{145} \right)}$. Building on work by Di Francesco, Golinelli, and Guitter (2000), we conjecture that the meandric permuton describes the scaling limit of uniform meandric permutations, i.e., the permutations induced by a simple loop in the plane which crosses a line a specified number of times.

We show that for any sequence of random permutations which converges to one of the above random permutons, the length of the longest increasing subsequence is sublinear. This proves that the length of the longest increasing subsequence is sublinear for Baxter, strong-Baxter, and semi-Baxter permutations and leads to the conjecture that the same is true for meandric permutations. We also prove that the closed support of each of the random permutons in our class has Hausdorff dimension one. Finally, we prove a re-rooting invariance property for the meandric permuton and write down a formula for its expected pattern densities in terms of LQG correlation functions (which are known explicitly) and the probability that an SLE$_8$ hits a given set of points in numerical order (which is not known explicitly). We conclude with a list of open problems. 
\end{abstract}

%KEYWORDS: Permuton, permutation, longest increasing subsequence, meander, meandric permutation, Schramm-Loewner evolution, Liouville quantum gravity
 
%AMS SUBJECT CLASS: 60J67 (SLE), 05A05 (permutations), 60D05 (geometric probability)

\tableofcontents

\section{Introduction}
\label{sec-intro}

\subsection{Overview} 
\label{sec-overview}

A \textbf{permuton} is a probability measure $\perm$ on $[0,1]^2$ whose one-dimensional marginals are each equal to Lebesgue measure on $[0,1]$. 
Permutons are of interest because they describe the scaling limits of various types of random permutations (see \cite[Section 1.1]{borga-skew-permuton} and the references therein) and they identify a natural space where one can study large deviations principles for permutations (as done in \cite{Mukherjee-estimation,Kenyon-Permutations-densities,Borga-LDP-perm}). 
To be more precise, each permutation is associated to a permuton as follows. 
 
\begin{defn} \label{def-permutation}
For a permutation $\sigma$ on $[1,n] \cap\BB Z$, we define the \textbf{size} of $\sigma$ by $|\sigma| = n$. We use the one-line notation to write permutations, that is, if $\sigma$ is a permutation of size $n$ then we write $\sigma=\sigma(1),\dots,\sigma(n)$. 
We define the \textbf{permuton associated with $\sigma$} to be the measure $\perm_\sigma$ on $[0,1]^2$ which is equal to $|\sigma| $ times Lebesgue measure on
\eqbn
\bigcup_{j =1}^{|\sigma|} \left[ \frac{j-1}{|\sigma|} , \frac{j}{|\sigma|} \right] \times \left[ \frac{\sigma(j)-1}{|\sigma|} , \frac{\sigma(j)}{|\sigma|} \right] .
\eqen
\end{defn}

\begin{figure}[ht!]
\begin{center}
\includegraphics[width=.3\textwidth]{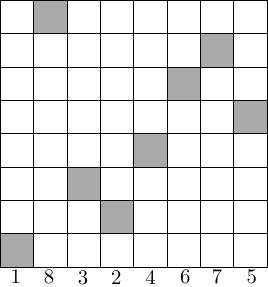}  
\caption{\label{fig-permuton} The permuton $\perm_\sigma$ associated with the permutation $\sigma=1,8,3,2,4,6,7,5$ is equal to 8 times the Lebesgue measure on the gray subset of $[0,1]^2$.}
\end{center}
\vspace{-3ex}
\end{figure}

See Figure~\ref{fig-permuton} for an illustration. 
We equip the space of permutons with the usual weak topology on measures, and we say that a sequence of permutations converges to a permuton $\perm$ if their associated permutons converge weakly. We refer to~\cite[Remark 2.1.2]{borga-thesis} for a brief history of the theory of permutons.

In this paper, we will study a certain class of random permutons which we define precisely in Section~\ref{sec-permuton-def}. We will be especially interested in two special cases.
\begin{itemize}
\item The \textbf{meandric permuton} is a random permuton which we define in Definition~\ref{def-meander-permuton} and which we conjecture describes the scaling limit of uniform meandric permutations, i.e., the  permutations induced by a pair of intersecting Jordan curves (see Section~\ref{sec-meander} for more details).
\item The \textbf{skew Brownian permutons} are a two-parameter family of random permutons introduced by Borga in~\cite{borga-skew-permuton} which describe the scaling limits of various types of random pattern-avoiding permutations (see Section~\ref{sec-skew-permuton} for more details).
\end{itemize}  
For various subsets of our class of permutons, including the two special cases listed above, we will show that the length of the longest increasing subsequence for any sequence of permutations converging to the permuton is sublinear (Theorem~\ref{thm-lis}); and that the dimension of the closed support of the permuton is one (Theorem~\ref{thm-permuton-dim}). We will also establish a re-rooting property for the meandric permuton (Theorem~\ref{thm-permuton-re-root}) and a formula for its pattern densities in terms of hitting probabilities for SLE$_8$ and the correlation functions of the LQG area measure (Theorem~\ref{thm:density}). 

The common feature of all the permutons we consider is that they can be represented in terms of a \textbf{Liouville quantum gravity (LQG)} surface decorated by a pair of coupled \textbf{Schramm-Loewner evolution (SLE)} curves.  
There is a vast literature concerning SLE and LQG, but not much prior knowledge of this literature is needed to understand this paper. All of the necessary background knowledge about SLE and LQG is explained in Section~\ref{sec-background}. Our proofs will only use relatively simple statements about these objects which we re-state as necessary and which can be taken as black boxes. 
\bigskip

\noindent\textbf{Acknowledgments.} {We thank two anonymous referees and Yuanzheng Wang for helpful comments on an earlier version of this paper.} We thank Rick Kenyon for helpful discussions on the uniform spanning tree. We thank Tony Guttmann and Iwan Jensen for helpful discussions on numerical simulations. We thank Zhihan Li for coding the Markov chain Monte Carlo algorithm from \cite{heitsch2011meander} which was used to obtain the pictures in Figure~\ref{sim-menders-permuton} and the numerical simulations for Conjecture~\ref{conj-meander-lis-precise}.
Part of the work on this paper was carried out during visits by J.B.\ to Chicago and Penn, by E.G.\ to Stanford, and by X.S.\ to Stanford. We thank these institutions for their hospitality. J.B.\ was partially supported by NSF grant DMS-2441646. E.G.\ was partially supported by a Clay research fellowship. X.S. was supported by the NSF grant DMS-2027986 and the NSF Career Award 2046514. 

\subsection{Permutons constructed from SLEs and LQG} 
\label{sec-permuton-def}

Let us now define the family of permutons which we will consider. 
Fix parameters $\gamma \in (0,2)$ and $\kappa_1,\kappa_2 > 4$. Let $h$ be a random generalized function on $\BB C$ corresponding to a singly marked unit area $\gamma$-Liouville quantum sphere, and let $\mu_h$ be its associated $\gamma$-LQG area measure. We will review the definitions of these objects in Sections~\ref{sec-lqg} and~\ref{sec-sphere}, but for the time being the reader can just think of $\mu_h$ as a random, non-atomic, Borel probability measure on $\BB C$ which assigns positive mass to every open subset of $\BB C$. 

Independently from $h$, let $(\eta_1,\eta_2)$ be a random pair consisting of a whole-plane space-filling SLE$_{\kappa_1}$ curve and a whole-plane space-filling SLE$_{\kappa_2}$ curve, each going from $\infty$ to $\infty$. We will review the definition of these random curves in Section~\ref{sec-sle}, but for the time being the reader can just think of $(\eta_1,\eta_2)$ as a pair of random non-self-crossing space-filling curves in $\BB C$ which each visit almost every point of $\BB C$ exactly once. We parametrize each of $\eta_1$ and $\eta_2$ by $\mu_h$-mass, i.e., in such a way that $\mu_h(\eta_1([0,t])) = \mu_h(\eta_2([0,t])) = t$ for each $t \in [0,1]$.

We emphasize that the coupling of our two SLE curves, viewed modulo time parametrization, is arbitrary. We will be interested in cases where the two curves determine each other as well as the case where the two curves are independent.

%Note that since $\mu_h(\BB C) = 1$ and $\eta_1$ and $\eta_2$ start and end at $\infty$, we have $\eta_1(0) = \eta_1(1) = \eta_2(0) = \eta_2(1) = \infty$. 
Let $\psi : [0,1]\rta[0,1]$ be a Lebesgue measurable function such that 
\eqb \label{eqn-psi-property}
\eta_1(t) = \eta_2(\psi(t)),\quad\forall t \in [0,1] .
\eqe 
The \textbf{permuton associated with $(h,\eta_1,\eta_2)$} is the random probability measure $\perm$ on $[0,1]^2$ defined by
\eqb \label{eqn-permuton-def}
\perm(A) = \perm_{\eta_1,\eta_2}(A) = \op{Leb}\left\{ t\in [0,1] : (t,\psi(t)) \in A \right\} ,\quad \text{$\forall$ Borel sets $A\subset [0,1]^2$}, 
\eqe 
where $\op{Leb}$ denotes Lebesgue measure. Using the fact that $\mu_h$-a.e.\ point in $\BB C$ is hit only once by each of $\eta_1$ and $\eta_2$, is easy to see that the definition of $\perm$ does not depend on the choice of $\psi$ (see Lemma~\ref{lem-permuton-defined}). 

Before we state our main results, we discuss the two special cases of the construction in this subsection which are particularly interesting. See Figure~\ref{fig-venn-diagram} for a Venn diagram of the different permutons we will consider in this paper. 

{
\begin{remark}
Instead of parametrizing $\eta_1$ and $\eta_2$ by $\gamma$-LQG area, one could instead parametrize $\eta_1$ and $\eta_2$ by $\nu$-mass, where $\nu$ is a non-atomic probability measure on $\BB C$. Then the same construction as above would yield a different permuton which depends on $\nu$. In the case when $\nu$ is mutually absolutely continuous with respect to Lebesgue measure on $\BB C$, our proofs of Theorems~\ref{thm-lis} and~\ref{thm-permuton-dim} go through with minor modifications. However, we do not treat this case explicitly in this paper since we do not know of any natural permutation models which it is connected to. 
\end{remark}
}

\begin{figure}[ht!]
\begin{center}
\includegraphics[scale=1]{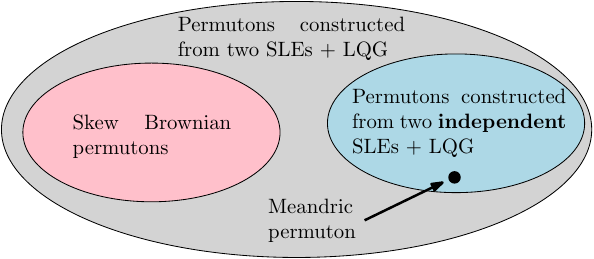}  
\caption{\label{fig-venn-diagram} Venn diagram of the different classes of random permutons considered in this paper. The outer grey oval represents all random permutons of the type constructed in \eqref{eqn-permuton-def}. The meandric permuton and the skew Brownian permutons are special cases of this construction. In particular, the meandric permuton falls in the subclass of permutons constructed from independent SLEs.
}
\end{center}
\vspace{-3ex}
\end{figure}

\subsection{Meanders, meandric permutations, and their conjectured connection to SLEs and LQG}
\label{sec-meander}

\begin{defn} \label{def-meander}
A \textbf{meander} of size $n\in\BB N$ is a simple loop $\ell$ (i.e., a homeomorphic image of the circle) in $\BB R^2$ which crosses the real line $\BB R  $ exactly $2n$ times and does not touch the line without crossing it, viewed modulo orientation-preserving homeomorphisms from $\BB R^2$ to $\BB R^2$ which take $\BB R$ to $\BB R$.
\end{defn}

See the left-hand sides of Figures~\ref{fig-meander}~and~\ref{sim-menders-permuton} for some illustrations of meanders. A meander of size $n$ can equivalently be thought of as a pair of Jordan curves in the sphere which intersect exactly $2n$ times and which do not intersect without crossing, viewed modulo orientation-preserving homeomorphisms from the sphere to itself. The correspondence with Definition~\ref{def-meander} is obtained by viewing $\BB R$ as a curve in the Riemann sphere. See the right-hand side of Figure~\ref{fig-meander}.

\begin{figure}[ht!]
\begin{center}
\includegraphics[width=1\textwidth]{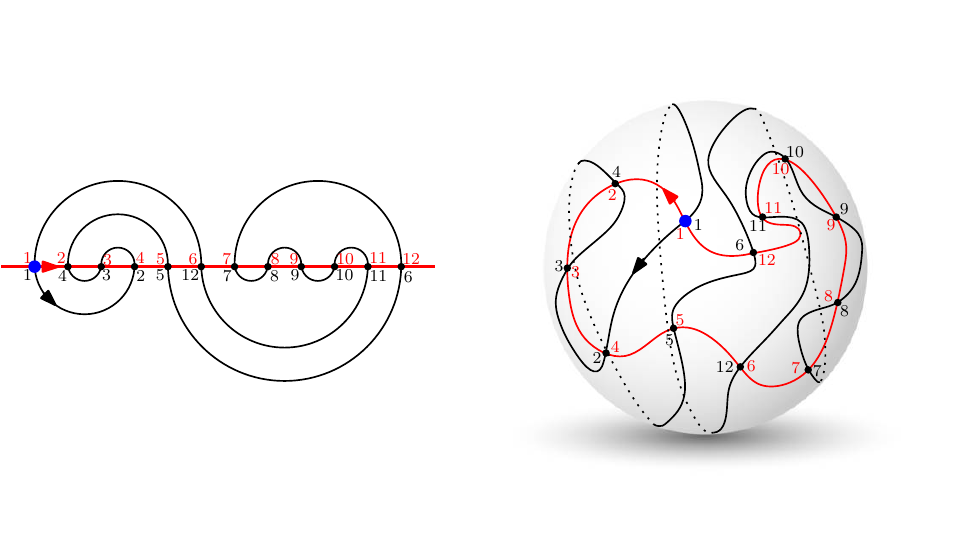}  
\caption{\label{fig-meander}\textbf{Left:} A meander of size $6$, with the intersection points of the loop and the line labeled by the order in which they are hit by the loop (in black) and the line (in red) starting at the blue vertex. The associated meandric permutation is $\sigma_{\ell}=1,4,3,2,5,12,7,8,9,10,11,6$ (in one-line notation); see Definition~\ref{def-meandric-perm} below. \textbf{Right:} The same meander of size $6$ seen as a pair of Jordan curves in the sphere.
}
\end{center}
\vspace{-3ex}
\end{figure}

The study of meanders dates back to at least the work of Poincar\'e~\cite{poincare-meander}, although the term ``meander'' was first coined by Arnold~\cite{arnold-meander}. Meanders are of interest in physics and computational biology as models of polymer folding~\cite{dgg-meander-folding}. There is a vast mathematical literature devoted to the enumeration of various types of meanders, which has connections to many different subjects, from combinatorics to theoretical physics and more recently to the geometry of moduli spaces \cite{Delecroix-enum-meand}. See~\cite{zvonkin-meander-survey} for a brief recent survey of this literature and~\cite{lacroix-meander-survey} for a more detailed account. Additionally, meanders are the connected components of \emph{meandric systems}, recently studied in \cite{ckst-noodle,fn-meander-system,Kargin-cycles-meander,Goulden-Asymptotics-meandric,Feray-components-mendric, bgp-meander-system}. 

The problem of enumerating \emph{all} meanders of size $n$ is open and seemingly quite difficult. But, it was conjectured by Di Francesco, Golinelli, and Guitter~\cite{dgg-meander-asymptotics}  that the number of meanders of size $n$ behaves asymptotically like $C \cdot A^n n^{-\alpha}$, where $\alpha = (29 + \sqrt{145})/12$. This conjecture was later numerically tested by Jensen and Gutmann~\cite{Jensen-num-meanders}, where they also proposed the numerical estimate of $A \approx 12.26$. 
We will discuss this conjecture in more detail in Section~\ref{sec-meander-conjecture}; see also Remark~\ref{remark-lqg-physics}. The best known rigorous bounds for the constant $A$ are $11.380\leq A \leq 12.901$, as proved in \cite{Albert-bounds-meanders}.

In this paper, we will be interested in the scaling limit of a uniform sample $\ell_n$ from the set of all meanders of size $n$. 
One way to formulate such a scaling limit is as follows.   
We associate with $\ell_n$ the planar map $M_n$ whose vertices are the $2n$ intersection points of the loop $\ell_n$ with the real line, whose edges are the segments of the loop or the line $\BB R$ between these intersection points (where we consider the two infinite rays of the line $\BB R$ as being an edge from the leftmost to the rightmost intersection point), and whose faces are the connected components of $\BB R^2 \setminus (\ell_n \cup \BB R)$. The planar map $M_n$ is equipped with two Hamiltonian paths $P_n^1, P_n^2 : [0,2n]\cap\BB Z\rta \{\text{vertices of $M_n$}\}$, one of which is induced by the left-right ordering of $\BB R$ and one of which is induced by $\ell_n$ (these are the black and red paths in Figure~\ref{fig-meander}). We view each of these paths as starting and ending at the leftmost intersection point of $\ell_n$ with $\BB R$. 

To talk about convergence, we can, e.g., view $(M_n,P_n^1,P_n^2)$ as a curve-decorated metric measure space (equipped with the graph metric and the counting measure on vertices) and ask whether it has a scaling limit with respect to the generalization of the Gromov-Hausdorff topology for curve-decorated metric measure spaces~\cite{gwynne-miller-uihpq}. As we will explain in Section~\ref{sec-meander-conjecture}, the arguments of~\cite{dgg-meander-asymptotics} lead to the following conjecture.

\begin{conj} \label{conj-meander}
Let $(M_n,P_n^1,P_n^2)$ be the random planar map decorated by two Hamiltonian paths associated with the uniform meander $\ell_n$ as above.
Then $(M_n,P_n^1,P_n^2)$ converges under an appropriate scaling limit to a Liouville quantum gravity (LQG) sphere with matter central charge $\ccM = -4$ (equivalently, from \eqref{eqn:lqg-ccM}, with coupling constant $\gamma = \sqrt{\frac13 \left( 17 - \sqrt{145} \right)}$), decorated by a pair of independent whole-plane SLE$_8$ curves from $\infty$ to $\infty$, each parametrized by LQG mass with respect to the quantum sphere. 
\end{conj}

An equivalent formulation of the conjecture is that $(M_n,P_n^1,P_n^2)$ should be in the same universality class as a uniform 3-tuple consisting of a planar map decorated by two spanning trees (and their associated discrete Peano curves). 
See Section~\ref{sec-meander-conjecture} for more details about the reasons behind the conjecture. We will not make any direct progress toward a proof of Conjecture~\ref{conj-meander} in this paper. Rather, we will study the conjectural limiting object directly.

\begin{remark} \label{remark-mismatched}
Most results concerning SLE-decorated Liouville quantum gravity are specific to the case when $\kappa \in \{\gamma^2,16/\gamma^2\}$~\cite{shef-zipper,wedges} (see~\cite{ghs-mating-survey} for a survey of such results). In the setting of Conjecture~\ref{conj-meander}, the values of $\gamma$ and $\kappa$ are \emph{mismatched} in the sense that $\kappa \notin \{\gamma^2,16/\gamma^2\}$. This case is much less well understood than the matched case.  
\end{remark}

Another possible perspective on the scaling limit of meanders is to look at the permutation associated with a meander $\ell$.

\begin{defn} \label{def-meandric-perm}
	For a meander $\ell$ of size $n$, the corresponding \textbf{meandric permutation} $\sigma_{\ell}$ is the permutation of $[1,2n]\cap\BB Z$ obtained as follows. Order the $2n$ intersection points of $\ell$ with $\BB R$ from left to right. Then, consider the second order in which these intersection points are hit by $\ell$ when we start $\ell$ from the leftmost intersection point and traverse it counterclockwise. Then $\sigma_\ell (i)=j$ if some intersection point is the $i$-th visited point by the first order and the $j$-th visited point by the second one.
\end{defn} 

\begin{figure}[ht!]
	\begin{center}
		\includegraphics[width=.25\textwidth]{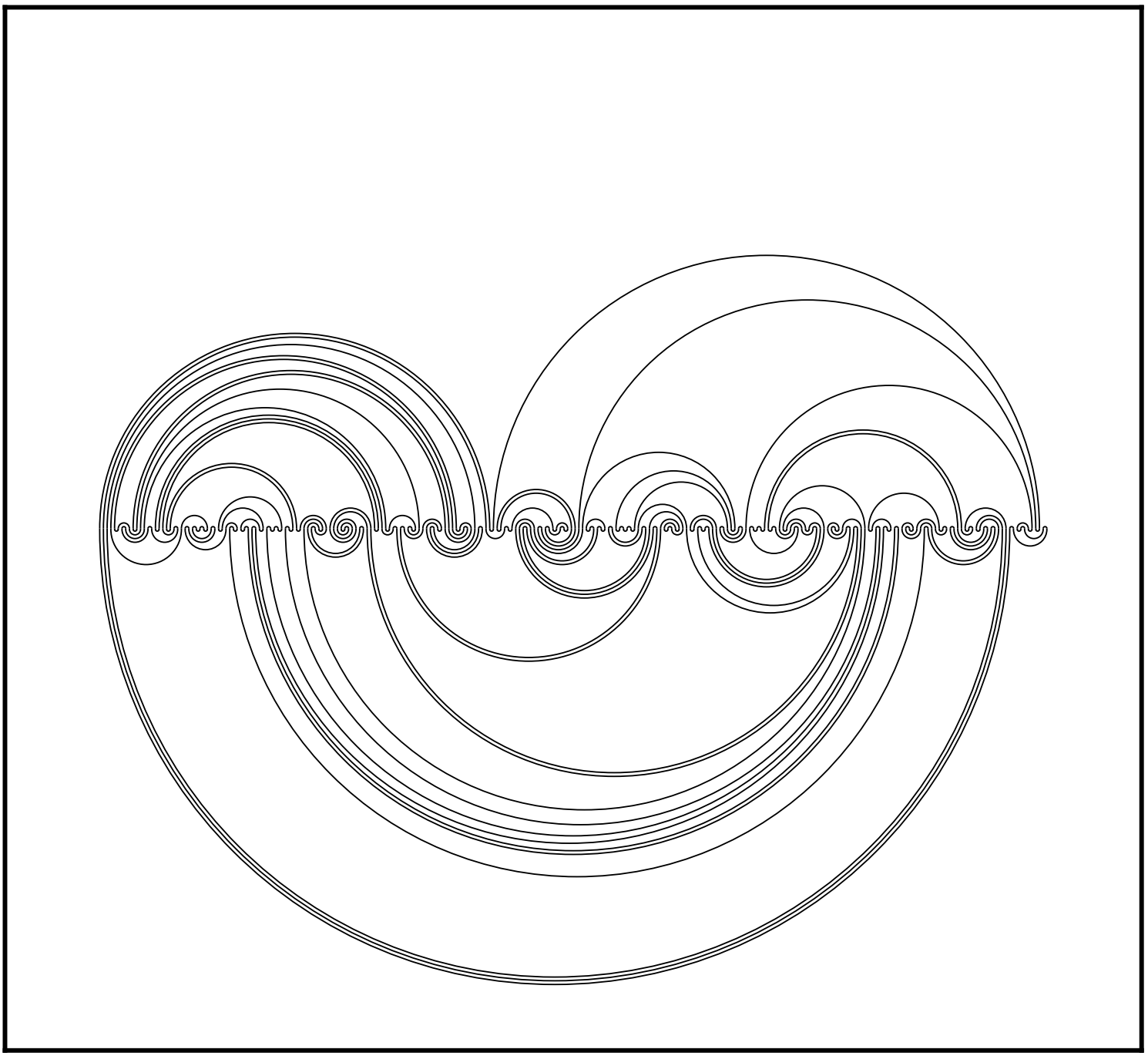}
		\includegraphics[width=.25\textwidth]{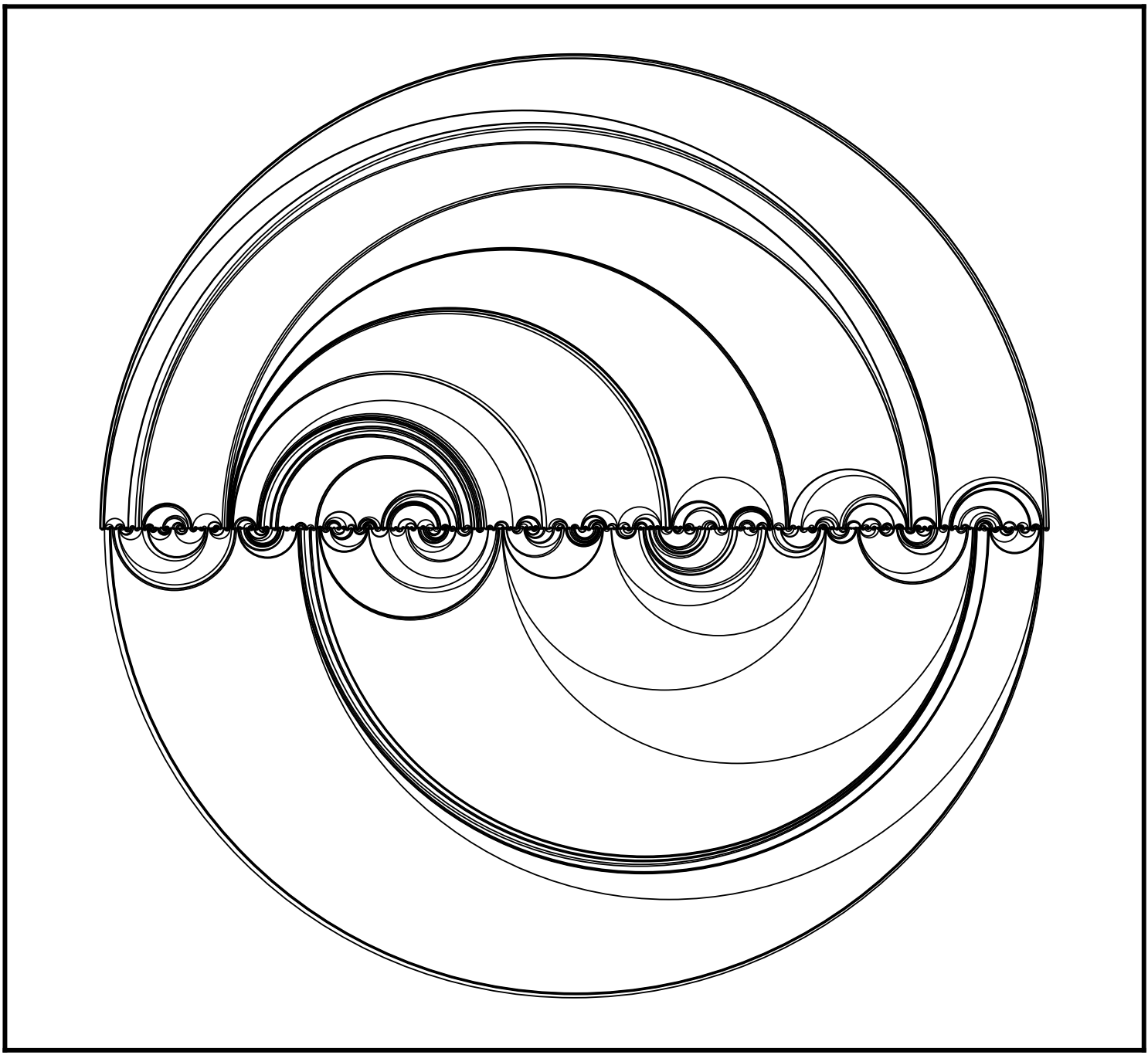}
		\includegraphics[width=.23\textwidth]{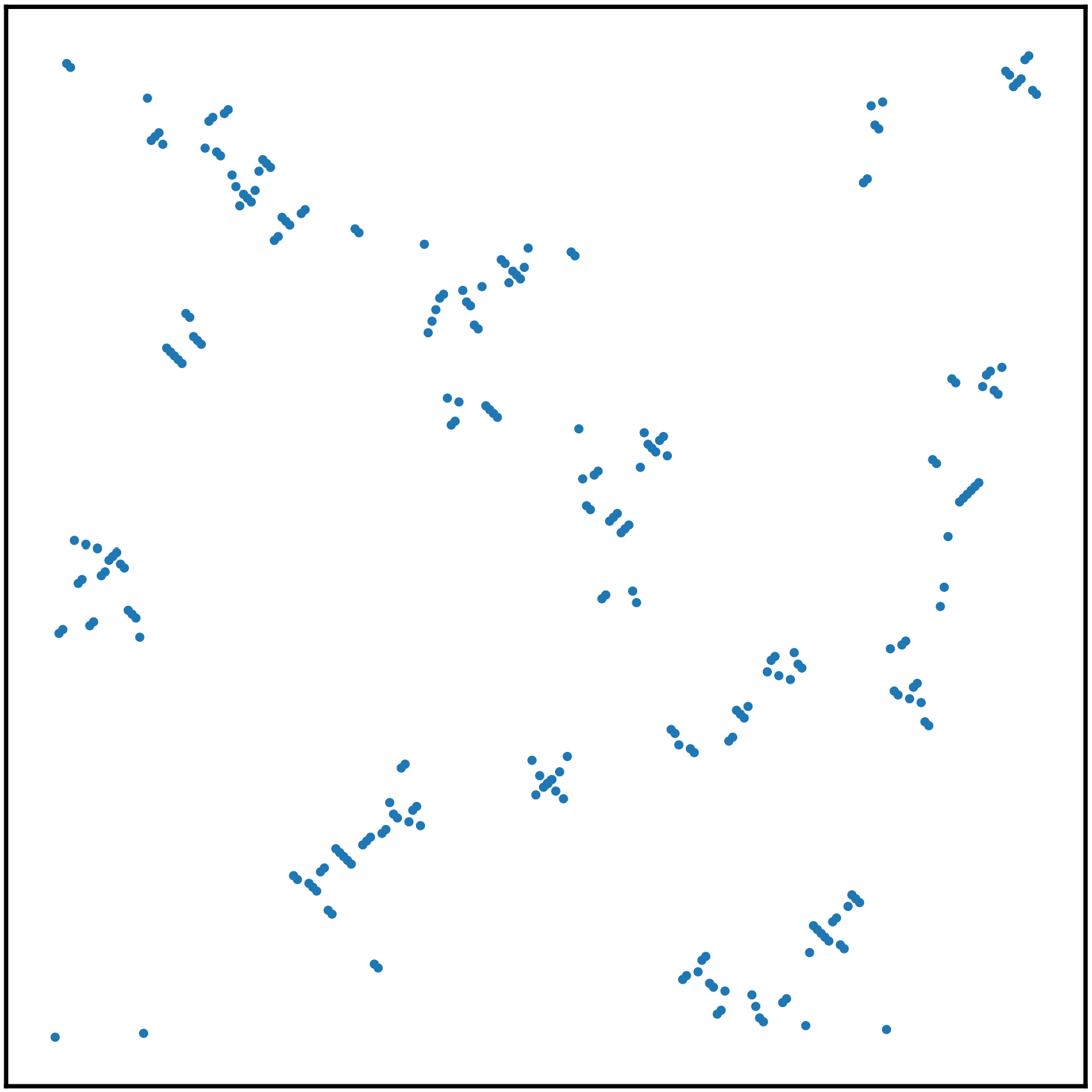}
		\includegraphics[width=.23\textwidth]{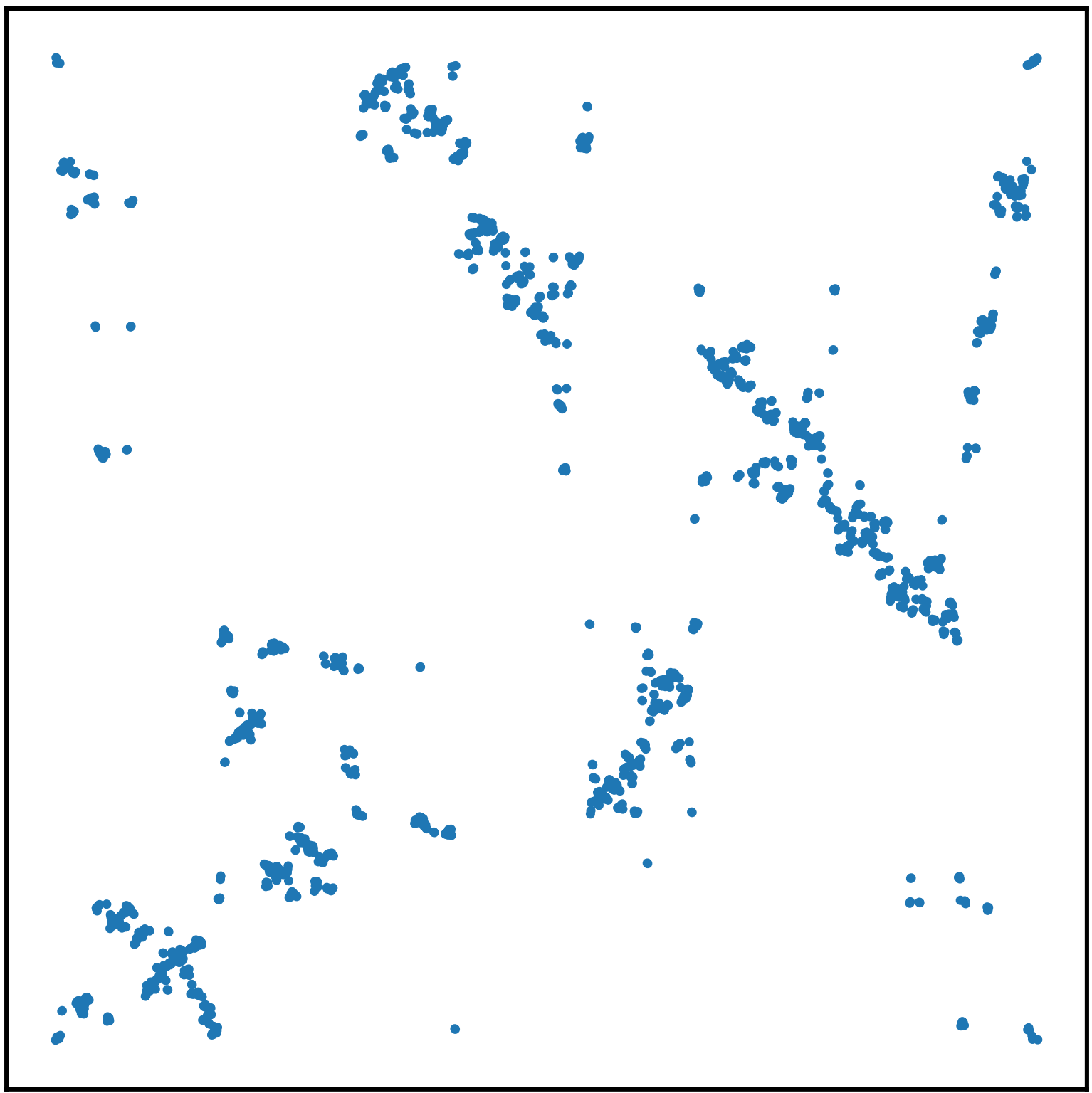} 
		\caption{\label{sim-menders-permuton}\textbf{Left:} Two large uniform meanders $\ell$ of size 256 and 2048. \textbf{Right:} The plots of the two corresponding meandric permutations $\sigma_{\ell}$; the blue dots correspond to the values $(i,\sigma_\ell(i))_{i = 1,\dots,|\sigma_\ell|}$. These simulations are obtained using the Markov chain Monte Carlo algorithm from \cite{heitsch2011meander}.}
	\end{center}
	\vspace{-3ex}
\end{figure}

See Figure~\ref{fig-meander} for an example and the right-hand side of Figure~\ref{sim-menders-permuton} for the plots of two large uniform meandric permutations. Note that any meandric permutation fixes 1. It is easily verified that the mapping  $\ell\mapsto\sigma_{\ell}$ is a bijection from the set of meanders of size $n$ to the set of meandric permutations of size $2n$. The study of meandric permutations dates back to at least the paper~\cite{rosentiehl-meander}, where they are called ``planar permutations''. We note that meandric permutations are defined differently in various other literature; see Section~\ref{sec-cyclic-meander} just below for further explanations.

Conjecture~\ref{conj-meander} leads naturally to the following definition and conjecture.

\begin{defn} \label{def-meander-permuton}
The \textbf{meandric permuton} is the permuton $\perm$ defined as in~\eqref{eqn-permuton-def} in the case when $\gamma = \sqrt{\frac13 \left( 17 - \sqrt{145} \right)}$, $\kappa_1 = \kappa_2 = 8$, and $\eta_1$ and $\eta_2$ are independent, viewed as curves modulo time parametrization. 
\end{defn}

\begin{conj} \label{conj-meander-permuton}
For $n\in\BB N$, let $\sigma_n$ be a uniform sample from the set of meandric permutations of size $2n$ (i.e., the permutation associated with a uniform meander of size $n$).
Then the associated permuton $\perm_{\sigma_n}$ (Definition~\ref{def-permutation}) converges in law to the meandric permuton. 
\end{conj}

It is possible that Conjecture~\ref{conj-meander-permuton} is easier to prove than Conjecture~\ref{conj-meander}. This is because permutons are in some sense simpler objects than SLE-decorated LQG and because the set of permutons is compact with respect to the weak topology (so tightness is automatic). Nevertheless, we do not know of a strategy for proving either Conjecture~\ref{conj-meander} or~\ref{conj-meander-permuton} (see, however, Problem~\ref{prob-char}). 

{It is shown in~\cite{bg-permuton-msrble} that} $h$ and the SLE$_8$ curves $\eta_1$ and $\eta_2$, viewed modulo conformal and anticonformal automorphisms of $\BB C$, are a.s.\ given by an explicit measurable functions of the meandric permuton $\perm$ {(the paper~\cite{bg-permuton-msrble} also proves the analogous statement for skew Brownian permutons)}. Consequently, Conjecture~\ref{conj-meander-permuton} can be thought of as a form of convergence of meanders toward SLE-decorated LQG.

\subsubsection{Cyclic meandric permutations}
\label{sec-cyclic-meander}

In various other literature, e.g.,~\cite{lacroix-meander-survey}, the meandric permutation is instead defined as follows. Label the intersection points of the meander $\ell$ from left to right as $1,\dots,2n$. For $j=1,\dots,2n$, let $\tau_\ell(j)$ be the index of the intersection point hit by the loop (traversed counterclockwise) immediately after it hits $j$. Then $\tau_\ell$ is a cycle of length $2n$. To distinguish $\tau_\ell$ from the meandric permutation $\sigma_\ell$ of Definition~\ref{def-meandric-perm}, we call $\tau_\ell$ the \textbf{cyclic meandric permutation}.

If we express in one-line notation the inverse  $\sigma_\ell^{-1}$ of our meandric permutation $\sigma_\ell$ , then the same expression for  $\sigma_\ell^{-1}$, interpreted as a permutation written in cyclic notation, is equal to $\tau_\ell$. The permutations $\tau_\ell$ and $\sigma_\ell^{-1}$ are related by 
\eqb \label{eqn-permutation-relation}
\sigma^{-1}_\ell(j) = \tau_\ell^{j-1}(1) \quad \text{and}\quad \tau_\ell(j) = \sigma_\ell^{-1}(\sigma_\ell(j)+1) , \quad\forall j =1,\dots,2n ,
\eqe 
where $2n+1$ should be interpreted as $1$ here. One reason why $\sigma_\ell$ (or $\sigma_\ell^{-1}$) is more natural than $\tau_\ell$ from our perspective is the following conjecture. 

\begin{conj} \label{conj-cyclic-meander}
Let $\ell_n$ be a uniform meander of size $n$ and let $\tau_{\ell_n}$ be its associated cyclic meandric permutation. The permuton associated with $\tau_{\ell_n}$ converges in law to the identity permuton (i.e., the uniform measure on the diagonal in $[0,1]^2$) as $n\rta\infty$. 
\end{conj}

See Section~\ref{sec-cyclic-conj} for a heuristic justification of Conjecture~\ref{conj-cyclic-meander} using Conjecture~\ref{conj-meander}.

\subsection{Skew Brownian permutons} 
\label{sec-skew-permuton}

The skew Brownian permutons are a two-parameter family of random permutons, indexed by $(\rho,q) \in (-1,1) \times (0,1)$, which were first introduced in~\cite{borga-skew-permuton}. The original construction of the skew Brownian permuton involves a pair of correlated Brownian excursions with correlation $\rho $. The parameter $q$ controls, roughly speaking, the tendency of the permuton to be increasing. This can be made precise by looking at the probability that a pair of points $(X_1,Y_1) , (X_2,Y_2)$ sampled from the permuton satisfy $X_1 < X_2$ and $Y_1 > Y_2$~\cite[Proposition 1.14]{bhsy-baxter-permuton}. The case $q=1/2$ is symmetric, in the sense that this probability is $1/2$. We will not need the Brownian motion construction of the skew Brownian permuton in this paper.

It is shown in~\cite[Theorem 1.17]{borga-skew-permuton} that for each choice of parameters $(\rho,q) \in (-1,1) \times (0,1)$ for the skew Brownian permuton, there exists $\gamma \in (0,2)$ and a coupling of two whole-plane space-filling SLE$_{\kappa = 16/\gamma^2}$ curves from $\infty$ to $\infty$ such that the permuton~\eqref{eqn-permuton-def} coincides with the skew Brownian permuton. See Theorem~\ref{thm-skew-permuton} below for a precise statement. This representation has previously been used to prove various properties of the skew Brownian permuton in~\cite{bhsy-baxter-permuton}. 
 
One of the main reasons to study skew Brownian permutons is that they describe the scaling limits of various types of pattern-avoiding random permutations, for example the following. 

\begin{defn} \label{def-baxter}
Let $\sigma$ be a permutation of size $n$.
\begin{itemize}
\item $\sigma$ is \textbf{semi-Baxter} if there does not exist $1 \leq i < j < k \leq n$ such that $\sigma(j+1) < \sigma(i) < \sigma(k) < \sigma(j)$. 
\item $\sigma$ is \textbf{Baxter} if $\sigma$ is semi-Baxter and additionally there does not exist $1 \leq i < j < k \leq n$ such that $\sigma(j ) < \sigma(k) < \sigma(i) < \sigma(j+1)$.
\item $\sigma$ is \textbf{strong-Baxter} if $\sigma$ is Baxter and additionally there does not exist $1\leq i < j < k \leq n$ such that $\sigma(j+1) < \sigma(k) < \sigma(i) < \sigma(j)$. 
\end{itemize}
\end{defn}

\noindent We have the following convergence results toward skew Brownian permutons.

\begin{itemize}
\item The permutons associated with uniform semi-Baxter permutations converge in law to the skew Brownian permuton with $\rho = -(1+\sqrt 5)/4$ and $q=1/2$~\cite{borga-strong-baxter}. 
\item The permutons associated with uniform Baxter permutations converge in law to the skew Brownian permuton with $\rho = -1/2$ and $q=1/2$~\cite{bm-baxter-permutation} (this special case of the skew Brownian permuton is sometimes called the \textbf{Baxter permuton}). 
\item The permutons associated with uniform strong-Baxter permutations converge in law to the skew Brownian permuton with parameters $\rho \approx  - 0.2151$ and $q\approx 0.3008$~\cite{borga-strong-baxter}.   
\end{itemize}

\begin{remark} \label{remark-separable}
The boundary case $\rho=1$, $q\in (0,1)$ for the skew Brownian permuton corresponds to the Brownian separable permuton~\cite{bassino-separable-permuton} and biased versions thereof \cite{Bassino-universal-perm,Borga-decorated-perm}, as proved in~\cite[Theorem 1.12]{borga-skew-permuton}. We will not consider the case $\rho=1$ in the present paper since the representation in terms of SLE and LQG from~\cite[Theorem 1.17]{borga-skew-permuton} does not work the same way in this case. However, some of the results in this paper have already been proven for the Brownian separable permuton by other methods. It is conjectured in \cite[Section 1.6, Item 7]{borga-skew-permuton} that the Brownian separable permuton can be described in terms of SLE and LQG using the critical mating of trees results from~\cite{ahps-critical-mating}.
\end{remark}

\subsection{Main results} 
\label{sec-results}

The main results of this paper and their proofs (given in Sections~\ref{sec-lis-proof} through~\ref{sec-re-root-proof}) can be read independently from each other.

\subsubsection{Longest increasing subsequence is sublinear} 
\label{sec-lis}

Assume that we are in the setting of Section~\ref{sec-permuton-def}. We allow $\gamma \in (0,2)$ and $\kappa_1,\kappa_2 > 4$ to be arbitrary, but we require that the joint distribution of the SLE curves $(\eta_1,\eta_2)$, viewed as curves modulo time parametrization, satisfies one of the following two conditions.
\begin{itemize}
\item $\eta_1$ and $\eta_2$ are independent viewed modulo time parametrization (as is the case in the definition of the meandric permuton); or
\item $\kappa_1=\kappa_2 $ and $(\eta_1,\eta_2)$ are coupled together as in the SLE / LQG description of the skew Brownian permuton\footnote{We do not require that $\kappa_i=16/\gamma^2$ but only that $\kappa_1=\kappa_2$.} for some choice of $(\rho,q) \in (-1,1)\times (0,1)$ (Theorem~\ref{thm-skew-permuton}). 
\end{itemize}
Let $\perm$ be the associated permuton as in~\eqref{eqn-permuton-def}. {Note that the possibilities for $\perm$ include both the meandric permuton and the skew Brownian permutons.} In this subsection we will consider the longest increasing subsequence for random permutations which converge to $\perm$.

\begin{defn} \label{def-lis}
For a permutation $\sigma$, we define the \textbf{length of the longest increasing (resp.\ decreasing) subsequence} $\op{LIS}(\sigma)$ (resp.\ $\op{LDS}(\sigma)$) to be the maximal cardinality of a set $L\subset [1,|\sigma|] \cap \BB Z$ such that the restriction of $\sigma$ to $L$ is monotone increasing (resp.\ decreasing). 
\end{defn} 

There is a huge literature devoted to the asymptotic behavior of $\op{LIS}(\sigma_n)$ and $\op{LDS}(\sigma_n)$ for various types of large random permutations $\sigma_n$. 
For uniform random permutations $\sigma_n$, the question of investigating $\op{LIS}(\sigma_n)$ was raised in the 1960s by Ulam~\cite{Ulam-lis}. In this case, one has $\op{LIS}(\sigma_n) \sim 2\sqrt n$~\cite{Hammersley-lis,Vervsik-lis,Logan-lis}. 
The strongest known result is due to Dauvergne and Vir\'{a}g~\cite{Dauvergne-lis}, who showed that the scaling limit of the longest increasing subsequence in a uniform permutation is the \emph{directed geodesic} of the \emph{directed landscape}. 
The study of $\op{LIS}(\sigma_n)$ is connected with many other problems in combinatorics and probability theory, such as last passage percolation and random matrix theory; we refer the reader to the book of  Romik~\cite{Romik-lis} for an overview. 

In recent years, many extensions beyond uniform permutations have been considered:
\begin{itemize}
	\item when $\sigma_n$ is a uniform random pattern-avoiding permutation~\cite{Deutsch-lis,Madras-lis, Mansour-lis,bassino-lis};
	\item when $\sigma_n$ is conjugacy-invariant with few cycles, in particular \emph{Ewens-distributed}~\cite{Kammoun-lis};
	\item when $\sigma_n$ follows the \emph{Mallows distribution}, or is a product of such random permutations (\cite{Ke-lis-mallows} and references therein).
\end{itemize}
For the first two items the results are far from being complete, and our next result is a new contribution towards the understanding of the behavior of $\op{LIS}(\sigma_n)$ for random permutations $\sigma_n$ which are not chosen uniformly from all possibilities.  

\begin{thm} \label{thm-lis}
Assume that $(\eta_1,\eta_2)$ satisfy the hypotheses at the beginning of this subsection and let $\perm$ be as in~\eqref{eqn-permuton-def}. 
Let $\{\sigma_n\}_{n\in\BB N}$ be a sequence of random permutations of size $|\sigma_n| \rta\infty$ whose associated permutons $\perm_{\sigma_n}$ converge in law to $\perm$ with respect to the weak topology.  
Then the longest increasing subsequence of $\sigma_n$ is sublinear in $|\sigma_n|$, i.e., $\op{LIS}(\sigma_n) / |\sigma_n| \rta 0$ in probability.  The same is true for the longest decreasing subsequence. 
\end{thm}

Theorem~\ref{thm-lis} will be a consequence of a general condition on a permuton which implies that any permutations converging to it have sublinear longest increasing subsequences (Proposition~\ref{prop-monotone-set}). We will verify this condition in our setting in Section~\ref{sec-lis-proof}, via a surprisingly simple (4-page) argument using SLE and LQG.

The analog of Theorem~\ref{thm-lis} for the biased Brownian separable permuton (which corrresponds to the skew Brownian permuton in the limiting case $\rho = 1$) was proven in~\cite[Theorem 1.10]{bassino-lis}, using very different arguments from the ones in this paper. 
As an immediate corollary of Theorem~\ref{thm-lis} and the results of~\cite{borga-skew-permuton} (see Theorem~\ref{thm-skew-permuton}), we obtain the following. 

\begin{cor} \label{cor-skew-lis}
Let $\perm$ be the skew Brownian permuton with parameters $\rho \in (-1,1)$ and $q\in (0,1)$. 
Let $\{\sigma_n\}_{n\in\BB N}$ be a sequence of random permutations of size $|\sigma_n| \rta\infty$ whose associated permutons $\perm_{\sigma_n}$ converge in law to $\perm$ with respect to the weak topology.  
Then $\op{LIS}(\sigma_n) / |\sigma_n| \rta 0$ and $\op{LDS}(\sigma_n)/|\sigma_n| \rta 0$ in probability.  
\end{cor} 

We now record the implication of Corollary~\ref{cor-skew-lis} for some particular random permutations which are known to converge to the skew Brownian permuton (see Section~\ref{sec-skew-permuton}). 

\begin{cor} \label{cor-baxter-lis}
For $n\in\BB N$, let $\sigma_n$ be a uniform sample from the set of Baxter permutations of size $n$ (Definition~\ref{def-baxter}). 
Then $\op{LIS}(\sigma_n)/ n \rta 0$ and $\op{LDS}(\sigma_n)/n \rta 0$ in probability.
The same is true for strong-Baxter permutations and semi-Baxter permutations.
\end{cor} 

Note that not all models of uniform pattern-avoiding permutations have sublinear longest increasing subsequences; see, e.g.,~\cite{Deutsch-lis,Madras-lis, Mansour-lis} for examples where the longest increasing subsequence is known to be linear.

\begin{remark} \label{remark-bipolar}
The study of the length of the longest increasing subsequence in a Baxter permutation $\op{LIS}(\sigma_n)$ is also motivated by its connection with the length of the longest directed path $\op{LDP}(m_n)$ from the source to the sink of a uniform bipolar orientation $m_n$ with $n$ edges. Recall that a bipolar orientation is a planar map equipped with an acyclic orientation of the edges with exactly one source (i.e.\ a vertex with only outgoing edges) and one sink (i.e.\ a vertex with only incoming edges), both on the outer face. Uniform bipolar orientations have been studied, e.g., in \cite{kmsw-bipolar,ghs-bipolar,bm-baxter-permutation}. 
		
Building on a bijection between Baxter permutations and bipolar orientations, introduced by Bonichon, Bousquet-M\'{e}lou and Fusy~\cite{bbf-bipolar-bijection}, it is possible to show that $\op{LIS}(\sigma_n)$ and $\op{LDP}(m_n)$ have the same law. Therefore,  Corollary \ref{cor-baxter-lis} implies that $\op{LDP}(m_n)$ is sublinear.
		
The study of $\op{LDP}(m_n)$ can be interpreted as a model of last passage percolation in random geometry, which, to the best of our knowledge, has been not investigated previously.
\end{remark}

We do not know that uniform meandric permutations converge to the meandric permuton, but due to Theorem~\ref{thm-lis}, the following statement is implied by Conjecture~\ref{conj-meander-permuton}.

\begin{conj} \label{conj-meander-lis} 
Let $\sigma_n$ be a uniform meandric permutation of size $2n$. Then $\op{LIS}(\sigma_n)/n \rta 0$ and $\op{LDS}(\sigma_n)/n \rta 0$ in probability.
\end{conj}

It is also of substantial interest to prove non-trivial \emph{lower} bound for $\op{LIS}(\sigma_n)$ and $\op{LDS}(\sigma_n)$, for appropriate families of random permutations $\{\sigma_n\}_{n\in\BB N}$ which converge in law to the permutons considered in this paper. A classical theorem of Erd\"os-Szekeres~\cite{erdos-szekeres} shows that for any permutation $\sigma$, one has $\op{LIS}(\sigma) \times \op{LDS}(\sigma) \geq |\sigma|$. Hence, either $\op{LIS}(\sigma)$ or $\op{LDS}(\sigma)$ (or both) must be at least $\lfloor \sqrt{|\sigma|} \rfloor$. 

We expect that for the random permutations considered in Corollary~\ref{cor-baxter-lis}, both $\op{LIS}(\sigma_n)$ and $\op{LDS}(\sigma_n)$ should be of order $n^\alpha$ for some $\alpha \in (1/2,1)$. By Remark~\ref{remark-bipolar}, similar statements should also hold for the length of the longest directed path from the the source to the sink of a uniform bipolar orientation. 
In the case of uniform meandric permutations, we have a more precise guess for the growth rate of $\op{LIS}(\sigma_n)$ and $\op{LDS}(\sigma_n)$ coming from numerical simulations.\footnote{The numerical simulations used for Conjecture \ref{conj-meander-lis-precise} were obtained through the implementation of the Markov chain Monte Carlo approach from \cite{heitsch2011meander}.} 
	
\begin{conj} \label{conj-meander-lis-precise} 
	Let $\sigma_n$ be a uniform meandric permutation of size $2n$. Then there exists $\alpha\approx 0.69$ such that $\op{LIS}(\sigma_n) = n^{\alpha + o(1)}$ and $\op{LDS}(\sigma_n) = n^{\alpha+o(1)}$ with probability tending to 1 as $n\to \infty$. 
	%and a random $C>0$ such that $\op{LIS}(\sigma_n)/n^\alpha \rta C$ and $\op{LDS}(\sigma_n)/n^\alpha \rta C$ in distribution.
\end{conj}	

We do not have a guess for the precise value of $\alpha$. 
In the settings of both Corollary~\ref{cor-baxter-lis} and~\ref{conj-meander-lis-precise}, we are not currently able to prove any lower bounds for $\op{LIS}(\sigma_n)$ and $\op{LDS}(\sigma_n)$ which are significantly better than the trivial bound $n^{1/2}$ which comes from Erd\"os-Szekeres. See Problems~\ref{prob-lis-lower} and~\ref{prob-lis-sle} for some relevant open problems.

\subsubsection{Dimension of the support is one} 
\label{sec-permuton-dim}

\begin{defn} \label{def-supp}
The \textbf{closed support} of a permuton $\perm$ is the set
\eqb \label{eqn-permuton-supp}
\op{supp} \perm :=  \left( \text{intersection of all closed sets $K\subset [0,1]^2$ with $\perm(K) = 1$} \right) .
\eqe
\end{defn}

For a general permuton, the Hausdorff dimension of its closed support can be any number in $[1,2]$. {In particular, there exist permutons constructed from pairs of space-filling curves for which the Hausdorff dimension of the support is arbitrarily close to two (see Remark~\ref{remark-permuton-rectangle}).} Our next theorem states that the closed supports of the permutons considered in this paper all have Hausdorff dimension one.

\begin{thm} \label{thm-permuton-dim}
Assume that we are in the setting of Section~\ref{sec-permuton-def} with an \emph{arbitrary} choice of $\gamma \in (0,2)$, $\kappa_1,\kappa_2 > 4$, and coupling $(\eta_1,\eta_2)$.  Let $\perm$ be as in~\eqref{eqn-permuton-def}. Let 
\eqb \label{eqn-full-intersect-set}
\mcl T = \mcl T(\eta_1,\eta_2) := \left\{(t,s) \in [0,1]^2 : \eta_1(t) = \eta_2(s)\right\} .
\eqe
Almost surely, both the closed support $\op{supp} \perm$ and the set $\mcl T$ have Hausdorff dimension one. 
\end{thm} 

{It is easy to see that Hausdorff dimension of the support of every permuton is at least one. So, Theorem~\ref{thm-permuton-dim} says that $\op{supp}\perm$ is in some sense as small as possible. Equivalently, $\perm$ is as far as possible from being absolutely continuous with respect to Lebesgue measure.}

It is easy to check that $\op{supp} \perm \subset \mcl T$ (Lemma~\ref{lem-permuton-inclusion}), but the reverse inclusion need not hold. For example, if $\eta_1=\eta_2$ a.s.\ then $\op{supp} \perm$ is the diagonal in $[0,1]^2$ but $\mcl T$ contains off-diagonal points which arise from points in $\BB C$ which are hit more than once by $\eta_1$ or $\eta_2$. It is in general a subtle question to determine which points of $\mcl T$ belong to $\op{supp} \perm$. {This question was recently addressed in \cite{bg-permuton-msrble}.} 

The following corollary is a special case of Theorem~\ref{thm-permuton-dim}. It in particular answers a question from~\cite[Section 1.6]{borga-skew-permuton}. 
 
\begin{cor} \label{cor-skew-brownian}
For each choice of parameters $(\rho,q) \in (-1,1) \times (0,1)$, the closed support of the skew Brownian permuton with parameters $\rho,q$ has Hausdorff dimension 1. The same is true for the meandric permuton. 
\end{cor}

It is shown in~\cite[Theorem 1.5]{maazoun-separable-permuton} that the dimension of the closed support of the biased Brownian separable permuton (i.e.\ the skew Brownian permuton with $\rho=1$ and $q\in (0,1)$; recall Remark~\ref{remark-separable}) is equal to 1. In fact, the one-dimensional Hausdorff measure of its support is bounded above by $\sqrt 2$. The proof in~\cite{maazoun-separable-permuton} is based on a description of the biased Brownian separable permuton in terms of Brownian motion, so it is possible (but not obvious) that a related proof could work for the skew Brownian permutons. However, arguments like those in~\cite{maazoun-separable-permuton} would not work for the meandric permuton or for other permutons generated from a pair of independent space-filling SLE curves.

Our proof of Theorem~\ref{thm-permuton-dim}, given in Section~\ref{sec-dim-proof}, will be based on a short argument using estimates for the Liouville quantum gravity metric which come from~\cite{afs-metric-ball,gs-lqg-minkowski}.

\subsubsection{Re-rooting invariance for the meandric permuton} 
\label{sec-re-root}

Let $\ell_n$ be a uniform meander of size $n$. 
%Suppose we traverse $\ell_n$ in the counterclockwise direction started from the leftmost point where it intersects $\BB R$, as in the definition of the meandric permutation (Definition~\ref{def-meandric-perm}).  
For $k\in [1,2n]\cap\BB Z$, let $x_k$ be the $k$-th intersection point of $\ell_n$ with $\BB R$ in left-right order.
For $k\in [1,2n] \cap\BB Z$, we obtain a new meander $\ell_n^{(k)}$ as follows. See Figure~\ref{fig-meander-re-root} for an illustration. 
Let $y \in (-\infty,x_{k})$ be chosen so that there is no intersection point of $\ell_n$ with $\BB R$ in the interval $[y,x_{k})$ and set $f_y(z):=1/(z-y)$, for all $z\in\mathbb C$. 
We then take $\ell_n^{(k)}$ to be the meander which is the image of $\ell_n$ under the map 
\begin{itemize}
	\item $z\mapsto f_y(z)$, if $x_k$ is traversed by $\ell_n$ from top to bottom;
	\item  $z\mapsto \overline{f_y(z)}$, if $x_k$ is traversed by $\ell_n$ from bottom to top.
\end{itemize} 
We refer to it as the meander $\ell_n$ re-rooted at $x_k$.
The leftmost point of $\ell_n^{(k)} \cap \BB R$ is equal to $1/(x_k-y)$. Hence, we can identify $\ell_n^{(k)}$ with the loop $\ell_n$, but traversed started at $x_k$ instead of at $x_1$.

\begin{figure}[ht!]
\begin{center}
\includegraphics[scale=1]{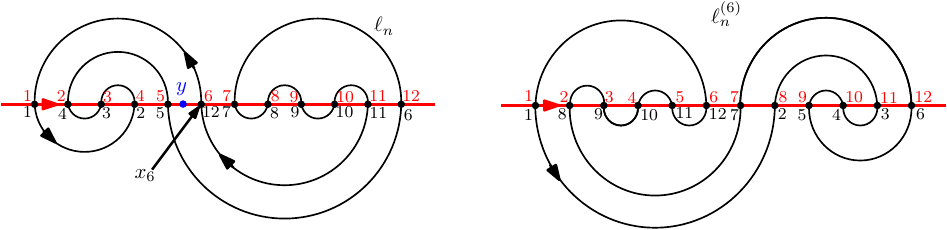}  
\caption{\label{fig-meander-re-root} \textbf{Left:} A meander $\ell_n$ and the corresponding meandric permutation $\sigma_{n}=1,4,3,2,5,12,7,8,9,10,11,6$. \textbf{Right:} The meander $\ell_n^{(6)}$ obtained by re-rooting $\ell_n$ at the point $x_{6}$ and the corresponding meandric permutation $\sigma^{(6)}_{n}=1,8,9,10,11,12,7,2,5,4,3,6$.}
\end{center}
\vspace{-3ex}
\end{figure}

The mapping $\ell_n \rta \ell_n^{(k)}$ is a bijection from the set of meanders of size $n$ to itself. Consequently, $\ell_n^{(k)}$ is also a uniform meander of size $n$. 
If $\sigma_n:=\sigma_{\ell_n}$ denotes the meandric permutation associated with $\ell_n$, then the meandric permutation associated with $\ell_n^{(k)}$ is 
\eqb \label{eqn-permutation-conj}
\sigma_n^{(k)} := \sigma_{\ell_n^{(k)}}= \tau_{\sigma_n(k)}^{-1} \circ \sigma_n  \circ \tau_k,  \quad \text{where} \quad \tau_j :=     j , j+1 , \dots,2n,1,2,\dots, j-1   .
\eqe 
Hence, the law of the uniform meandric permutation is invariant under the operation of pre- and post- composing with cyclic permutations in such a way that the resulting permutation still fixes 1. 
Our next theorem gives an analog of this property for the meandric permuton. See Figure~\ref{fig-re-root} for an illustration.

\begin{figure}[ht!]
\begin{center}
\includegraphics[scale=.85]{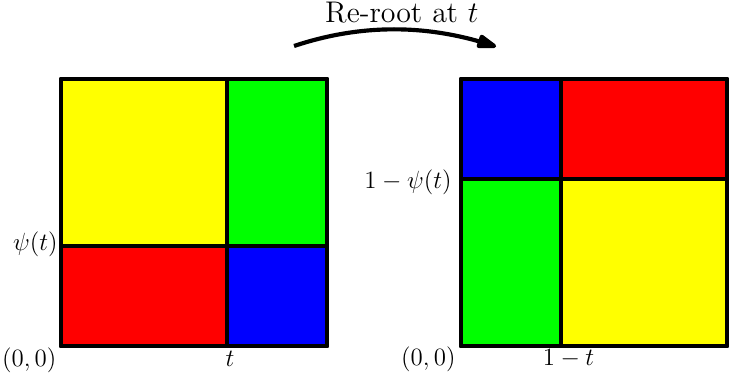}  
\caption{\label{fig-re-root} Visual representation of the re-rooting operation of Theorem~\ref{thm-permuton-re-root}. The re-rooted permuton $\perm_t$ is obtained by re-arranging the four rectangles shown in the figure, without changing the restrictions of the permuton to each of the four rectangles. One can check that conjugating by cyclic permutations as in~\eqref{eqn-permutation-conj} has a similar effect on the permuton associated with a given permutation (Definition~\ref{def-permutation}). 
}
\end{center}
\vspace{-3ex}
\end{figure}

\begin{thm} \label{thm-permuton-re-root}
Suppose that we are in the setting of Section~\ref{sec-permuton-def} with $\gamma\in (0,2)$ arbitrary, $\kappa_1 = \kappa_2 =8$, and $(\eta_1,\eta_2)$ are independent modulo time parametrization. 
Let $\perm$ be the permuton associated with $(\eta_1,\eta_2)$ as in~\eqref{eqn-permuton-def}.
For $t\in [0,1]$, let $\psi(t)$ be the (a.s.\ unique\footnote{
The a.s.\ uniqueness of $\psi(t)$ for a fixed $t\in [0,1]$ is proven in Lemma~\ref{lem-permuton-unique}, which also shows that $\psi(t)$ is the a.s.\ unique time for which $\eta_1(t) = \eta_2(\psi(t))$.
})
$s\in [0,1]$ for which $(t,s)$ is in the closed support of $\perm$. Let $\perm_t$ be the pushforward of the measure $\perm$ under the mapping $[0,1]^2 \rta [0,1]^2$ defined by
\eqbn
(u, v) \mapsto \left( u-t   - \lfloor u - t   \rfloor , v - \psi(t)   -  \lfloor v - \psi(t) \rfloor \right) . 
\eqen 
For each fixed $t\in [0,1]$, we have $\perm_t \eqD \perm$. 
\end{thm} 

As we will explain in Section~\ref{sec-re-root-proof}, Theorem~\ref{thm-permuton-re-root} is a consequence of a special re-rooting invariance property for SLE$_8$ (Proposition~\ref{prop-sle-re-root}). 
{It is shown in~\cite[Corollary 1.6]{bg-permuton-msrble} that the statement of Theorem~\ref{thm-permuton-re-root} is not true when either $\kappa_1$ or $\kappa_2$ is not equal to 8, using the fact that the permuton $\perm$ a.s.\ determines $\eta_1$, $\eta_2$, and $h$ modulo conformal and anticonformal maps.}
%We do not expect\footnote{We plan to give a formal proof of this claim in future work combining the results in Proposition~\ref{prop-sle-lqg-re-root} and the aforementioned fact that the field $h$ and the independent SLE curves $\eta_1$ and $\eta_2$, viewed modulo conformal and anticonformal automorphisms of $\BB C$, are a.s.\ given by measurable functions of the associated permuton $\perm$ as in \eqref{eqn-permuton-def}.} that the statement of Theorem~\ref{thm-permuton-re-root} is true when either $\kappa_1$ or $\kappa_2$ is not equal to 8. 
This suggests that the scaling limit of a uniform meandric permutation \emph{cannot} be one of the {permutons} of Section~\ref{sec-permuton-def} for $\kappa_1,\kappa_2\not=8$, which provides some further evidence in favor of Conjecture~\ref{conj-meander-permuton} (beyond the evidence discussed in Section~\ref{sec-meander-conjecture}). We do not yet have a special property which singles out the LQG parameter $\gamma =  \sqrt{\frac13 \left( 17 - \sqrt{145} \right)}$.

\subsubsection{Pattern density for the meandric permuton} 
\label{sec-density}
Let $\perm$ be a (possibly random) permuton and $n$ be a positive integer. Let $(\vec{X},\vec{Y})=(X_i, Y_i)_{i\in [1,n]\cap\BB Z}$ be $n$ i.i.d.\ points in $[0,1]^2$
sampled from  $ \perm$. 
Then the $n$ points almost surely have distinct $x$ and $y$ coordinates. For $j \in [1,n]\cap\BB Z$, let $\tau(j) $ be such that if $x_j$ is the $j$-th smallest  element in $(X_\ell)_{\ell\in [1,n]\cap\BB Z}$, then $y_j$ is the $\tau(j)$-th  smallest element in $(Y_\ell)_{\ell\in [1,n]\cap\BB Z}$. Then  $\tau$ is  a permutation of size $n$, which  we denote by $\Perm_n(\perm,\vec{X},\vec{Y})$.  For a fixed permutation $\sigma$ of size $n$, the probabilities $\BB P[\Perm_n(\perm,\vec{X},\vec{Y}) = \sigma  ]$ and $\BB P[\Perm_n(\perm,\vec{X},\vec{Y}) = \sigma \,|\, \perm]$ are called the \textbf{annealed} and \textbf{quenched pattern densities} of $\sigma$ in $\perm$.

If $\perm$ is a skew Brownian permuton, it was shown in~\cite[Proposition 1.14]{bhsy-baxter-permuton}  that the annealed pattern densities $\mathbbm P[\Perm_2(\perm,\vec{X},\vec{Y})=1,2]$ and  $\mathbb P[\Perm_2(\perm,\vec{X},\vec{Y})=2,1]$ are  $\frac{1}{2}\pm\frac{\theta}{\pi}$, where $\theta\in[-\frac \pi 2, \frac \pi 2]$ is the angle between the two space-filling SLE curves $(\eta_1,\eta_2)$ in their imaginary geometry coupling (see Theorem~\ref{thm-skew-permuton}). The exact relation between   $\theta$  and the parameters $(\rho,q)$ for the skew Brownian permuton is not known. For $n\ge 3$, although the distribution of $\Perm_n(\perm,\vec{X},\vec{Y})$ can in principle be expressed by SLE and LQG quantities, the relation would be quite involved; as explained in \cite[Remark 3.14]{bhsy-baxter-permuton}.

For a class of permutons including the meandric permuton, we have the following conceptually straightforward expression.

\begin{thm}\label{thm:density}
Suppose that we are in the setting of Section~\ref{sec-permuton-def} with $\gamma\in (0,2)$, $\kappa_1 , \kappa_2 >4$, and $(\eta_1,\eta_2)$ are independent modulo time parametrization. Recall that $\mu_h$ denotes the $\gamma$-LQG area measure corresponding to a singly marked unit area $\gamma$-Liouville quantum sphere. Let $\perm$ be the permuton associated with $(\eta_1,\eta_2)$ as in~\eqref{eqn-permuton-def}. For $i\in \{1,2\}$ and $z_1,\cdots, z_n\in \mathbb C$, let $P_{\kappa_i}(z_1,\cdots, z_n)$ be the probability that $\eta_i$ hits $z_j$ before $z_{j+1}$ for each $j=1,\dots,n-1$.  Let $\zeta(z_1,\cdots,z_n)$ be the joint density of $n$ points independently sampled from the measure $\mu_h$, i.e., $\zeta(z_1,\cdots,z_n)$ is the density of the probability measure $\BB E[\mu_h \otimes \cdots \otimes \mu_h]$ on $\BB C^n$. For a permutation $\sigma$ of size $n$,  the annealed pattern density is given by
\begin{equation}\label{eq:density}
\BB P\left[   \Perm_n(\perm,\vec{X},\vec{Y}) = \sigma \right]=n!\int_{\BB C^n} P_{\kappa_1}(z_1,\cdots, z_n)P_{\kappa_2}(z_{\sigma(1)},\cdots,z_{\sigma(n)}) \zeta (z_1,\cdots, z_n)\prod_{j=1}^n dz_j.
\end{equation}
\end{thm}
\begin{proof}
This follows from the independence of $\eta_1$ and $\eta_2$, and the definitions of $\perm$ and  $ \Perm_n(\perm,\vec{X},\vec{Y})$. The factor $n!$ enumerates the possible orders in which $\eta_1$ can hit $n$ distinct points on the plane. 
\end{proof}

The law of the field $h$ in Theorem~\ref{thm:density} describing a singly marked unit area $\gamma$-Liouville quantum sphere is not unique (rather, it is only defined up to a conformal change of coordinates, see Section~\ref{sec-lqg}). As explained in~\cite[Section 1.4]{BW-LCFT}, for an appropriate choice of $h$ the density function $\zeta(z_1,\cdots,z_n)$ can be expressed as correlation functions in Liouville conformal field theory on the sphere with all vertex insertions $\alpha_1,\dots,\alpha_n$ being equal to $\gamma$.
These functions  were recently solved  exactly via conformal bootstrap; see~\cite[Theorem~1.1]{GKRV-sphere} and~\cite[Theorem~9.3]{GKRV-Segal}. The probability $P_\kappa(z_1,\cdots, z_n)$ seems quite difficult to evaluate exactly. We list it as Problem~\ref{prob-sle-order} in Section~\ref{sec-open-problems}.

\begin{remark}[Positive quenched pattern density]
For a permutation $\sigma$ of size $n$, let $\pocc(\sigma, \perm)=\BB P[\Perm_n( \perm,\vec{X},\vec{Y}) =\sigma \mid \perm ]$ be the quenched pattern density of $\sigma$ in $\perm$.  It was shown in~\cite[Theorem 1.10]{bhsy-baxter-permuton} that $\pocc(\sigma, \perm)>0$ a.s.\ for all permutations $\sigma$ when $\perm$ is a skew Brownian permuton (with $\rho\neq 1$). 
%The proof proceeds by first showing that a certain event $E = E(\sigma)$ depending on $( \eta_1,\eta_2)$ occurs with positive probability~\cite[Lemma 4.1]{bhsy-baxter-permuton}. This is done using lemmas from~\cite{ig4} which says that the boundary curves of a space-filling SLE stopped when it hits a given point $z\in\BB C$ can be made to approximate certain fixed deterministic curves with positive probability. The definition of $E$ implies that if $E$ occurs, then $\pocc(\sigma, \perm) > 0$. Using a scale invariance plus tail triviality argument, the authors then deduce that a.s.\ the event $E$ occurs at some Euclidean scale, from which they infer that $\pocc(\sigma, \perm) > 0$ a.s.  
Essentially the same arguments as in~\cite[Section 4]{bhsy-baxter-permuton} show that $\pocc(\sigma,\perm) > 0$ a.s.\ in the setting of Theorem~\ref{thm:density} and also that  $P_\kappa(z_1,\dots,z_n) > 0$ for each $\kappa > 4$ and each distinct $z_1,\dots,z_n \in \BB C$.
%The bulk of the proof was to show that $\BB P[\pocc(\sigma, \perm)]>0$ for all $\sigma$, after which a standard tail triviality argument gives the almost sure occurrence. In  the setting of Theorem~\ref{thm:density} since $P_\kappa(z_1,\cdots, z_n)>0$ for all $z_1,\cdots, z_n$, we immediately have $\BB P[\pocc(\sigma, \perm)]>0$ for all $\sigma$. \xin{Now essentially the same proof as in~\cite{bhsy-baxter-permuton} shows that  $\pocc(\sigma, \perm)>0$ a.s.} 
As explained in~\cite[Corollary 1.8]{bhsy-baxter-permuton}, this implies that random permutations converging to $\perm$ as in Theorem~\ref{thm:density} have asympotically positive density for all permutation patterns  in the quenched sense. 
\end{remark}

\section{Preliminaries}
\label{sec-background}

\subsection{Gaussian free field} 
\label{sec-gff}

In what follows, we use the notation
\eqb \label{eqn-ball}
B_r(z) := \text{open Euclidean ball of radius $r$ centered at $z$}. 
\eqe
For $z , w \in\BB C$, we also define
\eqb \label{eqn-cov-kernel}
|z|_+ := \max\{|z| , 1\} \quad \text{and} \quad G_{\BB C}(z,w) := \log \frac{|z|_+  |w|_+ }{|z-w| }  .
\eqe 
The \textbf{whole-plane Gaussian free field (GFF)} is the centered Gaussian random generalized function (distribution) on $\BB C$ with covariances
\eqb \label{eqn-gff-cov}
\op{Cov}(h^{\BB C}(z) , h^{\BB C}(w))  := G_{\BB C}(z,w). 
\eqe
Note that $h^{\BB C}$ is not well-defined pointwise since the covariance kernel~\eqref{eqn-cov-kernel} diverges to $\infty$ as $z\rta w$. 
Nevertheless, for $z\in\BB C$ and $r>0$, one can define the average of $h^{\BB C}$ over the circle $\bdy B_r(z)$, which we denote by $h^{\BB C}_r(z)$. 
See~\cite[Section 3.1]{shef-kpz} for further discussion. 
 
One often views the whole-plane GFF as being defined modulo additive constant. Our choice of covariance kernel in~\eqref{eqn-gff-cov} corresponds to fixing the additive constant so that the circle average $h_1^{\BB C}(0)$ is zero (see, e.g.,~\cite[Section 2.1.1]{vargas-dozz-notes}). The law of the whole-plane GFF, viewed modulo additive constant, is invariant under complex affine transformations of $\BB C$. Equivalently,
\eqb \label{eqn-gff-translate}
h^{\BB C} \eqD h^{\BB C}(a\cdot + b) - h^{\BB C}_{|a|}(b) , \quad\forall a\in\BB C\setminus \{0\}, \quad\forall b\in\BB C . 
\eqe

{Let $U\subset\BB C$ be an open domain such that Brownian motion started at any point of $U$ a.s.\ exits $U$ at a finite time. We will also have occasion to mention the \textbf{zero-boundary GFF} on $U$, which is the centered Gaussian process on $U$ with covariance kernel given by the zero-boundary Green's function on $U$. Unlike the whole-plane GFF, one does not need to make a choice of additive constant for the zero-boundary GFF.}

We refer to~\cite{shef-gff,bp-lqg-notes,pw-gff-notes} for more background on the GFF. 

\subsection{Liouville quantum gravity} 
\label{sec-lqg}

In this subsection we give a brief review of the theory of Liouville quantum gravity (LQG). Our intention is to provide just enough background for the reader to understand the proofs in this paper. We refer to~\cite{gwynne-ams-survey,sheffield-icm} for brief introduction to LQG and to~\cite{bp-lqg-notes} for a detailed exposition. 
Let $\gamma \in (0,2)$ and let
\eqb \label{eqn-Q}
Q := \frac{2}{\gamma} + \frac{\gamma}{2} .
\eqe

\begin{defn} \label{def-lqg-surface}
A \textbf{$\gamma$-Liouville quantum gravity (LQG) surface with $k\in\BB N_0$ marked points} is an equivalence class of $k+2$-tuples $(U,h,z_1,\dots,z_k)$, where 
\begin{itemize}
\item $U\subset\BB C$ is open and $z_1,\dots,z_k \in U\cup \bdy U$, with $\bdy U$ viewed as a collection of prime ends;
\item $h$ is a generalized function on $U$ (which we will always take to be some variant of the Gaussian free field);
\item $(U,h,z_1,\dots,z_k)$ and $(\wt U, \wt h,\wt z_1,\dots,\wt z_k)$ are declared to be equivalent if there is a conformal map $\phi : \wt U \rta U$ such that
\eqb \label{eqn-lqg-coord}
\wt h = h\circ \phi + Q\log|\phi'| \quad \text{and} \quad \phi(\wt z_j)  = z_j ,\quad \forall j =1,\dots k ,
\eqe
where $Q$ is as in~\eqref{eqn-Q}.
\end{itemize}
If $(U,h,z_1,\dots,z_k)$ is an equivalence class representative, we refer to $h$ as an \textbf{embedding} of the quantum surface into $(U,z_1,\dots,z_k)$. 
\end{defn}
 
In this paper, all of the LQG surfaces we consider will be random, and the generalized function $h$ will be a \textbf{GFF plus a continuous function}, meaning that there is a coupling $(h , h^U)$ of $h$ with the whole-plane or zero-boundary GFF $h^U$ on $U$ (as appropriate) such that a.s.\ $h-h^U$ is a function which is continuous on $U$ except possibly at finitely many points.

If $h$ is a whole-plane GFF plus a continuous function, one can define the \textbf{LQG area measure} $\mu_h$ on $U$, which is a limit of regularized versions of $e^{\gamma h} \,d^2 z$, where $d^2z$ denotes Lebesgue measure on $U$~\cite{kahane,shef-kpz,rhodes-vargas-log-kpz}. Almost surely, the measure $\mu_h$ assigns positive mass to every open set and zero mass to every point, but is mutually singular with respect to Lebesgue measure. 

It was shown in~\cite{dddf-lfpp,gm-uniqueness} that one can also define the \textbf{LQG metric} $D_h$ on $U$, which is a limit of regularized versions of the Riemannian distance function associated with the Riemannian metric tensor $e^{\gamma h} (dx^2+dy^2)$, where $dx^2+dy^2$ is the Euclidean metric tensor on $U$. Almost surely, the metric $D_h$ induces the same topology on $U$ as the Euclidean metric and is a length metric (i.e., $D_h(z,w)$ is the infimum of the $D_h$-lengths of paths in $U$ from $z$ to $w$). The Hausdorff dimension of $(U,D_h)$ is a.s.\ equal to a deterministic, $\gamma$-dependent constant $d_\gamma > 2$ which is not known explicitly~\cite[Corollary 1.7]{gp-kpz}. See~\cite{ddg-metric-survey} for a survey of known results about the LQG metric. 

Both $\mu_h$ and $D_h$ are compatible with coordinate changes of the form~\eqref{eqn-lqg-coord}. That is, if $\phi : \wt U \rta U$ is a conformal map and $h$ and $\wt h$ are related as in~\eqref{eqn-lqg-coord}, then a.s.\
\eqb \label{eqn-metric-measure-coord}
\mu_h(A) = \mu_{\wt h}(\phi^{-1}(A)), \quad \text{$\forall$ Borel set $A\subset U$} \quad \text{and} \quad
D_h(z,w) = D_{\wt h}(\phi^{-1}(z) , \phi^{-1}(w)), \quad\forall z,w\in U .
\eqe 
See~\cite[Proposition 2.1]{shef-kpz} for $\mu_h$ and~\cite[Theorem 1.3]{gm-coord-change} for $D_h$. 

Both $\mu_h$ and $D_h$ are also locally determined by $h$. In the case of $\mu_h$, this means that for each open set $V\subset \BB C$, a.s.\ $\mu_h|_V$ is a measurable function of $h|_V$ (this follows, e.g., from the circle average construction of $\mu_h$ in~\cite{shef-kpz}). 
In the case of $D_h$, the locality property says the following. For $V\subset \BB C$, we define the \textbf{internal metric on $V$} to be the metric $D_h(\cdot,\cdot;V)$ on $V$ defined by
\eqb \label{eqn-internal-metric}
D_h(z,w;V) = \inf\left\{ \left(\text{$D_h$-length of $P$} \right) \,:\, \text{$P$ is a path in $V$ from $z$ to $w$}\right\}.
\eqe 
Then a.s.\ $D_h(\cdot,\cdot;V)$ is given by a measurable function of $h|_V$~\cite[Axiom II]{gm-uniqueness}.  

\subsection{The quantum sphere}
\label{sec-sphere}

For $\gamma\in (0,2)$, the unit area quantum sphere is the canonical LQG surface with the topology of the sphere. One sense in which this LQG surface is canonical is that it is expected, and in some cases proven, to describe the scaling limit of various types of random planar maps. There are a variety of different ways of defining the unit area quantum sphere. The first definitions appeared in~\cite{wedges,dkrv-lqg-sphere} and were proven to be equivalent in~\cite{ahs-sphere}. The definition we give here is~\cite[Definition 2.2]{ahs-sphere} with $k=3$ and $\alpha_1=\alpha_2=\alpha_3=\gamma$. 

\begin{defn} \label{def-sphere} 
Using the notation~\eqref{eqn-cov-kernel}, we define the \textbf{Liouville field}
\eqb \label{eqn-sphere-gaussian}
h_{\op{L}} := h^{\BB C} +  \gamma G_{\BB C}(1 ,\cdot) + \gamma G_{\BB C}(0,\cdot)  - (2Q-\gamma) \log|\cdot|_+ 
\eqe  
and we let $\mu_{h_{\op{L}}}$ be its associated $\gamma$-LQG measure. 

\begin{itemize} 
\item The \textbf{triply marked unit area quantum sphere} is the LQG surface $(\BB C , h , \infty, 0,1 )$, where the law $\BB P_h$ of $h$ is obtained from the law  $\BB P_{\tilde h_{\op{L}}}$ of
\eqb \label{eqn-sphere-subtract}
\tilde h_{\op{L}}:=h_{\op{L}} - \frac{1}{\gamma} \log   \mu_{h_{\op{L}}}(\BB C)
\eqe 
weighted by a $\gamma$-dependent constant times $\left[\mu_{h_{\op{L}}}(\BB C)\right]^{4/\gamma^2 - 2}$, i.e.\ $\frac{d\BB P_h}{d \BB P_{\tilde h_{\op{L}}}}=C(\gamma)\cdot\left[\mu_{h_{\op{L}}}(\BB C)\right]^{4/\gamma^2 - 2}$. 
\item The \textbf{doubly (resp.\ singly) marked unit area quantum sphere} is the doubly (resp.\ singly) marked quantum surface obtained from the triply marked unit area quantum sphere by forgetting the marked point at 1 (resp.\ the marked points at 1 and 0). 
\end{itemize}
\end{defn} 

We note that subtracting $\frac{1}{\gamma} \log \mu_{h_{\op{L}}}(\BB C)$ in~\eqref{eqn-sphere-subtract} makes it so that the unit area quantum sphere satisfies $\mu_h(\BB C) = 1$ a.s. The triply marked quantum sphere has only one possible embedding in $(\BB C , \infty,0,1)$ (Definition~\ref{def-lqg-surface}). This is because the only conformal automorphism of the Riemann sphere which fixes $\infty,0,1$ is the identity. By contrast, the singly (resp.\ doubly) marked quantum sphere has multiple possible embeddings, obtained by applying~\eqref{eqn-lqg-coord} when $\phi$ is a complex affine transformation (resp.\ a rotation and scaling). Any statements for a singly or doubly marked quantum sphere are assumed to be independent of the choice of embedding, unless otherwise specified. Many of the results in this paper are stated for a singly marked quantum sphere since we only need one marked point (i.e., the starting point of the SLE curves). However, it is sometimes convenient to work with a triply marked quantum sphere so that we can apply the exact description of the field from Definition~\ref{def-sphere}.

It is shown in~\cite[Proposition A.13]{wedges} that the marked points for the quantum sphere are independent samples from its associated LQG area measure. 

\begin{lem}[\cite{wedges}] \label{lem-sphere-pt}
Let $(\BB C  , h , \infty , 0 , 1  )$ be a triply marked unit area $\gamma$-Liouville quantum sphere. Conditional on $h$, let $Z_1,Z_2,Z_3\in\BB C$ be conditionally independent samples from the $\gamma$-LQG area measure $\mu_h$.  Then $(\BB C , h , Z_1,Z_2,Z_3)$ is a triply marked unit area quantum sphere, i.e., if $\phi : \BB C \cup\{\infty\} \rta \BB C\cup \{\infty\}$ is a conformal map such that $\phi(\infty)=Z_1$, $\phi(0) =Z_2$, and $\phi(1) =Z_3$, then $h\circ \phi  +Q\log|\phi'| \eqD h$. 
\end{lem}

The following absolute continuity statement allows us to transfer various results from the whole-plane GFF to the quantum sphere. 

\begin{lem} \label{lem-sphere-abs-cont}
Let $(\BB C , h , \infty, 0, 1  )$ be a triply marked quantum sphere, let $h^{\BB C}$ be a whole-plane GFF, and let
\eqb \label{eqn-sphere-abs-cont}
\wt h_* := h^{\BB C} - \gamma\log|\cdot| - \gamma \log|\cdot-1| 
\qquad 
\text{and} 
\qquad 
\wt h := h^{\BB C} - \gamma\log |\cdot|  .
\eqe
For each $R > 0$, the laws of $h|_{B_R(0)}$ and $\wt h_*|_{B_R(0)}$, viewed as distributions modulo additive constant, are mutually absolutely continuous.
Furthermore, for each $R > 0$, the laws of $h|_{B_R(0) \setminus \ol{B_{1/R}(1)}}$ and $\wt h|_{B_R(0) \setminus \ol{B_{1/R}(1)}}$, viewed as distributions modulo additive constant, are mutually absolutely continuous. 
\end{lem}
\begin{proof}
Let $h_{\op{L}}$ be the Liouville field from \eqref{eqn-sphere-gaussian}, defined using the same whole-plane GFF $h^{\BB C}$ as in~\eqref{eqn-sphere-abs-cont}. We first look at the case of $\wt h_*$. From the definition of $G_{\BB C}$ in~\eqref{eqn-cov-kernel}, we see that the difference $f := \wt h_* - h_{\op{L}}$ is a deterministic continuous function on $\BB C$ with locally finite Dirichlet energy, i.e., $\int_{B_R(0)} |\nabla f(z)|^2 \,d^2z < \infty$ for each $R>0$. From standard local absolute continuity results for the GFF (see, e.g.,~\cite[Proposition 2.9]{ig4}), the laws of $\wt h_*|_{B_R(0)}$ and $h_{\op{L}}|_{B_R(0)}$, viewed as distributions modulo additive constant, are mutually absolutely continuous for each $R>0$. By Definition~\ref{def-sphere}, the laws of $h$ and $h_{\op{L}}$ are mutually absolutely continuous, viewed as distributions modulo additive constant. This gives the statement for $\wt h_*$. The statement for $\wt h$ follows from the statement for $\wt h_*$ since $\gamma\log|\cdot-1|$ is smooth away from 1, so again by~\cite[Proposition 2.9]{ig4}, the fields $\wt h|_{B_R(0) \setminus \ol{B_{1/R}(1)}}$ and $\wt h_*|_{B_R(0) \setminus \ol{B_{1/R}(1)}}$, viewed as distributions modulo additive constant, are mutually absolutely continuous. 
\end{proof}

We now make two remarks which are not needed for the proofs of our main results, but which are relevant to the physics heuristics in Section~\ref{sec-meander-conjecture}.

\begin{remark}[Infinite measure on quantum spheres] \label{remark-infinite}
There is a natural infinite measure on (free area) quantum spheres which appears in~\cite{wedges,dkrv-lqg-sphere}. 
To produce a ``sample'' from the infinite measure on doubly marked (free area) quantum spheres, we first sample an area $A$ from the infinite measure $\BB 1_{(a>0)}a^{-\frac{4}{\gamma^2} } \, da$. 
We then consider the quantum surface $(\BB C , h + \frac{1}{\gamma} \log A , \infty , 0 )$, where $(\BB C , h , \infty , 0)$ is a doubly marked unit area quantum sphere~\cite[Lemma 2.8]{acsw-loop}. The infinite measure on singly marked (free area) quantum spheres is defined similarly, except that $A$ is instead sampled from $\BB 1_{(a>0)}a^{-\frac{4}{\gamma^2}-1} \, da$.  
\end{remark}

\begin{remark}[Physics interpretation and random planar maps connection] \label{remark-lqg-physics}
A $\gamma$-LQG surface can be interpreted as a model of gravity (essentially, random geometry) coupled to matter, where the matter is represented by a conformal field theory (CFT). The central charge of this CFT is related to the coupling constant $\gamma\in(0,2)$ by
\eqb \label{eqn:lqg-ccM}
\ccM = 25 - 6 \left( \frac{2}{\gamma} + \frac{\gamma}{2} \right)^2 \in (-\infty,1) . 
\eqe
This interpretation can be made rigorous by looking at random planar maps. Indeed, suppose that $(M_n , S_n)$ is a random planar map with $n$ edges decorated by a statistical physics model. We sample $M_n$ and $S_n$ simultaneously, so the marginal law of $M_n$ is given by the uniform measure on some collection of planar maps weighted by the partition function of $S_n$. Suppose our statistical physics model $S_n$ is such that the scaling limit of the version of $S_n$ in a regular lattice (e.g., $\BB Z^2$) is described by a CFT with central charge $\ccM$. Then in many cases the scaling limit of $M_n$ (e.g., with respect to the Gromov-Hausdorff distance) should be described by $\gamma$-LQG, where $\gamma$ and $\ccM$ are related as in~\eqref{eqn:lqg-ccM}. As an example, the scaling limit of the critical Ising model on $\BB Z^2$ is described by a CFT with central charge $\ccM = 1/2$, so random planar maps weighted by the critical Ising model partition function should converge to LQG with matter central charge $\ccM = 1/2$, equivalently, $\gamma=\sqrt3$. See~\cite[Section 3.1]{ghs-mating-survey} for more discussion. Using the formula for the law of the area for the infinite measure on single marked quantum spheres (Remark~\ref{remark-infinite}), it is conjectured that the partition function of $(M_n,S_n)$ grows like $A^n n^{-\alpha}$ where $A>0$ and
\begin{equation}\label{eq:exp}
\alpha=1+\frac{4}{\gamma^2}.
\end{equation}	
This explains why Conjecture~\ref{conj-meander} is related to the meander exponent $\alpha = (29 + \sqrt{145})/12$ in Section~\ref{sec-meander}; see Section~\ref{sec-meander-conjecture} for more details. 
\end{remark}

\subsection{Space-filling SLE} 
\label{sec-sle}

Let $\kappa > 4$. The \textbf{whole-plane space-filling SLE$_\kappa$} from $\infty$ to $\infty$ is a random non-self-crossing space-filling curve $\eta$ in $\BB C$ which starts and ends at $\infty$. We view $\eta$ as a curve defined modulo monotone increasing time re-parametrization. We give a precise definition of this curve below. Before we do so, we record some of its qualitative properties which follow from the construction in~\cite{ig4}. 
\begin{enumerate}[$(i)$]
\item When $\kappa\geq 8$, the curve $\eta$ is a two-sided version of chordal SLE$_\kappa$. It can be obtained from chordal SLE$_\kappa$ by stopping the chordal SLE curve at the first time when it hits some fixed interior point $z$, then ``zooming in'' near $z$. In particular, for any two times $s < t$ the region $\eta([s,t])$ is simply connected.
\item When $\kappa\in (4,8)$, the curve $\eta$ can be obtained from a two-sided variant of chordal SLE$_\kappa$ by iteratively ``filling in'' the bubbles which it disconnects from its target point. In this case, for typical times $s < t$, the region $\eta([s,t])$ is not simply connected. 
\item For each fixed $z\in\BB C$, the left and right outer boundaries of $\eta$ stopped when it hits $z$ are a pair of coupled SLE$_{16/\kappa}$-type curves from $z$ to $\infty$.
\item For any times $a < b$, the set $\eta([a,b])$ has non-empty interior. Moreover, if $\mu$ is a locally finite, non-atomic measure on $\BB C$ which assigns positive mass to every open set, then $\eta$ is continuous if we parametrize $\eta$ so that $\mu(\eta([a,b])) = b-a$ for each $a  <b$. %We know that $\eta$ is continuous when parametrized by Lebesgue measure. Assume that $\eta$ is parametrized this way. For $t\geq 0$, let $F(t) = \mu(\eta([0,t]))$. Since $\eta$ absorbs an open set during every non-trivial time interval, we get that $F$ is strictly increasing. Furthermore, $F$ has no upward jumps since $\eta$ is continuous and $\mu$ is non-atomic. Therefore, $F$ is continuous. By a compactness argument, $F$ is a homeomorphism. 
\item The law of $\eta$ is invariant under scaling, translation, rotation, and time reversal, i.e., $a\eta + b$ and $\eta( - \cdot)$ have the same law as $\eta$ modulo time parametrization for each $a \in\BB C \setminus \{0\}$ and each $b\in\BB C$. 
\item For each fixed $z\in\BB C$, a.s.\ $\eta$ hits $z$ exactly once. The maximum number of times that $\eta$ hits any $z\in\BB C$ is 3 if $\kappa\geq 8$ and is a finite, $\kappa$-dependent number if $\kappa \in (4,8)$~\cite[Theorem 6.3]{ghm-kpz}.
\end{enumerate}

We now give a formal definition of whole-plane space-filling SLE$_\kappa$, following~\cite[Section 1.2.3]{ig4} (see also~\cite[Section 1.4.1]{wedges}). The idea of the definition is based on SLE duality, which says that for $\kappa > 4$ the outer boundary of an SLE$_\kappa$ curve stopped at a given time should consist of one or more SLE$_{16/\kappa}$-type curves~\cite{zhan-duality1,zhan-duality2,dubedat-duality,ig1,ig4}. Let
\eqbn
\chi := \frac{\sqrt\kappa}{2} - \frac{2}{\sqrt\kappa} .
\eqen
Let $\wh h$ be a whole-plane GFF viewed modulo additive multiples of $2\pi\chi$, as in~\cite{ig4}. 
Also fix an angle $\theta \in [-\pi/2, \pi/2]$. We will define a whole-plane space-filling SLE$_\kappa$ curve $\eta$ from $\infty$ to $\infty$ coupled with $\wh h$. {The law of $\eta$ will not depend on $\theta$, but its coupling with $\wh h$ will depend on $\theta$.}

For $z\in\BB Q^2$, let $\beta_z^L$ and $\beta_z^R$ be the flow lines of $e^{i\wh h/\chi}$ started from $z$, in the sense of~\cite[Theorem 1.1]{ig4}, with angles\footnote{{We measure angles in counter-clockwise order, where zero angle corresponds to the north direction.}} $\theta$ and $\theta-\pi$, respectively. Each of these flow lines is a whole-plane SLE$_{16/\kappa}(2-16/\kappa )$ curve from $z$ to $\infty$ (whole-plane SLE$_{16/\kappa}(\rho)$ for $\rho > -2$ is a variant of whole-plane SLE introduced in~\cite[Section 2.1]{ig4}), and the field $\wh h$ defines a coupling between them. These flow lines will be the left and right outer boundaries of the space-filling SLE$_\kappa$ $\eta$ stopped when it hits $z$. 

By~\cite[Theorem 1.9]{ig4}, for distinct $z,w\in\BB Q^2$, the flow lines $\beta_z^L$ and $\beta_z^R$ cannot cross $\beta_w^L \cup \beta_w^R$, but they may hit and bounce off if $\kappa \in (4,8)$. Furthermore, a.s.\ the flow lines $\beta_z^L$ and $\beta_w^L$ (resp.\ $\beta_z^R$ and $\beta_w^R$) eventually merge into each other. See Figure~\ref{fig-flow-line-def} for an illustration.

\begin{figure}[ht!]
\begin{center}
\includegraphics[width=.4\textwidth]{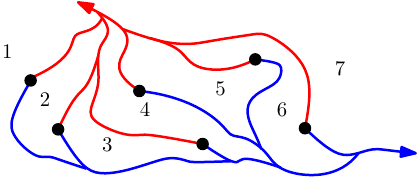}  
\caption{\label{fig-flow-line-def} The flow lines $\beta_z^L$ (in red) and $\beta_z^R$ (in blue) for several points $z\in\BB C$. For each $z$, these flow lines are the left and right outer boundaries, respectively, of $\eta$ stopped when it hits $z$. The figure shows the case when $\kappa\geq 8$. When $\kappa\in (4,8)$, the flow lines started from the same point can hit each other and bounce off (c.f.\ Figure~\ref{fig-space-filling-def}). The space-filling SLE$_\kappa$ fills in the regions 1,2,3,4,5,6, and 7 in this order.
}
\end{center}
\vspace{-3ex}
\end{figure}

We define a total ordering on $\BB Q^2$ by declaring that $z$ comes before $w$ if and only if $w$ lies in a connected component of $\BB C\setminus (\beta_z^L\cup \beta_z^R)$ which lies to the right of $\beta_z^L$ and to the left of $\beta_z^R$. It follows from~\cite[Theorem 1.16]{ig4} (see also~\cite[Footnote 4]{wedges}) that there is a unique space-filling path $\eta$ from $\infty$ to $\infty$ which hits the points of $\BB Q^2$ in the prescribed order and is continuous when it is parametrized, e.g., so that it traverses one unit of (two-dimensional) Lebesgue measure in one unit of time.  This curve $\eta$ is defined to be the \textbf{space-filling SLE$_\kappa$ counterflow line of $\wh h$ from $\infty$ to $\infty$ of angle\footnote{We highlight that in this paper we are using the convention that counterflow line of angle $\theta$ have the \emph{same} direction as flow line of angle $\theta$. We point out that in other papers (for instance, \cite{ig1}) the authors might use the opposite convention, that is, counterflow line of angle $\theta$ have the \emph{opposite} direction as flow line of angle $\theta$.} $\theta-\pi/2$}.
 
The law of the curve $\eta$ does not depend on $\theta$, but the coupling of $\eta$ with $\wh h$ depends on $\theta$. In particular, if $\theta \in (-\pi/2,\pi/2)$ and $\eta_1$ and $\eta_2$ are the space-filling SLE counterflow lines of $\wh h$ with angles 0 and $ \theta-\pi/2 $, respectively, then $\eta_1$ and $\eta_2$ are coupled together in a non-trivial way. By~\cite[Theorem 1.16]{ig4}, each of $\eta_1$ and $\eta_2$ is a.s.\ given by a measurable function of the other. The reason for our interest in this coupling is the following result, which is~\cite[Theorem 1.17]{borga-skew-permuton} and which is proven building on~\cite{wedges,ghs-bipolar}.

\begin{thm}[\cite{borga-skew-permuton}] \label{thm-skew-permuton}
Fix $\rho\in(-1,1)$ and $q\in [0,1]$. Also let $\gamma \in (0,2)$ and $\kappa>4$ satisfy $\rho = -\cos(\pi\gamma^2/4)$ and $\kappa=16/\gamma^2$. There exists $\theta  = \theta(\rho,q) \in [-\pi/2,\pi/2]$ such that the following is true. Let $(\BB C ,h,\infty)$ be a singly marked unit area $\gamma$-Liouville quantum gravity sphere. Also let $\eta_1,\eta_2$ be the space-filling SLE$_\kappa$ counterflow lines of angles 0 and $ \theta-\pi/2 $ of a common whole-plane GFF {$\wh h$ as above. We take $\wh h$ (and hence $(\eta_1,\eta_2)$, viewed modulo time parametrization) to be independent from $h$.} We then parametrize $\eta_1$ and $\eta_2$ so that they each traverse one unit of $\gamma$-LQG mass in one unit of time, {i.e., $\mu_h(\eta_i([0,t])) = t$ for each $t \in [0,1]$ and $i\in \{1,2\}$}. Then the permutation $\perm$ associated with $(\eta_1,\eta_2)$ as in~\eqref{eqn-permuton-def} coincides with the skew Brownian permuton with parameters $\rho$ and $q$.
\end{thm}

The exact correspondence between the parameters $\theta$ and $q$ is not known, but it is known that $\theta(\rho,1/2) = 0$ and that $q\mapsto \theta(\rho,q)$ is a homeomorphism from $[0,1]$ to $[-\pi/2,\pi/2]$~\cite[Remark 1.18]{borga-skew-permuton}.
%We note that the Baxter permuton (i.e., the scaling limit of Baxter permutations) corresponds to $(\rho,q) = (-1/2,1/2)$ and hence to $(\gamma,\kappa,\theta) = (\sqrt{4/3} ,12, \pi/2)$. 

\subsection{Basic properties of permutons constructed from SLEs and LQG}
\label{sec-permuton-basic}

Assume that we are in the setting of Section~\ref{sec-permuton-def}. In particular, $(\BB C  , h , \infty)$ is a singly marked unit area $\gamma$-Liouville quantum sphere and
$(\eta_1,\eta_2)$ is a coupling of a space-filling SLE$_{\kappa_1}$ curve and a space-filling SLE$_{\kappa_2}$ curve, sampled independently from $h$ and then parametrized by $\mu_h$-mass; $\psi : [0,1]\rta [0,1]$ satisfies $\eta_1(t) = \eta_2(\psi(t))$; and $\perm$ is the permuton as in~\eqref{eqn-permuton-def}. In this subsection, we will establish some basic properties of $\perm$. We first check that $\perm$ is well-defined, and give an equivalent definition. 

\begin{lem} \label{lem-permuton-defined}
Almost surely, the permuton $\perm$ is well-defined, i.e., the definition does not depend on the choice of $\psi$. Moreover, a.s.\ for each rectangle $[a,b]\times[c,d] \subset [0,1]^2$, 
\eqb \label{eqn-permuton-formula} 
\perm\left([a,b]\times[c,d]\right) = \mu_h\left( \eta_1([a,b]) \cap \eta_2([c,d]) \right)  .
\eqe
\end{lem}

We note that a similar statement to Lemma~\ref{lem-permuton-defined} is proven for the skew Brownian permuton in~\cite[Proposition 1.13]{bhsy-baxter-permuton}.

\begin{proof}[Proof of Lemma~\ref{lem-permuton-defined}]
The formula~\eqref{eqn-permuton-formula} uniquely determines $\perm$ and does not depend on $\psi$, so it suffices to prove~\eqref{eqn-permuton-formula}. 
For each $z\in \BB C$, a.s.\ $z$ is not a multiple point of $\eta_2$, i.e., $\eta_2$ hits $z$ exactly once. 
Since $h$ is independent from $(\eta_1,\eta_2)$ and $\eta_1$ and $\eta_2$ are parametrized by $\mu_h$-mass, a.s.\ the set of times $t\in [0,1]$ such that $\eta_1(t)$ is a multiple point of $\eta_2$ has zero Lebesgue measure. 

If $\eta_1(t)$ is not a multiple point of $\eta_2$, then for any choice of $\psi$ satisfying~\eqref{eqn-psi-property}, $\psi(t) \in [c,d]$ if and only if $\eta_1(t) \in \eta_2([c,d])$. By the preceding paragraph, a.s.\ this is the case for Lebesgue-a.e.\ $t\in [0,1]$, so a.s.\ for every rectangle $[a,b]\times[c,d]$, 
\alb
\perm\left([a,b]\times[c,d]\right)
&= \op{Leb}\left\{ t\in [a,b] : \psi(t) \in [c,d] \right\} \quad \text{(by~\eqref{eqn-permuton-def})} \notag\\
&= \op{Leb}\left\{ t\in [a,b] : \eta_1(t) \in \eta_2([c,d]) \right\}  \notag\\
&= \mu_h\left( \eta_1([a,b]) \cap \eta_2([c,d]) \right) ,
\ale
where the last equality is because $\eta_1$ is parametrized by $\mu_h$-mass. 
\end{proof}

We also have the following basic lemma about the closed support of $\perm$ (Definition~\ref{def-supp}).

\begin{lem} \label{lem-permuton-inclusion} 
Almost surely, for any choice of the function $\psi$ from~\eqref{eqn-psi-property},  
\eqb \label{eqn-permuton-inclusion}
\op{supp} \perm \subset \ol{\{(t,\psi(t)) : t\in [0,1]\} }\subset \{(t,s) \in [0,1] : \eta_1(t) = \eta_2(s) \} = \mcl T .
\eqe 
\end{lem}
\begin{proof}
The first inclusion in~\eqref{eqn-permuton-inclusion} is immediate from the definition~\eqref{eqn-permuton-def}. The second inclusion follows from the fact that $\eta_1(t) = \eta_2(\psi(t))$ for each $t\in [0,1]$ by definition; and the fact that the set $\{(t,s) \in [0,1] : \eta_1(s) = \eta_1(t) \}$ is closed (by the continuity of $\eta_1$ and $\eta_2$). 
\end{proof}

Both inclusions in~\eqref{eqn-permuton-inclusion} can potentially be strict. For example, if $\eta_1 = \eta_2$ then $\op{supp} \perm$ is the diagonal in $[0,1]^2$. The third set in~\eqref{eqn-permuton-inclusion} includes off-diagonal points, which arise from pairs of distinct times $(t,s)$ such that $\eta_1(t) = \eta_1(s)$ or $\eta_2(t) = \eta_2(s)$ (which are multiple points of $\eta_1$ or $\eta_2$). The middle set in~\eqref{eqn-permuton-inclusion} can include off-diagonal points or not, depending on the choice of $\psi$. %We plan to investigate the relationships between the sets in~\eqref{eqn-permuton-inclusion} further in future work. 

The following lemma tells us that the ambiguity in the choice of $\psi$ can arise only from multiple points of $\eta_2$. 
 
\begin{lem} \label{lem-permuton-unique}
Almost surely, for any choice of the function $\psi$ from~\eqref{eqn-psi-property}, the following is true. For each time $t\in [0,1]$ such that $\eta_1(t)$ is hit only once by $\eta_2 $, we have that $\psi(t)$ is the unique $s\in [0,1]$ such that $(t,s) \in \op{supp} \perm$. 
Furthermore, for each fixed $t\in [0,1]$, a.s.\ $\psi(t)$ is the unique $s\in [0,1]$ such that $(t,s) \in \op{supp} \perm$. 
\end{lem}
\begin{proof}
Since $\perm$ is a permuton, we have $\perm([a,b] \times [0,1]) = b-a$ for each $0\leq a \leq b \leq 1$. From this, we infer that $\op{supp} \perm$ intersects $[a,b] \times [0,1]$ for each $0\leq a  < b \leq 1$. Since $\op{supp} \perm$ is closed, this implies that $\op{supp}\perm$ intersects $\{t\} \times [0,1]$ for each $t \in [0,1]$. By Lemma~\ref{lem-permuton-inclusion}, if $\eta_2$ hits $\eta_1(t)$ only once, there is only one possible intersection point of $\op{supp} \perm$ with $\{t\}\times [0,1]$, namely $(t,\psi(t))$. This gives the first assertion of the lemma. To deduce the last statement, we note that if $t\in [0,1]$ is fixed, then $\eta_1(t)$ is independent from $\eta_2$ (viewed modulo time parametrization), and hence a.s.\ $\eta_2$ hits $\eta_1(t)$ only once (recall property (vi) at the beginning of Section \ref{sec-sle}).
\end{proof}

\section{Longest increasing subsequence is sublinear}
\label{sec-lis-proof}

In this section we will prove Theorem~\ref{thm-lis}, which asserts that the longest increasing subsequence is sublinear for permutations which converge to $\perm$, {where $\perm$ is a permuton constructed from SLE and LQG which satisfies one of the two conditions in Section~\ref{sec-lis}.}

\subsection{Longest increasing subsequences and monotone sets for permutons}

To prove Theorem~\ref{thm-lis}, we will use a general criterion on a permuton which implies a sublinear bound on the longest increasing subsequence for converging permutations, which we now state and prove.

\begin{figure}[ht!]
 \begin{center}
\includegraphics[scale=.65]{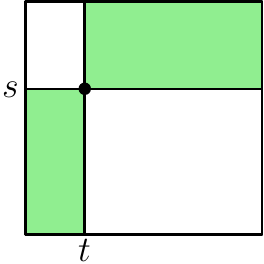}
\vspace{-0.01\textheight}
\caption{A set $A\subset [0,1]^2$ is monotone if for any point $(t,s) \in A$, the set $A$ is contained in the union of the green rectangles in the figure.
}\label{fig-monotone-set}
\end{center}
\vspace{-1em}
\end{figure}

\begin{defn} \label{def-monotone-set}
We say that a set $A\subset [0,1]^2$ is \textbf{monotone} if  
\eqb
A\subset \left( [0,t] \times [0,s]\right) \cup \left( [t,1] \times [s,1] \right) ,\quad \forall (t,s) \in A .
\eqe
\end{defn}

See Figure~\ref{fig-monotone-set} for an illustration. An example of a monotone set is the graph of a non-decreasing function $[0,1] \rta [0,1]$, or a subset of such a graph. The reason for our interest in monotone sets is the following deterministic result.  

\begin{prop} \label{prop-monotone-set}
Let $\perm$ be a permuton and let $\{\sigma_n\}_{n\in\BB N}$ be a sequence of permutations of size $|\sigma_n| \rta\infty$ whose associated permutons $\perm_{\sigma_n}$ converge weakly to $\perm$. 
Assume that $\perm(A) = 0$ for every monotone set $A\subset [0,1]^2$ (Definition~\ref{def-monotone-set}).
Then the longest increasing subsequence of $\sigma_n$ is sublinear, i.e., $\op{LIS}(\sigma_n) / |\sigma_n| \rta 0$. 
\end{prop}

{The proof of Proposition~\ref{prop-monotone-set} is elementary, but we are not aware of a reference for the statement or proof, so we provide the proof here.}

\begin{proof}[Proof of Proposition~\ref{prop-monotone-set}]
We will prove the contrapositive, i.e., we will assume that the largest increasing subsequence of $\sigma_n$ is not sublinear and show that there is a monotone subset of $[0,1]^2$ with positive $\perm$-mass.

Our assumption implies that after possibly replacing $\{\sigma_n\}_{n\in\BB N}$ by a subsequence, we can find $c > 0$ such that for each $n\in\BB N$, there is an increasing subsequence of $\sigma_n$ of length at least $c | \sigma_n|$. 
That is, for each $n\in\BB N$, there is a set $L_n \subset [1,|\sigma_n|] \cap \BB Z$ such that $\# L_n \geq c |\sigma_n|$ and $\sigma_n|_{L_n}$ is monotone increasing. 
For $n\in\BB N$, let
\eqb \label{eqn-discrete-mono}
A_n := \bigcup_{j\in L_n} \left[ \frac{j - 1}{|\sigma_n|} , \frac{ j }{|\sigma_n|} \right] \times \left[ \frac{\sigma_n(j) - 1}{|\sigma_n|} , \frac{ \sigma_n(j) }{|\sigma_n|} \right] \subset [0,1]^2 .
\eqe
By Definition~\ref{def-permutation} of $\perm_{\sigma_n}$, we have $\perm_{\sigma_n}(A_n) \geq c$. 

By compactness, after possibly passing to a further subsequence of $\{\sigma_n\}_{n\in\BB N}$, we can arrange that $A_n$ converges in the Hausdorff distance to a closed set $A\subset [0,1]^2$. We first argue that $\perm(A) > 0$. Indeed, for each $\ep > 0$, it holds for each large enough $n\in\BB N$ that $A_n$ is contained in the Euclidean $\ep$-neighborhood $B_\ep(A)$. Since $\perm_{\sigma_n} \rta \perm$ weakly,
\eqb
\perm(\ol{B_\ep(A)}) \geq \liminf_{n\rta\infty} \perm_{\sigma_n}(\ol{B_\ep(A)}) \geq \liminf_{n\rta\infty} \perm_{\sigma_n}(A_n) \geq c .
\eqe
Sending $\ep \rta 0$ gives $\perm(A) \geq c > 0$. 

We next claim that $A$ is a monotone set in the sense of Definition~\ref{def-monotone-set}. To this end, let $(t,s) \in A$. Since $A_n\rta A$ in the Hausdorff distance, by~\eqref{eqn-discrete-mono} we can find a $j_n \in L_n$ such that $(j_n,\sigma_n(j_n)) / |\sigma_n| \rta (t,s)$ as $n\rta\infty$. Since $\sigma_n$ is monotone on $L_n$, it holds for each $n\in\BB N$ that
\eqb \label{eqn-perm-rectangle}
A_n\subset 
\left( \left[ 0 , \frac{ j_n }{|\sigma_n|} \right] \times \left[ 0, \frac{ \sigma_n(j_n) }{|\sigma_n|} \right] \right) \cup 
\left( \left[ \frac{ j_n }{|\sigma_n|} , 1 \right] \times \left[  \frac{ \sigma_n(j_n) }{|\sigma_n|}  ,1  \right] \right)  .
\eqe 
Since $(j_n,\sigma_n(j_n)) / |\sigma_n| \rta (t,s)$, for each $\ep > 0$ it holds for each large enough $n\in\BB N$ that the set on the right side of~\eqref{eqn-perm-rectangle} is contained in the $\ep$-neighborhood of $([0,t] \times [0,s]) \cup ([t,1] \times [s,1])$. Since $A_n\rta A$, it follows that for each $\ep > 0$,
\eqb
A \subset \left(\text{$\ep$-neighborhood of $([0,t] \times [0,s]) \cup ([t,1] \times [s,1]) $}\right) .
\eqe
Sending $\ep\rta 0$ shows that $A$ is a monotone set.
\end{proof}

\subsection{Monotone sets for permutons constructed from SLEs and LQG}

Henceforth assume that we are in the setting of Theorem~\ref{thm-lis}. To prove the theorem, we need to check the condition of Proposition~\ref{prop-monotone-set} for the random permuton $\perm$. This turns out to be straightforward using SLE and LQG theory.

Let $\eta_1$ and $\eta_2$ be the coupled or independent space-filling SLE curves of parameters $\kappa_1$ and $\kappa_2$ from Theorem~\ref{thm-lis}. 
For $z\in\BB C$, let $\mcl Q_z$ be the set of $w\in\BB C$ such that $\eta_1$ hits $w$ before it hits $z$ and $\eta_2$ hits $w$ after it hits $z$. More precisely, $w\in \mcl Q_z$ if and only if 
\eqb \label{eqn-non-mono-set}
\exists t_1,s_1,t_2,s_2 \in [0,1] \quad \text{such that} \quad t_1 \leq s_1, \quad t_2 \leq s_2, \quad \eta_1(s_1) = \eta_2(t_2) = z,\quad \eta_1(t_1) = \eta_2(s_2) = w .
\eqe
See Figure~\ref{fig-quadrant} for an illustration of $\mcl Q_z$. 

{We claim that $\mcl Q_z$ is a closed set. Indeed, if $\eta_1$ and $\eta_2$ each hit $z$ only once (say at times $s_1$ and $t_2$, respectively) then $\mcl Q_z = \eta_1([0,s_1]) \cap \eta_2([t_2,1])$ is the intersection of closed sets, so is closed (note that this is the case for $\mu_h$-a.e.\ $z\in\BB C$). In general, each of $\eta_1$ and $\eta_2$ hits $z$ at most finitely many times (recall Section~\ref{sec-sle}), so $\mcl Q_z$ is the union of finitely many closed sets of the above type.}
Furthermore, $\mcl Q_z$ depends only on $\eta_1$ and $\eta_2$ viewed modulo time parametrization, so $\mcl Q_z$ is independent from $h$. 
%The definition of $\mcl Q_z$ is slightly simpler if $z$ is hit only once by each of $\eta_1$ or $\eta_2$ (which is true for $\mu_h$-a.e.\ $z\in\BB C$), since in this case there is only one possible choice for each of $s_1$ and $t_2$.

To prove Theorem~\ref{thm-lis}, we will first show that for $\mu_h$-a.e.\ $z\in\BB C$, the set $\mcl Q_z$ has positive density with respect to $\mu_h$, i.e., we will prove the following lemma.

\begin{figure}[t!]
 \begin{center}
\includegraphics[scale=.85]{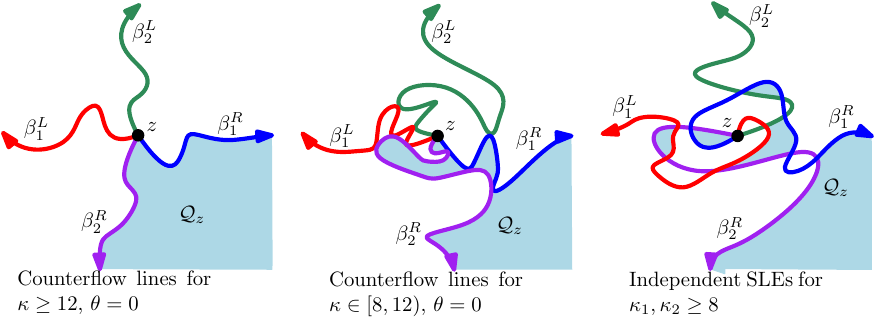}
\vspace{-0.01\textheight}
\caption{Illustration of the set $\mcl Q_z$ of~\eqref{eqn-non-mono-set} for three of the different pairs of SLE curves $(\eta_1,\eta_2)$ considered in Theorem~\ref{thm-lis}. Here, $z$ is a point which is hit only once by each of $\eta_1$ and $\eta_2$.  In each case, $\mcl Q_z$ is the light blue region, $\eta_1$ travels from south to north and $\eta_2$ travels from west to east. The left and right boundaries of $\eta_1$ (resp.\ $\eta_2$) stopped when it hits $z$ are shown in red and blue (resp.\ green and purple). The set $\mcl Q_z$ is the region lying to the south of the union of the red and blue curves and to the east of the union of the green and purple curves. In the case of the skew Brownian permuton, $\eta_1$ and $\eta_2$ are two space-filling counterflow lines of a GFF with angles $0$ and $ \theta-\pi/2 $ (Theorem~\ref{thm-skew-permuton}). This case is illustrated in the left and middle panels. The red, blue, green, and purple curves do not cross each other, but may hit each other and bounce off (as shown in the second panel). One can determine the values of $\kappa$ and $\theta$ for which this happens using~\cite[Proposition 3.28]{ig4} and~\cite[Lemma 15]{dubedat-duality}. In the case of two independent space-filling SLEs (right panel), the red and blue curves can cross the green and purple curves, so the region $\mcl Q_z$ is not connected. The three pictures in the case when $\kappa_1 \in (4,8)$ (resp.\ $\kappa_2 \in (4,8)$) are similar, except that the red and blue curves (resp.\ the green and purple curves) can hit each other. Our proof is the same for all of the different cases considered in Theorem~\ref{thm-lis}. 
}\label{fig-quadrant}
\end{center}
\vspace{-1em}
\end{figure}

\begin{lem} \label{lem-sphere-limsup}
Assume that we are in the setting of Theorem~\ref{thm-lis}.  
Almost surely, 
\eqb \label{eqn-sphere-limsup}
\limsup_{r \rta 0} \frac{\mu_h\left( \mcl Q_z \cap  B_r(z)    \right)}{\mu_h(B_r(z))} > 0 ,\quad \text{for $\mu_h$-a.e.\ $z\in\BB C$} , 
\eqe
where $B_r(z)$ is a Euclidean ball as in~\eqref{eqn-ball}. 
\end{lem}

We will eventually show that every monotone set for $\perm$ has zero $\mu_h$-mass by combining Lemma~\ref{lem-sphere-limsup} with the Lebesgue density theorem for the Radon measure $\mu_h$ (see the proof of Proposition~\ref{prop-sle-mono-set}). 
We will prove Lemma~\ref{lem-sphere-limsup} via a scaling argument. To do this, we need to work in an infinite-volume setting, i.e., we need to work with a field which satisfies an exact spatial scale invariance property in law (the unit area quantum sphere does not satisfy such a property). We choose to work with a whole-plane GFF with a $\gamma$-log singularity at 0 since {the local behavior of this field near zero is the same as the local behavior of the quantum sphere field near a point sampled from its LQG area measure $\mu_h$} (see Lemma~\ref{lem-sphere-abs-cont} for a precise statement).

\begin{lem} \label{lem-gff-limsup}
Let $\gamma\in (0,2)$ and $\kappa_1,\kappa_2 > 4$. Let $\wt h$ be a whole-plane GFF plus $\gamma\log|\cdot|^{-1}$. Let $\eta_1$ and $\eta_2$ be two whole-plane space-filling SLE curves from $\infty$ to $\infty$ of parameters $\kappa_1$ and $\kappa_2$, sampled independently from $\wt h$. Assume that $(\eta_1,\eta_2)$ are coupled together in one of the two possible ways listed just before Theorem~\ref{thm-lis}. 
Define the set $\mcl Q_0$ as in~\eqref{eqn-non-mono-set}. 
Almost surely, 
\eqb \label{eqn-gff-limsup}
\limsup_{r \rta 0} \frac{\mu_{\wt h}\left( \mcl Q_0 \cap  B_r(0)    \right)}{\mu_{\wt h}(B_r(0))} > 0 .
\eqe
\end{lem}

\begin{proof} 
We first claim that a.s.\ $\mcl Q_0$ has non-empty interior.   
For $i\in \{1,2\}$, let $\tau_0^i$ be the (a.s.\ unique) time when $\eta_i$ hits 0.  
By the construction of space-filling SLE in Section~\ref{sec-sle}, the set $\bdy \eta_i((-\infty,\tau_0^i])$ is the union of two SLE$_{16/\kappa_i}$-type curves $\beta_i^L$ and $\beta_i^R$ from 0 to $\infty$, namely, the left and right boundaries of $\eta_i((-\infty,\tau_0^i])$.  
The set $\mcl Q_0$ is the closure of the intersection of the region $\eta_1((-\infty,\tau_0^1])$ which lies to the left of $\beta_1^L$ and to the right of $\beta_1^R$; and the region $\eta_2([\tau_0^2,\infty))$ which lies to the right of $\beta_2^L$ and to the left of $\beta_2^R$. See Figure~\ref{fig-quadrant} for an illustration. 

{
In the case when $\eta_1$ and $\eta_2$ are coupled together as in the SLE/LQG description of the skew Brownian permuton, we recall from Section~\ref{sec-sle} that $\beta_1^L , \beta_1^R,\beta_2^L$, and $\beta_2^R$ are the flow lines of a whole-plane GFF with angles $ \pi/2$, $-\pi/2$, $\theta$, and $\theta-\pi$, respectively. 
It follows from~\cite[Proposition 3.28]{ig4} that a.s.\ $\beta_2^R \subset \eta_1((-\infty,\tau_0^1])$, $\beta_2^R$ does not spend a non-trivial interval of time tracing $\bdy \eta_1((-\infty,\tau_0^1])$, and the intersection of $\beta_2^R$ with the closure of each connected component of the interior of $\eta_1((-\infty,\tau_0^1])$ is a simple curve between two points of the boundary of the component. Each of these closures is homeomorphic to a closed Euclidean disk. By the Jordan curve theorem, $\beta_2^R$ divides each connected component of the interior of $\eta_1((-\infty,\tau_0^1])$ into two non-empty regions. The one of these regions to the left of $\beta_2^R$ is contained in $\mcl Q_0$; hence $\mcl Q_0$ has non-empty interior. 

In the case when $\eta_1$ and $\eta_2$ are independent, modulo time parametrization, we first note that $\beta_2^L$ has a positive chance to cross $\beta_1^L$ (this follows, e.g., from~\cite[Lemma 3.8]{ig4} applied under the conditional law given $\beta_1^L$). By the scale invariance of the joint law of $(\beta_1^L , \beta_2^L)$, we get that there exists $q > 0$ such that for every $\ep > 0$, $\beta_1^L$ and $\beta_2^L$ have probability at least $q$ to cross before leaving $B_\ep(0)$. By a tail triviality argument (using the Brownian motions that determine the Loewner driving functions for $\beta_1^L$ and $\beta_2^L$), we get that in fact $\beta_1^L$ and $\beta_2^L$ a.s.\ cross infinitely many times. Consequently, a.s.\ there exist infinitely many open regions which lie to the left of $\beta_1^L$ and to the right of $\beta_2^L$. Since the ranges of $\beta_1^R$ and $\beta_2^R$ are closed sets of empty interior, it follows that a.s.\ each of these regions contains an open subset of $\mcl Q_0$. This concludes the proof of our claim.
}

\medskip

The joint law of $(\eta_1,\eta_2)$, viewed as curves modulo time parametrization, is invariant under scaling space. Hence the same is true of the law of $\mcl Q_0$. Therefore, a.s.\ $\mcl Q_0 \cap B_r(0)$ has non-empty interior for every $r>0$. 
Since a.s.\ $\mu_{\wt h}$ assigns positive mass to every open set, we get that a.s.\ $\mu_{\wt h}(\mcl Q_0 \cap B_r(0)) > 0$ for every $r > 0$. 

By the LQG coordinate change formula~\eqref{eqn-metric-measure-coord}, a.s.\ for each $r>0$,
\eqb \label{eqn-cone-coord}
\mu_{\wt h(r \cdot ) + Q\log r }(A  ) =   \mu_{\wt h}( r A) ,\quad \text{for each Borel set $A\subset \BB C$.}
\eqe 
Furthermore, by~\eqref{eqn-gff-translate}, the law of $\wt h$ is scale invariant modulo additive constant, i.e., for each $r > 0$ there is a random $C_r > 0$ such that $\wt h(r \cdot ) - C_r \eqD \wt h$. 
Adding a constant $C$ to $\wt h$ results in scaling $\mu_{\wt h}$ by $e^{\gamma C}$, so does not change the ratios of the $\mu_{\wt h}$-masses of any two sets. From this,~\eqref{eqn-cone-coord}, and the scale invariance of the law of $\mcl Q_0$, we get {
\alb
\frac{\mu_{\wt h}\left( \mcl Q_0 \cap  B_r(0)    \right) }{ \mu_{\wt h}(B_r(0)) }
&= \frac{\mu_{\wt h(r\cdot) + Q \log r }\left( (r^{-1}\mcl Q_0) \cap  B_1(0)    \right) }{ \mu_{\wt h(r\cdot) + Q \log r} (B_1(0)) } \qquad \text{(by~\eqref{eqn-cone-coord})} \\
&= \frac{e^{\gamma C_r} r^{\gamma Q} \mu_{\wt h(r\cdot)  -  C_r }\left( (r^{-1}\mcl Q_0) \cap  B_1(0)    \right) }{ e^{\gamma C_r} r^{\gamma Q} \mu_{\wt h(r\cdot) - C_r }(B_1(0)) } \quad \text{(by the scaling property of $\mu_{\wt h}$)} \\
&\eqD \frac{ \mu_{\wt h  }\left( \mcl Q_0 \cap  B_1(0)    \right) }{  \mu_{\wt h  }(B_1(0)) }   \quad \text{(by scale invariance in law)} .
\ale
}

Thus, the law of $\mu_{\wt h}\left( \mcl Q_0 \cap  B_r(0)    \right) / \mu_{\wt h}(B_r(0))$ does not depend on $r$. As we noted above, this random variable is a.s.\ positive. 
Consequently, for each $p\in (0,1)$ there exists $c > 0$ such that 
\eqb
\BB P\left[ \frac{\mu_{\wt h}\left( \mcl Q_0 \cap  B_r(0)    \right)}{\mu_{\wt h}(B_r(0))} \geq c \right] \geq p ,\quad \forall r > 0 .
\eqe
Hence, with probability at least $p$, there exist arbitrarily small values of $r > 0$ for which 
\eqbn
\frac{\mu_{\wt h}\left( \mcl Q_0 \cap  B_r(0)    \right) }{ \mu_{\wt h}(B_r(0)) } \geq  c .
\eqen
Since $p$ is arbitrary, the limsup in~\eqref{eqn-gff-limsup} is a.s.\ positive.
\end{proof}

\begin{proof}[Proof of Lemma~\ref{lem-sphere-limsup}]
Conditional on $(h,\eta_1,\eta_2)$, let $Z  $ be sampled from the probability measure $\mu_h$. It suffices to show that a.s.\ 
\eqb \label{eqn-limsup-uniform}
\limsup_{r \rta 0} \frac{\mu_h\left( \mcl Q_Z \cap  B_r(Z)    \right)}{\mu_h(B_r(Z))} > 0 .
\eqe 

{Let $W$ be another point sampled from $\mu_h$ in such a way that $Z,W$ are conditionally independent given $ h$.}
By Lemma~\ref{lem-sphere-pt}, the law of $(\BB C , h , \infty, Z,W )$ is that of a triply marked unit area quantum sphere. 
That is, if $\phi : \BB C\rta\BB C$ is the complex affine map with $\phi(0)= Z$ and $\phi(1) = W$, then the surface $(\BB C , h\circ\phi  +Q\log |\phi'| ,\infty, 0, 1  )$ has the law described in Definition~\ref{def-sphere}. 
By Lemma~\ref{lem-sphere-abs-cont}, the restriction of $h\circ\phi  +Q\log |\phi'|$ to $B_{1/2}(0)$ is absolutely continuous with respect to {$\wt h|_{B_{1/2}(0)}$, where $\wt h$ is a whole-plane GFF plus $\gamma\log|\cdot|^{-1}$ (as in Lemma~\ref{lem-gff-limsup})}, with both random distributions viewed modulo additive constant. We note that adding a constant to the field does not change any of the ratios in~\eqref{eqn-limsup-uniform}.   

The law of $(\eta_1,\eta_2)$, viewed modulo time parametrization, is invariant under scaling and spatial translation, and $(\eta_1,\eta_2)$ is independent from $(Z,W)$, hence from $\phi$. {Consequently, $\phi^{-1}(\mcl Q_Z) \eqD \mcl Q_0$ and $\phi^{-1}(\mcl Q_Z)$ is independent from $h$.
By combining this with the absolute continuity in the previous paragraph, we get that the joint law of $((h\circ\phi  +Q\log |\phi'| )|_{B_{1/2}(0)} , \phi^{-1}(\mcl Q_Z))$ is absolutely continuous with respect to the joint law of $(\wt h|_{B_{1/2}(0)} , \mcl Q_0)$, with the distributions viewed modulo additive constant. }
We can therefore apply Lemma~\ref{lem-gff-limsup} to get that a.s.\ 
\eqb \label{eqn-limsup-uniform0}
\limsup_{r \rta 0} \frac{\mu_{h\circ \phi + Q\log|\phi'|}\left( \phi^{-1}( \mcl Q_Z ) \cap B_r(0)  \right)}{\mu_{h\circ \phi + Q\log|\phi'|}(B_r(0))} > 0 .
\eqe 
Since $\phi$ is complex affine, $\phi(B_r(0)) = B_{|\phi'(0)| r}(Z)$. {By this and the LQG coordinate change formula~\eqref{eqn-metric-measure-coord}, 
\alb
\frac{\mu_{h\circ \phi + Q\log|\phi'|}\left( \phi^{-1}( \mcl Q_Z ) \cap B_r(0)  \right)}{\mu_{h\circ \phi + Q\log|\phi'|}(B_r(0))} 
= \frac{\mu_{h }\left(  \mcl Q_Z   \cap B_{|\phi'(0)| r}(Z)  \right)}{\mu_{h }(B_{|\phi'(0)| r}(Z))} ,\quad\forall r > 0 . 
\ale}
Thus,~\eqref{eqn-limsup-uniform0} implies~\eqref{eqn-limsup-uniform}.
\end{proof}

We now prove an SLE  version of the ``monotone sets have zero mass'' condition from Proposition~\ref{prop-monotone-set}.

\begin{prop} \label{prop-sle-ordered} 
Assume that we are in the setting of Theorem~\ref{thm-lis}. Almost surely, the following is true. Let $X\subset \BB C$ be a $\mu_h$-measurable set with the following property: for $\mu_h$-a.e.\ pair of points $z,w\in X$, the curves $\eta_1$ and $\eta_2$ hit $z$ and $w$ in the same order. 
Then $\mu_h(X) = 0$. 
\end{prop}
\begin{proof}
For $z\in \BB C$, let $\mcl Q_z$ be as in~\eqref{eqn-non-mono-set}. By Lemma~\ref{lem-sphere-limsup}, a.s.\  
\eqb \label{eqn-use-quadrant}
\limsup_{r \rta 0} \frac{\mu_h\left( \mcl Q_z \cap  B_r(z)    \right)}{\mu_h(B_r(z))} > 0 , \quad \text{for $\mu_h$-a.e.\ $z\in\BB C$}. 
\eqe
By the Lebesgue density theorem for general Radon measures on $\BB C$~\cite[Corollary 2.14]{mattila-book}, a.s.\ for every $\mu_h$-measurable set $X \subset \BB C$, 
\eqb \label{eqn-use-density}
\lim_{r\rta 0} \frac{\mu_h(B_r(z) \cap X)}{\mu_h(B_r(z))} = 1,\quad \text{for $\mu_h$-a.e.\ $z\in X$}. 
\eqe 
If $z$ and $w$ are two points which are hit in the same order by $\eta_1$ and $\eta_2$, and $z$ and $w$ are each hit at most once by each of $\eta_1$ and $\eta_2$, then the definition~\eqref{eqn-non-mono-set} of $\mcl Q_z$ shows that $w\notin \mcl Q_z$. From this observation and the fact that a.s.\ $\mu_h$-a.e.\ $z \in \BB C$ is hit at most once by each of $\eta_1$ and $\eta_2$, we see that if  $X$ is a $\mu_h$-measurable set with the property in the proposition statement, then $\mu_h(X \cap \mcl Q_z) = 0$ for $\mu_h$-a.e.\ $z\in X$. By combining this with~\eqref{eqn-use-quadrant}, we obtain that a.s., for every such set $X$, 
\eqb \label{eqn-no-density}
\limsup_{r\rta 0} \frac{\mu_h(B_r(z) \cap X)}{\mu_h(B_r(z))} < 1,\quad \text{for $\mu_h$-a.e.\ $z\in X$}. 
\eqe 
By~\eqref{eqn-use-density} and~\eqref{eqn-no-density}, we get that $\mu_h(X)  = 0$.
\end{proof}

\begin{prop} \label{prop-sle-mono-set} 
Assume that we are in the setting of Theorem~\ref{thm-lis}. Almost surely, for each monotone set $A\subset [0,1]^2$ (Definition~\ref{def-monotone-set}), we have $\perm(A) = 0$.
\end{prop}
\begin{proof} 
Let $\psi : [0,1] \rta [0,1]$ be a function as in the definition~\eqref{eqn-permuton-def} of $\perm$, so that $\eta_1(t) = \eta_2(\psi(t))$ for each $t\in [0,1]$. 
For a monotone set $A\subset [0,1]^2$, let 
\eqb \label{eqn-mono-set-projection}
 T_A = \left\{ t\in [0,1] : (t,\psi(t)) \in A  \right\}  .
\eqe
We claim that a.s.\ for every monotone set $A$, it holds that $\mu_h$-a.e.\ pair of points $z,w\in \eta_1(T_A)$ are hit in the same order by $\eta_1$ and $\eta_2$. 
Almost surely, the set of points in $\BB C$ which are hit more than once by either $\eta_1$ or $\eta_2$ has zero $\mu_h$-mass, so we can assume that each of $z$ and $w$ is hit exactly once by each of $\eta_1$ and $\eta_2$. 
Fix $z,w\in T_A$ and let $t_z\in T_A$ and $t_w \in T_A$ be the times when $\eta_1$ hits $z$ and $w$, respectively, and assume without loss of generality that $t_z < t_w$.  
By definition of $T_A$, $(t_z,\psi(t_z))$ and  $(t_w,\psi(t_w))$ belong to $A$ which is monotone. By Definition~\ref{def-monotone-set} of a monotone set and since $t_z<t_w$, it follows that $\psi(t_z) < \psi(t_w)$.
Since $\eta_2$ hits each of $z$ and $w$ only once, $\psi(t_z)$ and $\psi(t_w)$ are the unique times at which $\eta_2$ hits $z$ and $w$. Hence $z$ and $w$ are hit in the same order by $\eta_1$ and $\eta_2$.

By Proposition~\ref{prop-sle-ordered} and the preceding claim, it is a.s.\ the case that for every monotone set $A$, the $\mu_h$-mass of $\eta_1(T_A)$ is zero. Since $\eta_1$ is parametrized by $\mu_h$-mass, this implies that the one-dimensional Lebesgue measure of $T_A$ is zero. By the definition~\eqref{eqn-permuton-def} of $\perm$, this means that $\perm(A) = 0$. 
\end{proof}

\begin{proof}[Proof of Theorem~\ref{thm-lis}]
The statement for the longest increasing subsequence follows combining Propositions~\ref{prop-monotone-set} and~\ref{prop-sle-mono-set}. The statement for the longest decreasing subsequence can be obtained by the same argument. Alternatively, one can observe that if $\eta_1$ and $\eta_2$ are a pair of SLE curves as in Theorem~\ref{thm-lis}, then the pair consisting of $\eta_1$ and the time reversal of $\eta_2$ also satisfies the conditions of Theorem~\ref{thm-lis}.  
\end{proof}
 
{
\begin{remark}
The proof of Theorem~\ref{thm-lis} did not use many properties which are specific to SLE and LQG. 
Indeed, all we needed was some form of approximate scale invariance in a neighborhood of a point sampled from the LQG measure $\mu_h$ (in our setting, this comes from the local absolute continuity statements in the proof of Lemma~\ref{lem-sphere-limsup}) and the fact that for a point $Z$ sampled from $\mu_h$, the set $\mcl Q_Z$ a.s.\ has non-empty interior. 
Consequently, the arguments of this section could in principle be extended to prove analogs of Theorem~\ref{thm-lis} for other permutons constructed from pairs of random space-filling curves as in Section~\ref{sec-permuton-def}. 
However, we do not have in mind any important examples of such permutons where the curves are not SLE curves parametrized by LQG mass, so we do not pursue this point further here.
\end{remark}
}

\section{Dimension of the support is one}
\label{sec-dim-proof}

In this section we will prove Theorem~\ref{thm-permuton-dim}, which asserts that the dimension of the closed support $\op{supp} \perm$ is one for \emph{any} of the permutons defined in Section~\ref{sec-permuton-def}. We recall from Theorem~\ref{thm-permuton-dim} that $\mcl T = \{(t,s) \in [0,1]^2 : \eta_1(t) = \eta_2(s)\}$ .

Since $\perm$ is a permuton, by definition it is immediate that the projection of $\op{supp} \perm$ onto the first coordinate of $[0,1]^2$ is all of $[0,1]$.  Since the projection map is 1-Lipschitz, this implies that the Hausdorff dimension of $\op{supp} \perm$ is at least one. From this and Lemma~\ref{lem-permuton-inclusion}, we get that a.s.\ 
\eqb \label{eqn-permuton-dim-trivial}
1 \leq \dim_{\mcl H} \op{supp} \perm \leq \dim_{\mcl H} \mcl T,
\eqe
where $\dim_{\mcl H}$ denotes the Hausdorff dimension of a set. 
It therefore suffices to show that a.s.\ $\dim_{\mcl H} \mcl T \leq 1$.

\subsection{Core argument}
\label{sec-permuton-dim-core}

By the definition of $\mcl T$, if $[a,b]\times [c,d] \subset [0,1]^2$ is a rectangle, then $\mcl T$ intersects $[a,b]\times[c,d]$ if and only if $\eta_1([a,b]) \cap \eta_2( [c,d] ) \not=\emptyset $. We therefore seek an upper bound for the probability that the images of two small time intervals under $\eta_1$ and $\eta_2$ intersect. The following lemma is the main input in the proof of Theorem~\ref{thm-permuton-dim}.

\begin{lem} \label{lem-sle-intersect-prob}
Assume that we are in the setting of Theorem~\ref{thm-permuton-dim}. 
Fix $\zeta\in (0,1)$ and $\delta \in (0,1/2)$. 
There exist $E_\ep = E_\ep(\zeta,\delta) \in \sigma(h,\eta_1,\eta_2)$ for $\ep > 0$ such that $\lim_{\ep\rta 0} \BB P[E_\ep] = 1$ and the following is true.
Let $S\in [\delta,1-\delta]$ be sampled uniformly from Lebesgue measure on $[\delta,1-\delta]$, independently from everything else.
For each $t\in [\delta,1-\delta]$ and each small enough $\ep > 0$ (depending on $\delta$),  
\eqb \label{eqn-sle-intersect-prob}
\BB P\left[E_\ep ,\, \eta_1([t,t+\ep]) \cap \eta_2([S,S+\ep]) \not=\emptyset\right] \leq \ep^{1 - \zeta }.
\eqe
\end{lem}

When we apply Lemma~\ref{lem-sle-intersect-prob}, we will take $\zeta$ and $\delta$ to be small. Note that the statement of the lemma would not be true if we took $S$ to be a fixed deterministic time instead of a random time. The reason is that the coupling of $\eta_1$ and $\eta_2$, viewed as curves modulo time parametrization, is arbitrary. For example, we allow $\eta_1=\eta_2$ in which case the probability that $\eta_1([t,t+\ep]) \cap \eta_2([t,t+\ep]) \not=\emptyset $ is 1. 

The idea of the proof of Lemma~\ref{lem-sle-intersect-prob} is as follows. Since $\eta_1$ and $\eta_2$ are parametrized by LQG mass, the set $\eta([t,t+\ep])$ has $\mu_h$-mass $\ep$ and the conditional law of $\eta_2(S)$ given $(h,\eta_1,\eta_2)$ is that of a sample from the restriction of $(1-2\delta)^{-1} \mu_h$ to $\eta_2([\delta,1-\delta])$. Therefore, the probability that $\eta_2(S) \in \eta_1([t,t+\ep])$ is at most of order $\ep$. To deduce an upper bound for the probability in~\eqref{eqn-sle-intersect-prob}, we will improve this estimate by showing that it is unlikely for $\eta_2(S)$ to be ``close'' to $\eta_1([t,t+\ep])$ and that it is unlikely for $\eta_2([S,S+\ep])$ to be ``large''. 

It is convenient to measure ``close'' and ``large'' in terms of the $\gamma$-LQG metric $D_h$ associated with $h$. This is because a space-filling SLE curve parametrized by LQG mass behaves nicely with respect to this metric (see Lemma~\ref{lem-sphere-holder} below). In what follows, we let $d_\gamma$ be the Hausdorff dimension of $\BB C$ with respect to the metric $D_h$, which is a.s.\ equal to a finite deterministic constant~\cite[Corollary 1.7]{gp-kpz}. 
We will need the following two lemmas.

\begin{lem} \label{lem-sphere-holder}
Let $\gamma \in (0,2)$ and $\kappa>4$. 
Let $(\BB C , h , \infty)$ be a singly marked unit area $\gamma$-Liouville quantum sphere. Let $\eta : [0,1]\rta \BB C$ be a whole-plane space-filling SLE$_\kappa$ from $\infty$ to $\infty$ sampled independently from $h$, then parametrized by $\gamma$-LQG mass with respect to $h$. 
Also fix $\zeta\in (0,1)$ and $\delta \in (0,1/2)$. Almost surely, for each small enough $\ep > 0$ (how small is random),  
\eqb \label{eqn-sphere-holder}
D_h(\eta(s) , \eta(t)) \leq \ep^{1/d_\gamma - \zeta} ,\quad\forall s , t \in [\delta,1-\delta] \quad \text{such that} \quad |s-t| \leq \ep .
\eqe
\end{lem}

For the next lemma, we introduce the notation
\eqb \label{eqn-lqg-ball-def}
\mcl B_r(z;D_h) := \left( \text{$D_h$-ball of radius $r$ centered at $z$} \right) , \quad \forall z\in\BB C\cup\{\infty\} ,\quad\forall r > 0 .
\eqe
 
\begin{lem} \label{lem-sphere-ball}
Let $(\BB C ,h , \infty)$ be a singly marked unit area $\gamma$-Liouville quantum sphere. Also fix $\zeta\in (0,1)$. Almost surely, for each small enough $r > 0$ (how small is random), 
\eqb
r^{d_\gamma +  \zeta} \leq \mu_h\left(\mcl B_r(z;D_h) \right) \leq r^{d_\gamma -\zeta} , \quad\forall z \in \BB C \cup \{\infty\} .
\eqe
\end{lem}

Lemmas~\ref{lem-sphere-holder} and~\ref{lem-sphere-ball} are straightforward consequences of known estimates for the LQG metric~\cite{afs-metric-ball,gs-lqg-minkowski}, but some work is needed to transfer from the case of a whole-plane GFF (which is the case treated by the results we cite) to the case of a unit area quantum sphere. We postpone the proofs until after the proof of Theorem~\ref{thm-permuton-dim}.

\begin{proof}[Proof of Lemma~\ref{lem-sle-intersect-prob}]
Fix $\wt\zeta\in (0,\zeta)$ to be chosen later, in a deterministic manner depending only on $\gamma$ and $\zeta$. 
Let $E_\ep = E_\ep(\wt \zeta,\delta)$ be the event that the following is true:
\begin{enumerate}[$(i)$]
\item For each $i\in \{1,2\}$, we have $D_h(\eta_i(s) , \eta_i(t)) \leq \ep^{1/d_\gamma - \wt\zeta}$ for each $ s , t \in [\delta/2 ,1-\delta/2] $ such that $|s-t| \leq \ep$; \label{item-intersect-holder}
\item $\mu_h\left(\mcl B_r(z;D_h) \right) \leq r^{d_\gamma -\wt\zeta}$ for each $r \in (0,  2 \ep^{1/d_\gamma-\wt\zeta}]$.  \label{item-intersect-ball}
\end{enumerate}
By Lemma~\ref{lem-sphere-holder} (applied once to each of $\eta_1$ and $\eta_2$) and Lemma~\ref{lem-sphere-ball}, we have $\BB P[E_\ep] \rta 1$ as $\ep\rta 0$. 

We henceforth assume that $\ep < \delta/2$. Let $t$ and $S$ be as in the lemma statement. If $E_\ep$ occurs, then by condition~\eqref{item-intersect-holder} in the definition of $E_\ep$, 
\eqb \label{eqn-sle-in-ball}
\eta_1([t,t+\ep]) \subset \mcl B_{\ep^{1/d_\gamma - \wt\zeta}}(\eta_1(t) ; D_h)  \quad \text{and} \quad \eta_2([S,S+\ep]) \subset \mcl B_{\ep^{1/d_\gamma - \wt\zeta}}(\eta_2(S) ; D_h) .
\eqe
By condition~\eqref{item-intersect-ball} in the definition of $E_\ep$, on this event the $\mu_h$-mass of the larger ball $\mcl B_{2\ep^{1/d_\gamma - \wt\zeta}}(\eta_1(t) ; D_h)$ is at most a $(\gamma,\wt\zeta)$-dependent constant times $\ep^{(d_\gamma -\wt\zeta) (1 / d_\gamma - \wt\zeta) }$. Since $S$ is sampled from Lebesgue measure on $[\delta,1-\delta]$, independently from everything else, and $\eta_2$ is parametrized by $\mu_h$-mass, the conditional law of $\eta_2(S)$ given $(h,\eta_1,\eta_2)$ is given by the restriction of $(1-2\delta)^{-1} \mu_h$ to $\eta_2([\delta,1-\delta])$. Therefore,
\eqb \label{eqn-use-ball-mass}
\BB P\left[ E_\ep ,\, \eta_2(S) \in \mcl B_{2\ep^{1/d_\gamma - \wt\zeta}}(\eta_1(t) ; D_h) \right] 
\preceq \BB E\left[ \BB 1_{E_\ep} \mu_h\left( \mcl B_{2\ep^{1/d_\gamma - \wt\zeta}}(\eta_1(t) ; D_h) \right) \right]
\preceq \ep^{(d_\gamma -\wt\zeta) (1 / d_\gamma - \wt\zeta) } ,
\eqe
where $\preceq$ denotes inequality up to a finite constant which depends only on $\gamma , \wt\zeta , \delta$. 

By~\eqref{eqn-sle-in-ball} and the triangle inequality for $D_h$, if $E_\ep$ occurs and $\eta_2(S) \notin \mcl B_{2\ep^{1/d_\gamma - \wt\zeta}}(\eta_1(t) ; D_h)$, then $\eta_1([t,t+\ep])\cap \eta_2([S,S+\ep]) = \emptyset$. Therefore,~\eqref{eqn-use-ball-mass} implies that 
\eqbn
\BB P\left[E_\ep ,\, \eta_1([t,t+\ep]) \cap \eta_2([S,S+\ep]) \not=\emptyset\right] \preceq \ep^{(d_\gamma -\wt\zeta) (1 / d_\gamma - \wt\zeta) } .
\eqen
We now obtain~\eqref{eqn-sle-intersect-prob} by choosing $\wt\zeta > 0$ to be small enough that $(d_\gamma-\wt\zeta)(1/d_\gamma - \wt\zeta) < 1-\zeta$. 
\end{proof}

\begin{proof}[Proof of Theorem~\ref{thm-permuton-dim}] 
By~\eqref{eqn-permuton-dim-trivial}, we only need to show that a.s.\ $\dim_{\mcl H}\mcl T\leq 1$. 
Fix $\zeta\in (0,1)$ and $\delta \in (0,1/2)$. By the countable stability of Hausdorff dimension, it suffices to show that a.s.\ 
\eqb \label{eqn-permuton-dim-show}
\dim_{\mcl H} \left( \mcl T \cap [\delta,1-\delta]^2  \right)  \leq 1 +\zeta .
\eqe
To this end, let $E_\ep$ be the event of Lemma~\ref{lem-sle-intersect-prob} with $\zeta/2$ in place of $\zeta$, and recall that $\BB P[E_\ep] \rta 1$ as $\ep\rta 0$. We will upper-bound the expected number of $\ep\times \ep$ squares needed to cover the set in~\eqref{eqn-permuton-dim-show} on the event $E_\ep$.
 
Since $\mcl T = \{(t,s) \in [0,1]^2 : \eta_1(t) = \eta_2(s)\}$, for $x , y \in [0,1-\ep]^2$, the set $\mcl T$ intersects $[x,x+\ep] \times [y,y+\ep]$ if and only if $\eta_1([x , x+\ep]) \cap \eta_2([y,y+\ep]) \not=\emptyset$. Therefore, the expected number of $\ep\times \ep$ squares needed to cover $\mcl T \cap [\delta,1-\delta]^2$, truncated on the event $E_\ep$, is at most
\eqb \label{eqn-permuton-dim-sum}
\BB E\left[  \BB 1_{E_\ep} \sum_{x \in [\delta - \ep ,1-\delta ] \cap \ep\BB Z}  \sum_{y \in [\delta - \ep ,1-\delta ] \cap \ep\BB Z} \BB 1_{\eta_1([x , x+\ep]) \cap \eta_2([y,y+\ep]) \not=\emptyset} \right] .
\eqe

For each $x,y\in  [\delta - \ep ,1-\delta ] \cap \ep\BB Z$, the square $[x,x+\ep] \times [y,y+\ep]$ is contained in $[t,t+2\ep]\times[s,s+2\ep]$ for each $t\in [x-\ep,x]$ and each $s\in [y-\ep,y]$. Therefore, the quantity~\eqref{eqn-permuton-dim-sum} is bounded above by
\allb \label{eqn-permuton-dim-int}
&\ep^{-2} \BB E\left[ \BB 1_{E_\ep}  \int_{\delta-2\ep}^{1-\delta} \int_{\delta-2\ep}^{1-\delta}  \BB 1_{\eta_1([t, t+2 \ep]) \cap \eta_2([s,s+2\ep]) \not=\emptyset}  \,ds\,dt     \right]  \notag\\
&\qquad\qquad =   \ep^{-2} \BB P\left[ E_\ep , \, \eta_1([T, T+2 \ep]) \cap \eta_2([S,S+2\ep]) \not=\emptyset \right]  ,
\alle
where $S$ and $T$ are sampled uniformly and independently from $[\delta-2\ep,1-\delta]$, independently from everything else. 

By Lemma~\ref{lem-sle-intersect-prob} (with $\zeta/2$ in place of $\zeta$), for each small enough $\ep > 0$ the right side of~\eqref{eqn-permuton-dim-int} is at most $\ep^{-1-\zeta/2}$. Hence the expected number of $\ep\times\ep$ squares needed to cover $\mcl T \cap [\delta,1-\delta]^2$ on the event $E_\ep$ is at most $\ep^{-1-\zeta/2}$. 
{From this, Markov's inequality, and the fact that $\lim_{\ep\to 0} \BB P[E_\ep] = 1$, we obtain that a.s.\ the lower Minkowski dimension of $  \mcl T \cap [\delta,1-\delta]^2  $ is at most $1+\zeta$. This implies~\eqref{eqn-permuton-dim-show}. Note that our proof does not show that the Minkowski dimension of $\mcl T$ is at most 1 since Minkowski dimension is not countably stable.} 
\end{proof}
%\BB P[N_\ep \BB 1_{E_\ep} > \ep^{-1-\zeta}] \leq \ep^{\zeta/2}$, so $N_\ep \BB 1_{E_\ep} \rta 0$ in probability as $\ep\rta 0$. Therefore, $N_\ep\rta 0$ in probability as $\ep\rta 0$.  

{
\begin{remark} \label{remark-permuton-rectangle}
The dimension upper bound of Theorem~\ref{thm-permuton-dim} is not true in general for pairs of permutons constructed from space-filling curves. Indeed, for each $d \in (1,2)$, there exists a pair of space-filling curves $\eta_1 , \eta_2 : [0,1] \to [0,1]^2$ in the unit square, parametrized by Lebesgue measure, such that the following is true. If we let $\perm$ be the permuton associated with $(\eta_1,\eta_2)$ as in~\eqref{eqn-permuton-def}, then the Hausdorff dimension of $\op{supp}\perm$ is at least $d$. 

To see this, let $N \geq M \geq 3$ be odd integers. Via a straightforward iterative procedure, one can construct a space-filling curve $\eta_1 : [0,1]  \to [0,1]^2$ such that for each integer $k\geq 1$ and each $i \in \{1,\dots,(NM)^k\}$, the set $\eta_1([(i-1) (NM)^{-k} , i (NM)^{-k}])$ is an $N^{-k} \times M^{-k}$ rectangle with corners in $N^{-k} \BB Z \times M^{-k} \BB Z$. Moreover, $\eta_1$ hits Lebesgue-a.e.\ each point of $[0,1]^2$ exactly once. One can similarly construct $\eta_2$ so that it satisfies the same conditions but with $N$ and $M$ interchanged. Then for each $i,j \in \{1,\dots, (NM)^k\}$, the Lebesgue measure of $\eta_1([(i-1) (NM)^{-k} , i (NM)^{-k}]) \cap \eta_2 ([(j-1) (NM)^{-k} , j (NM)^{-k}])  $ is equal to either $N^{-2k}$ or zero. By Lemma~\ref{lem-permuton-defined} (which is valid in our setting, with the same proof),
\eqb  \label{eqn-perm-rectangle}
\perm\left( [(i-1) (NM)^{-k} , i (NM)^{-k}]) \times [(j-1) (NM)^{-k} , j (NM)^{-k}]  \right) \leq N^{-2k} .
\eqe 
To lower-bound the Hausdorff dimension of $\op{supp} \perm $, consider without loss of generality a covering $\{S_n \}_{n \geq 1}$ of $\op{supp} \perm $ by squares $S_n = [(i_n-1) (NM)^{-k_n} , i_n (NM)^{-k_n}] \times [(j_n -1)(NM)^{-k_n} , j_n (NM)^{-k_n}]$ for $k_n \in \BB N$ and $i_n , j_n \in \{1,\dots,(NM)^{k_n}\}$. Since the total mass of $\perm$ is 1, it follow from~\eqref{eqn-perm-rectangle} that
\eqbn
1\leq \sum_{n=1}^\infty N^{-2k_n} = \sum_{n=1}^\infty   2^{-\Delta/2} [\op{diam}(S_n)]^\Delta  ,\quad \text{where} \quad 
\Delta = \frac{\log N^2 }{ \log (NM )} .
\eqen 
By making $N$ sufficiently large ($M$ fixed), we can make $\Delta$ as close to 2 as we like. 
\end{remark}
}

\subsection{Proofs of LQG metric estimates}
\label{sec-lqg-metric-proofs}

To conclude the proof of Theorem~\ref{thm-permuton-dim}, it remains to prove our estimates for the LQG metric, namely Lemmas~\ref{lem-sphere-holder} and~\ref{lem-sphere-ball}. We will prove these estimates by applying known results for the whole-plane GFF, then using local absolute continuity, in the form of Lemma~\ref{lem-sphere-abs-cont}, to transfer these results to a singly marked unit-area $\gamma$-Liouville quantum sphere $(\BB C,h,\infty)$.

By the LQG coordinate change formula~\eqref{eqn-metric-measure-coord} for $\mu_h$ and for $D_h$, the statements of both lemmas do not depend on the particular choice of embedding $h$ (recall Definition~\ref{def-lqg-surface}). Hence, we can assume without loss of generality that our embedding is chosen in such a way that $(\BB C ,h , \infty , 0, 1  )$ is a triply marked unit area quantum sphere. Due to Lemma~\ref{lem-sphere-pt}, such an embedding can be obtained by sampling points $Z,W\in\BB C$ from $\mu_h$, then applying a complex affine map which takes $Z$ to 0 and $W$ to 1.  

Let
\eqb \label{eqn-new-gff}
\wt h = h^{\BB C} - \gamma\log |\cdot| ,
\eqe
where $h^{\BB C}$ is a whole-plane GFF normalized so that its average over the unit circle is zero. For $R > 0$, we define the open set
\eqb \label{eqn-annular-region}
U_R := B_R(0) \setminus \ol{B_{1/R}(1)} . 
\eqe 
By Lemma~\ref{lem-sphere-abs-cont},
\allb\label{eqn-abs-cont}
&\text{For each $R> 0$, the laws of $h|_{U_R}$ and $\wt h|_{U_R}$, viewed modulo additive constant,} \notag \\
&\text{are mutually absolutely continuous.} 
\alle  

\begin{proof}[Proof of Lemma~\ref{lem-sphere-holder}]

\noindent\textit{Step 1: estimate for the whole-plane GFF.} Fix $\zeta\in (0,1)$. 	
With $\wt h$ as in \eqref{eqn-new-gff}, let $\wt\eta$ be a whole-plane space-filling SLE$_\kappa$ from $\infty$ to $\infty$ sampled independently from $\wt h$ and then parametrized by $\gamma$-LQG mass with respect to $\wt h$. 
It is shown in~\cite[Theorem 1.4]{gs-lqg-minkowski} that\footnote{
Strictly speaking, \cite[Theorem 1.4]{gs-lqg-minkowski} considers the circle average embedding $h^\gamma$ of a $\gamma$-quantum cone instead of the field $\wt h$ of~\eqref{eqn-new-gff}. However, the statement for $\wt h$ can easily be deduced from the statement for $h^\gamma$ using the fact that the restrictions of $\wt h$ and $h^\gamma$ to the unit disk agree in law (see~\cite[Definition 4.10]{wedges} and the discussion just after) along with the scaling property of the whole-plane GFF~\eqref{eqn-gff-translate}.
}
an analog of the lemma statement holds with $\wt h$ in place of $h$. Namely, a.s.\ for each $R>0$ it holds for each small enough $\ep > 0$ (how small is random and depends on $R$ and $\zeta$) that
\eqb \label{eqn-cone-holder}
D_{\wt h}(\wt \eta(s) , \wt\eta(t)) \leq \ep^{1/d_\gamma - \zeta} ,
\quad\forall s,t \in \BB R 
\quad \text{such that}   \quad |s-t| \leq \ep 
\quad\text{and} \quad \wt\eta(s) , \wt\eta(t) \in B_{2R}(0). 
\eqe

\medskip

\noindent\textit{Step 2: absolute continuity.}
We will now deduce the analogue result as in~\eqref{eqn-cone-holder} when $\wt h$ is replaced with $h$ using an absolute continuity argument. 
Recall the internal metric $D_{\wt h}(\cdot,\cdot ; V)$ for $V\subset\BB C$ open which is defined in~\eqref{eqn-internal-metric}.
As explained just after~\eqref{eqn-internal-metric}, this metric is a measurable function of $\wt h|_V$.

If $\ep > 0$ is small enough that $\ep^{1/d_\gamma - \zeta}  < D_{\wt h}\left( U_R , \bdy U_{2 R} \right)$, then $D_{\wt h}(z,w) = D_{\wt h}(z,w ; U_{2R})$ for each $z,w\in U_R$ with $|z-w| \leq \ep^{1/d_\gamma-\zeta}$. The locality property of $D_{\wt h}$, as explained in the preceding paragraph, therefore implies that for such a choice of $\ep$, the restricted distance function 
\eqbn
\left(D_{\wt h}(z,w) : z,w\in  U_{ R} ,\: D_{\wt h}(z,w) \leq \ep^{1/d_\gamma-\zeta} \right)
\eqen
is a.s.\ determined by $\wt h|_{U_{2R}}$. Therefore, for small enough $\ep > 0$, the event that~\eqref{eqn-cone-holder} holds is determined by $\wt h|_{U_{2R}}$ and $\wt\eta$, viewed modulo time parametrization. 

By~\eqref{eqn-abs-cont}, there are random $C , \wt C  \in \BB R$ (depending on $R$ but not on $\ep$) such that the laws of $(h + C)|_{U_{2R}}$ and $(\wt h + \wt C)|_{U_{2R}}$ are mutually absolutely continuous. Adding a constant to $\wt h$ scales LQG areas and distances by a constant factor. Since $\eta$ and $\wt\eta$ agree in law modulo time parametrization, combining these facts with the conclusion of the preceding paragraph shows that~\eqref{eqn-cone-holder} also holds for $h$, i.e., a.s.\ for each $R>0$ it holds for each small enough $\ep  > 0$ that
\eqb \label{eqn-sphere-holder0}
D_{  h}( \eta(s) , \eta(t)) \leq \ep^{1/d_\gamma - \zeta} ,
\quad\forall s,t \in [0,1] \quad \text{such that} \quad |s-t| \leq \ep 
\quad\text{and}  \quad  \eta(s) ,  \eta(t) \in U_R . 
\eqe

\medskip

\noindent\textit{Step 3: transferring to $\eta([\delta,1-\delta])$.}
Let $T$ be the (a.s.\ unique) time at which $\eta$ hits 1. Almost surely, $\eta([T-\delta,T+\delta])$ contains a neighborhood of 1 and $\eta([\delta,1-\delta])$ is a compact subset of $\BB C$. Therefore, a.s.\ there exists $R > 0$ such that 
\eqbn
\eta\left( [\delta,1-\delta] \setminus [T-\delta,T+\delta] \right) \subset U_R .
\eqen
From this and~\eqref{eqn-sphere-holder0}, we obtain the following restricted version of the lemma statement. Almost surely, for each small enough $\ep > 0$, 
\eqb \label{eqn-sphere-holder1}
D_h(\eta(s) , \eta(t)) \leq \ep^{1/d_\gamma - \zeta} ,\quad\forall s , t \in [\delta,1-\delta] \setminus [T-\delta,T+\delta] \quad \text{such that} \quad |s-t| \leq \ep .
\eqe
To remove the restriction that $s,t\notin [T-\delta,T+\delta]$, we re-sample the marked point at 1. Let $S$ be sampled uniformly from $[0,1]$, independently from everything else. Since $\eta$ is parametrized by $\mu_h$-mass, the conditional law of $\eta(S)$ given $(h,\eta)$ is that of a sample from $\mu_h$. By Lemma~\ref{lem-sphere-pt}, the triply marked quantum surface $(\BB C ,h , \infty,  0, \eta(S)  )$ is a triply marked unit area quantum sphere.
Hence, if $\phi : \BB C\rta \BB C$ is the complex affine map which fixes zero and satisfies $\phi(1) = \eta(S)$, then $( h\circ \phi + Q\log|\phi'| , \phi^{-1} \circ \eta) \eqD (h,\eta)$. 
By~\eqref{eqn-sphere-holder1} with $( h\circ \phi + Q\log|\phi'| , \phi^{-1} \circ \eta)$ instead of $(h,\eta)$ together with the LQG coordinate change formula for $\mu_h$ and $D_h$~\eqref{eqn-metric-measure-coord}, we get that~\eqref{eqn-sphere-holder1} also holds with $S$ instead of $T$. Since $S$ is uniformly distributed on $[0,1]$ and is independent from $h$ and $\eta$, the version of~\eqref{eqn-sphere-holder1} with $S$ instead of $T$ implies the lemma statement. 
\end{proof}
 
\begin{remark}
At the end of the proof of Lemma~\ref{lem-sphere-holder}, we could have instead argued that the conditional law of $T$ given the curve-decorated quantum surface $(\BB C ,h , \infty , 0, \eta)$ is uniform on $[0,1]$, but we found it easier to explain the argument by introducing $S$. 
\end{remark}

\begin{proof}[Proof of Lemma~\ref{lem-sphere-ball}]
By~\cite[Theorem 1.1]{afs-metric-ball},\footnote{The statement of~\cite[Theorem 1.1]{afs-metric-ball} is given for a whole-plane GFF without a log singularity. However, the proof still works for a whole-plane GFF plus $\gamma \log |\cdot|^{-1}$. Alternatively, the statement for a whole-plane GFF plus $\gamma \log |\cdot|^{-1}$ can be deduced by applying the statement for a whole-plane GFF in logarithmically many dyadic annuli centered at the origin and taking a union bound.} 
we have the following analog of the lemma statement with the field $\wt h$ from \eqref{eqn-new-gff} in place of $h$. For each $R > 0$, it is a.s.\ the case that for each small enough $r  > 0$, 
\eqb \label{eqn-cone-ball}
r^{d_\gamma + \zeta} \leq \mu_{\wt h}\left(\mcl B_r(z;D_h) \right) \leq r^{d_\gamma -\zeta} , \quad\forall z \in B_R(0) .
\eqe 
By a similar absolute continuity argument as in the proof of Lemma~\ref{lem-sphere-holder}, we deduce from~\eqref{eqn-cone-ball} that for each $R>0$, it is a.s.\ the case that for each small enough $r > 0$, the estimate~\eqref{eqn-cone-ball} holds with $h$ in place of $\wt h$. 

The law of the triply marked quantum sphere $(\BB C ,h , \infty , 0,1 )$ is invariant under interchanging the marked points at 0 and $\infty$ (this follows, e.g., from Lemma~\ref{lem-sphere-pt}). Hence, $h(1/\cdot) \eqD h$. We can therefore deduce from~\eqref{eqn-cone-ball} with $h$ in place of $\wt h$ that for each $R>0$, a.s.\ for each small enough $r>0$, 
\eqb \label{eqn-sphere-ball-invert}
r^{d_\gamma + \zeta} \leq \mu_h\left(\mcl B_r(z;D_h) \right) \leq r^{d_\gamma -\zeta} , \quad\forall z \in \BB C\setminus B_{1/R}(0) .
\eqe 
Combining~\eqref{eqn-cone-ball} (with $h$ in place of $\wt h$) and~\eqref{eqn-sphere-ball-invert} gives the lemma statement.  
\end{proof}

\section{Re-rooting invariance}
\label{sec-re-root-proof}

In this section we will prove the re-rooting invariance property for permutons constructed from $\gamma$-LQG with $\gamma\in (0,2)$ decorated by two independent SLE$_8$ curves, as given by Theorem~\ref{thm-permuton-re-root}.

\subsection{Re-rooting invariance for SLE$_8$}
\label{sec-sle-re-root}

The main input in the proof of Theorem~\ref{thm-permuton-re-root} is a re-rooting property for SLE$_8$. To state this property, we make the following definition.  
 
\begin{defn}[Re-rooting a curve] \label{defn-curve-re-root}
Let $P : [0,1] \rta \BB C \cup \{\infty\}$ be a curve in the Riemann sphere with $P(0) = P(1)$ and $P([0,1]) = \BB C\cup \{\infty\}$. 
For $z\in \BB C$, let $\tau_z$ be the first time at which $P$ hits $z$. 
We define the \textbf{re-rooting} of $P$ at $z$ to be the curve $P^z : [0,1]\rta \BB C\cup\{\infty\}$ obtained by concatenating $P|_{[\tau_z, 1]}$ followed by $P|_{[0,\tau_z]}$. That is,
\eqb  \label{eqn-curve-re-root}
P^z(t) := \begin{cases}
P(t + \tau_z) ,\quad &\text{if $t\in [0, 1-\tau_z]$} ,\\
P(t - (1-\tau_z) ) ,\quad &\text{if $t \in [ 1-\tau_z ,1]$}  .
\end{cases}
\eqe 
\end{defn}

If we think of $P$ as a space-filling loop based at the point $P(0) = P(1)$, then $P^z$ is the same loop with the base point taken to be $z$ instead. 
The following lemma tells us that whole-plane space-filling SLE$_\kappa$ is re-rooting invariant if and only if $\kappa=8$.

\begin{prop}[Re-rooting invariance for SLE$_8$] \label{prop-sle-re-root}
Let $\kappa > 4$ and let $\eta$ be a whole-plane space-filling SLE$_\kappa$ from $ \infty$ to $\infty$. Assume $\eta$ is parametrized by the interval $[0,1]$, so that $\eta: [0,1] \rta \BB C \cup \{\infty\}$. 
If $\kappa = 8$, then for any $z\in\BB C$ the re-rooted curve $\eta^z$ has the law of a whole-plane space-filling SLE$_\kappa$ from $z$ to $z$, viewed modulo time parametrization. Equivalently, if $\phi^z$ is a fractional linear transformation taking $z$ to $\infty$ (i.e., $\phi^z(w) = a/(z-w)$ for some non-zero $a\in\BB C$), then $\phi^z\circ \eta^z$ has the same law as $\eta$, viewed modulo time parametrization. 
If $\kappa \not=8$, then this property is not true for any $z\in\BB C$. 
\end{prop}

As our proof will show, the reason why the symmetry property of Proposition~\ref{prop-sle-re-root} holds only for $\kappa=8$ is that chordal space-filling SLE$_\kappa$ is reversible if and only if $\kappa=8$~\cite[Theorem 1.19]{ig4}.

\begin{proof}[Proof of Proposition~\ref{prop-sle-re-root}]
By the construction of space-filling SLE from flow lines of the whole-plane Gaussian free field (Section~\ref{sec-sle}), together with the translation invariance of the law of the whole-plane GFF, viewed modulo additive constant~\eqref{eqn-gff-translate}, the law of $\eta$, viewed modulo time parametrization, is invariant under spatial translation. We can therefore assume without loss of generality that $z=0$. 

In~\cite[Section 1.2.3]{ig4} the authors construct a chordal version of space-filling SLE$_\kappa$ between two marked boundary points of a simply connected domain via a similar construction to the one described in Section~\ref{sec-sle}.  
As explained in~\cite[Footnote 4]{wedges}, one can construct the whole-plane space-filling SLE$_\kappa$ from the chordal version as follows. See Figure~\ref{fig-space-filling-def} for an illustration.  
\begin{itemize}
\item Sample a pair of whole-plane SLE$_{16/\kappa}(2-16/\kappa)$ curves $\beta^L ,\beta^R$ from 0 to $\infty$, coupled together so that they can be realized as the flow lines of angle $\pi/2$ and $-\pi/2$ for the whole-plane GFF. Equivalently, first sample a whole-plane SLE$_{16/\kappa}(2-16/\kappa)$ curve $\beta^L$ from 0 to $\infty$; then, conditional on $\beta^L$, sample a chordal SLE$_{16/\kappa}(-8/\kappa ; -8/\kappa)$ from 0 to $\infty$ in the complement of the first curve.\footnote{
Whole-plane SLE$_{16/\kappa}(\rho)$ and chordal SLE$_\kappa(\rho^L;\rho^R)$ for $\rho,\rho^L,\rho^R >-2$ are variants of SLE$_{16/\kappa}$. We will not need the precise definitions of these processes here. We refer to~\cite[Section 2.2]{ig4} and~\cite[Section 2.2]{ig1}, respectively, for more background on these processes.
}
\item The curves $\beta^L$ and $\beta^R$ divide space into two open domains: $U_-$, which lies to the left and $\beta^L$ and to the right of $\beta^R$, and $U_+$ which lies to the right of $\beta^L$ and to the left of $\beta^R$. 
\item If $\kappa\geq 8$, since $\beta^L$ and $\beta^R$ do not hit each other except at their starting point, the regions $U_-$ and $U_+$ are simply connected. Conditional on $\beta^L$ and $\beta^R$, let $\eta_-$ (resp.\ $\eta_+$) be a chordal SLE$_\kappa$ from 0 to $\infty$ in $U_-$ (resp.\ $U_+$). Let $\eta$ be the concatenation of the time reversal of $\eta_-$ followed by $\eta_+$. 
\item If $\kappa\in (4,8)$, then $\beta^L$ and $\beta^R$ intersect in an uncountable totally disconnected set which is hit in the same order by $\beta^L$ and $\beta^R$. In this case, each of $U_-$ and $U_+$ consists of a countable ordered string of connected components (``beads''), each of which has one boundary arc which is part of $\beta^L$ and one boundary arc which is part of $\beta^R$. Each bead has two marked boundary points, where these two arcs meet. Conditional on $\beta^L$ and $\beta^R$, let $\eta_-$ be a curve from 0 to $\infty$ in $\ol U_-$ which consists of the concatenation of conditionally independent chordal space-filling SLE$_\kappa$ curves between the marked points of the beads of $U_-$. Similarly define $\eta_+$ in $\ol U_+$. Then, let $\eta$ be the concatenation of the time reversal of $\eta_-$ followed by $\eta_+$. 
\end{itemize}
 
\begin{figure}[ht!]
\begin{center}
\includegraphics[width=.8\textwidth]{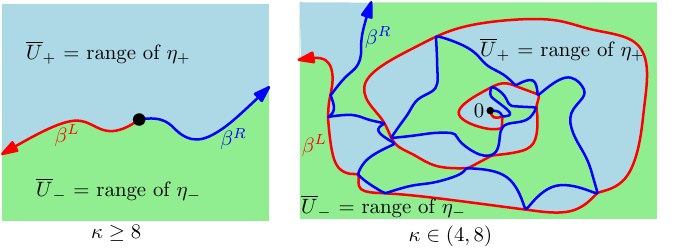}  
\caption{\label{fig-space-filling-def} Illustration of the construction of whole-plane space-filling SLE from the chordal version used in the proof of Proposition~\ref{prop-sle-re-root}, in the cases $\kappa\geq 8$ and $\kappa \in (4,8)$.}
\end{center}
\vspace{-3ex}
\end{figure}
 
The re-rooted curve $\eta^0$ is obtained by concatenating $\eta_+$ followed by the time reversal of $\eta_-$. Let $\phi : \BB C\rta\BB C$ be a fractional linear transformation taking 0 to $\infty$ (i.e., $\phi(w) = a/w$ for some non-zero $a\in\BB C$). We will now argue that when $\kappa = 8$, the curve $\phi\circ\eta^0$ agrees in law with $\eta$ modulo time parametrization. 
\begin{itemize}
\item By the reversibility of whole-plane SLE$_2$~\cite[Theorem 1.20]{ig4} followed by the reversibility of chordal SLE$_2(-1;-1)$~\cite[Theorem 1.1]{ig2}, the time reversals of $\phi \circ \beta^L  $ and $ \phi \circ \beta^R $ have the same joint law as $(\beta^L , \beta^R)$, viewed as curves modulo time parametrization. In particular, $(\phi(U_+) , \phi(U_-)) \eqD (U_- , U_+)$. 
\item By the reversibility of chordal SLE$_8$~\cite[Theorem 1.19]{ig4}, the conditional law given $(\beta^L,\beta^R)$ of the time reversals of $ \phi\circ\eta_- $ and $ \phi\circ\eta_+ $ is that of a pair of conditionally independent chordal SLE$_8$'s from 0 to $\infty$ in $\phi(U_-)$ and $\phi(U_+)$, respectively.
\item Consequently, concatenating $\phi\circ\eta_+$ followed by the time reversal of $\phi\circ\eta_-$ gives us a new curve with the same law as $\eta$, viewed modulo time parametrization. 
\end{itemize}

When $\kappa\not=8$, whole-plane SLE$_{16/\kappa}(2-16/\kappa )$ and chordal SLE$_{16/\kappa}(-8/\kappa ; -8/\kappa)$ are still reversible (by the same references), but chordal space-filling SLE$_\kappa$ is not~\cite[Theorem 1.19]{ig4}. Hence, the exact same argument as above shows that $\phi\circ\eta^0$ does not agree in law with $\eta$ modulo time parametrization when $\kappa \not= 8$. 
\end{proof}

\subsection{Re-rooting invariance for SLE$_8$-decorated $\gamma$-LQG}
\label{sec-sle-lqg-re-root}

Using Proposition~\ref{prop-sle-re-root}, we can obtain the following similar symmetry property for SLE$_8$-decorated LQG. The property is most easily stated in terms of curve-decorated LQG surfaces, which we now define building on Definition~\ref{def-lqg-surface}

\begin{defn} \label{def-curve-decorated}
Let $\gamma\in (0,2)$. A \textbf{$\gamma$-LQG surface with a single marked point, decorated by two curves} is an equivalence class of 5-tuples $(U,h,z, \eta_1,\eta_2)$, where 
\begin{itemize}
\item $U\subset\BB C$ is open and $z \in U\cup \bdy U$, with $\bdy U$ viewed as a collection of prime ends;
\item $h$ is a generalized function on $U$ (which we will always take to be some variant of the Gaussian free field);
\item $\eta_1 : [a_1,b_1] \rta U \cup \bdy U$ and $\eta_2 : [a_2,b_2] \rta U\cup \bdy U$ are curves in $U \cup \bdy U$;
\item $(U,h,z,\eta_1,\eta_2)$ and $(\wt U, \wt h,\wt z, \wt\eta_1 ,\wt\eta_2)$ are declared to be equivalent if there is a conformal map $\phi : \wt U \rta U$ such that
\eqbn
\wt h = h\circ \phi + Q\log|\phi'| ,\quad \phi(\wt z )  = z  ,\quad \phi\circ\wt\eta_1=\eta_1 \quad \text{and} \quad \phi\circ\wt\eta_2=\eta_2
\eqen
where $Q$ is as in~\eqref{eqn-Q}.
\end{itemize} 
\end{defn}

\begin{prop}[Re-rooting invariance for SLE$_8$-decorated $\gamma$-LQG] \label{prop-sle-lqg-re-root}
Let $\gamma\in (0,2)$ and $\kappa_1,\kappa_2 > 4$. Let $(\BB C , h , \infty)$ be a singly marked unit-area quantum sphere. Also let $\eta_1 , \eta_2$ be two independent whole-plane space-filling SLE curves from $\infty$ to $\infty$, with parameters $\kappa_1,\kappa_2 $, respectively, sampled independently from $h$ and then parametrized by $\gamma$-LQG mass with respect to $h$. Fix $t\in [0,1]$. If $\kappa_1=\kappa_2 = 8$ (and $\gamma\in (0,2)$ is arbitrary), then we have the re-rooting invariance property 
\eqb \label{eqn-sle-lqg-re-root}
\left( \BB C , h , \eta_1(t) , \eta_1^{\eta_1(t)} , \eta_2^{\eta_1(t)} \right) \eqD \left(\BB C ,h , \infty , \eta_1,\eta_2 \right) 
\eqe
as curve-decorated quantum surfaces. This property is not true if either $\kappa_1$ or $\kappa_2$ is not equal to 8. 
\end{prop}

Before giving the proof of Proposition~\ref{prop-sle-lqg-re-root}, we explain why it implies Theorem~\ref{thm-permuton-re-root}.

\begin{proof}[Proof of Theorem~\ref{thm-permuton-re-root} assuming Proposition~\ref{prop-sle-lqg-re-root}] 
For $t \in [0,1]$, define the permuton $\perm_t$ as in the theorem statement. Almost surely, $\eta_1(t)$ it not a multiple point for $\eta_1$. Furthermore, by Lemma~\ref{lem-permuton-unique}, a.s.\ $\psi(t)$, as defined in the theorem statement, is the unique time at which $\eta_2$ hits $\eta_1(t)$. Hence, by Definition~\ref{defn-curve-re-root} of the re-rooting operation,
\eqb \label{eqn-curves-shift}
\eta_1(u) = \eta_1^{\eta_1(t)}\left( u - t - \lfloor u- t \rfloor \right) \quad \text{and} \quad
\eta_2(u) = \eta_2^{\eta_1(t)}\left( u - \psi(t) - \lfloor u- \psi(t) \rfloor \right) ,\quad \forall u \in [0,1] .
\eqe 
By~\eqref{eqn-curves-shift}, Lemma~\ref{lem-permuton-defined}, and the definition of $\perm_t$, a.s.\ for every rectangle $[a,b]\times[c,d] \subset [0,1]^2$, 
\eqb
\perm_t\left( [a,b] \times[c,d] \right) = \mu_h\left( \eta_1^{\eta_1(t)}\left( [a,b] \right) \cap \eta_2^{\eta_1(t)}\left( [c,d] \right) \right) .
\eqe
By Proposition~\ref{prop-sle-lqg-re-root}, we therefore have $\perm_t \eqD \perm$. 
\end{proof}

It remains to prove Proposition~\ref{prop-sle-lqg-re-root}. 
We first prove the re-rooting invariance property of Proposition~\ref{prop-sle-lqg-re-root} when we re-root at a uniformly random time, rather than a deterministic time. This re-rooting invariance property is easier to prove because if $T \in [0,1]$ is sampled uniformly from Lebesgue measure, independently from everything else, then $\eta(T)$ is a sample from $\mu_h$, so we can apply Lemma~\ref{lem-sphere-pt} and Proposition~\ref{prop-sle-re-root}. 

\begin{lem}[Re-rooting invariance for SLE$_8$-decorated $\gamma$-LQG, random time] \label{lem-sle-lqg-re-root0}
Let $(\BB C ,h , \infty)$, $\eta_1$, and $\eta_2$ be as in Proposition~\ref{prop-sle-lqg-re-root}. 
Let $T $ be sampled uniformly from Lebesgue measure on $[0,1]$, independently from everything else. If $\kappa_1=\kappa_2 = 8$ (and $\gamma\in (0,2)$ is arbitrary), then  
\eqb \label{eqn-sle-lqg-re-root0}
\left( \BB C , h , \eta_1(T) , \eta_1^{\eta_1(T)} , \eta_2^{\eta_1(T)} \right) \eqD \left(\BB C ,h , \infty , \eta_1,\eta_2 \right) 
\eqe
as curve-decorated quantum surfaces. This property is not true if either $\kappa_1$ or $\kappa_2$ is not equal to 8. 
\end{lem}
\begin{proof}
Write $Z = \eta_1(T)$. Since $\eta_1$ is parametrized by $\gamma$-LQG mass $\mu_h$ with respect to $h$, the point $Z$ is a uniformly sample from $\mu_h$, independent from $\eta_1$ and $\eta_2$. Since the marked point for a single marked quantum sphere is sampled uniformly from its LQG area measure (Lemma~\ref{lem-sphere-pt}), we have $(\BB C , h , Z) \eqD (\BB C , h  , \infty)$ as quantum surfaces. 

Since $(h,Z)$ is independent from $\eta_1$ and $\eta_2$, viewed modulo time parametrization, Proposition~\ref{prop-sle-re-root} shows that when $\kappa_1=\kappa_2=8$, the conditional law of $(\eta_1^Z,\eta_2^Z)$ given $(h,Z)$ is that of a pair of conditionally independent whole-plane SLE$_8$ curves from $Z$ to $Z$. Moreover, these curves are parametrized by $\mu_h$-mass. Therefore,~\eqref{eqn-sle-lqg-re-root0} holds when $\kappa_1=\kappa_2=8$. 

If, say, $\kappa_1\not=8$, then Proposition~\ref{prop-sle-re-root} shows that the conditional law of $ \eta_1^Z $ given $(h,Z)$ is \emph{not} that of a whole-plane space-filling SLE$_{\kappa_1}$ from $Z$ to $Z$. Therefore,~\eqref{eqn-sle-lqg-re-root0} does not hold if $\kappa_1 \not=8$. The same argument also works if $\kappa_2\not=8$. 
\end{proof}

\begin{proof}[Proof of Proposition~\ref{prop-sle-lqg-re-root}]
The re-rooting invariance property~\eqref{eqn-sle-lqg-re-root} at a fixed deterministic time immediately implies the re-rooting invariance property~\eqref{eqn-sle-lqg-re-root0} at a time $T$ sampled uniformly from $[0,1]$. Hence, Lemma~\ref{lem-sle-lqg-re-root0} implies that~\eqref{eqn-sle-lqg-re-root} does not hold if either $\kappa_1$ or $\kappa_2$ is not equal to 8. We need to prove that~\eqref{eqn-sle-lqg-re-root} holds when $\kappa_1=\kappa_2=8$. 

Consider the space of 5-tuples $(\BB C ,\frk h , z , P_1,P_2)$ where $\frk h$ is a generalized function, $z\in \BB C\cup \{\infty\}$, and $P_1,P_2 : [0,1]\rta \BB C\cup \{\infty\}$ are curves such that $P_1(0) = P_1(1)$, $P_1([0,1]) = \BB C\cup\{\infty\}$, and the same is true for $P_2$.
For $t\in [0,1]$, we define the re-rooting operator on this space of 5-tuples by 
\eqb
\mcl R_t \left(\BB C , \frk h  , z ,  P_1 , P_2     \right)  
= \left( \BB C , \frk h ,  P_1(t) , P_1^{P_1(t)} , P_2^{P_1(t)} \right)  ,
\eqe
here using the notation of Definition~\ref{defn-curve-re-root}. 
From~\eqref{eqn-curve-re-root}, one can check that if $s,t\in [0,1]$ and neither $P_1(t)$ nor $P_1([s+t])$ is hit more than once by either $P_1$ or $P_2$, then
\eqb \label{eqn-re-root-compose}
\mcl R_s \circ \mcl R_t = \mcl R_{[s+t]} ,\quad \text{where} \quad [s+t] := s+t - \lfloor s + t \rfloor .
\eqe

Now let $t\in [0,1]$ be fixed and let $T$ be sampled uniformly from Lebesgue measure on $[0,1]$, independently from everything else. Then $[t+T]$ is also sampled uniformly from Lebesgue measure on $[0,1]$ independently from everything else. Therefore, if $\kappa_1=\kappa_2 = 8$ then Lemma~\ref{lem-sle-lqg-re-root0} gives
\eqb \label{eqn-sle-lqg-re-root-shift}
\mcl R_T \left(\BB C ,h , \infty , \eta_1,\eta_2 \right)  \eqD  \left(\BB C ,h , \infty , \eta_1,\eta_2 \right) \eqD \mcl R_{[t+T]} \left(\BB C ,h , \infty , \eta_1,\eta_2 \right) ,
\eqe
as curve-decorated quantum surfaces. Since $\eta_1(T)$ and $\eta_1([t+T])$ are sampled uniformly from $\mu_h$, independently from $\eta_1$ and $\eta_2$, a.s.\ each of $\eta_1(T)$ and $\eta_1([t+T])$ is hit only once by each of $\eta_1$ and $\eta_2$. So, we can apply~\eqref{eqn-re-root-compose} followed by the second equality in~\eqref{eqn-sle-lqg-re-root-shift} to obtain
\eqb \label{eqn-sle-lqg-re-root-new}
\mcl R_t \circ \mcl R_T \left(\BB C ,h , \infty , \eta_1,\eta_2 \right) \eqD  \left(\BB C ,h , \infty , \eta_1,\eta_2 \right) .
\eqe 
By the first equality in~\eqref{eqn-sle-lqg-re-root-shift}, the equality in law~\eqref{eqn-sle-lqg-re-root-new} also holds with $\mcl R_t\circ\mcl R_T(\BB C ,h , \infty,\eta_1,\eta_2)$ replaced by $\mcl R_t\left(\BB C ,h , \infty , \eta_1,\eta_2 \right)$ on the left. That is,~\eqref{eqn-sle-lqg-re-root} holds.
\end{proof}

\section{Physics background for the meander conjecture}
\label{sec-meander-conjecture}

In this section we review the developments in statistical physics that lead to Conjecture~\ref{conj-meander}, namely,  the scaling limit of meanders is described by LQG  with matter central charge $\ccM = -4$, decorated by two independent SLE$_8$ curves. Section~\ref{subsec:Hamiltonian} gives a quick overview of the ideas that lead to Conjecture~\ref{conj-meander}, then the consecutive sections review in more detail the reasoning. At the end of the section (Section~\ref{sec-cyclic-conj}), we will also give a heuristic justification of Conjecture~\ref{conj-cyclic-meander} using Conjecture~\ref{conj-meander}.

\subsection{Meanders, Hamiltonian cycles, and spanning trees}\label{subsec:Hamiltonian}

The scaling limits of a large number of 2D statistical physics models are expected to be described by conformal field theories (CFTs); see~\cite{CFT-book} for a review. A key characteristic of a CFT is its central charge $\cc \in \BB R$, which encodes how its partition function responds to deformations of the background metric.  
Although the full field-theoretic description of the scaling limit has only been achieved for a restricted class of models such as the critical Ising model~\cite{BPZ1984}, for a wide range of models, the central charge can be derived from the asymptotic behavior of the partition function of the model on periodic lattices.
For example, by Kirchhoff's matrix tree theorem, the scaling limit of the uniform spanning tree on a large class of 2D lattices should be a CFT with central charge $\cc=-2$.
 
 The Peano curve of a spanning tree on a 2D lattice is  equivalent to the Hamiltonian cycle  on a related directed lattice~\cite{Kasteleyn-Hamiltonian}, which means that the uniform directed Hamiltonian cycle on a large class of directed lattices should also have central charge $\cc=-2$. However, for the uniform undirected Hamiltonian cycle on the honeycomb lattice,  the central charge  turns out to be $\cc=-1$ instead of $\cc=-2$~\cite{BN-FPL,BSY-FPL}.  
The honeycomb lattice is an example of a cubic planar map whose dual graph is a triangulation.
As explained in~\cite{GKN-Hamiltonian}, the central charge discrepancy between directed and undirected Hamiltonian cycles is due to a  special property of the honeycomb lattice, that is, its dual graph is an Eulerian triangulation (i.e.\ each vertex has an even degree). Since this property does not hold for a typical cubic planar map (whose dual is a triangulation), the  uniform random triangulation  decorated by a Hamiltonian cycle on its dual should  be in the same universality class as a uniform spanning tree decorated map. The scaling limit of this latter decorated map is expected to be $\gamma$-LQG with matter central charge $\ccM = -2$ (equivalently, $\gamma\stackrel{\eqref{eqn:lqg-ccM}}{=}\sqrt{2}$; recall the discussion in Remark~\ref{remark-lqg-physics}), decorated by an SLE$_8$ curve; see, e.g.,~\cite[Appendix A]{shef-burger}.

A meander can be viewed as a pair of Hamiltonian cycles on a planar map whose dual is an Eulerian quadrangulation, with the additional constraint that at each vertex the two cycles cross each other. Following the line of reasoning as in~\cite{GKN-Hamiltonian} with a crucial new insight, Di Francesco, Golinelli, and Guitter~\cite{dgg-meander-asymptotics} 
argued that the uniformly sampled meander should be in the same universality class as random planar maps decorated  by a \emph{pair} of independent spanning trees. Since the central charge is believed to be additive for independent models, such planar maps are expected to converge to LQG with matter central charge $\ccM =(-2) + (-2) = -4$ (equivalently, $\gamma\stackrel{\eqref{eqn:lqg-ccM}}{=}\sqrt{\frac13 \left( 17 - \sqrt{145} \right)}$), decorated by a pair of independent SLE$_8$ curves; as stated in Conjecture~\ref{conj-meander}.

 In the rest of this section, we review in more detail the reasoning of~\cite{dgg-meander-asymptotics}.  The key is a subtle relation between the fully packed and the dense O$(n)$ loop model on different lattices. See Figure~\ref{fig-table-On-models} for a summary of the models discussed in the following sections.
 
 \begin{figure}[ht!]
 	\begin{center}
 		\includegraphics[scale=.771]{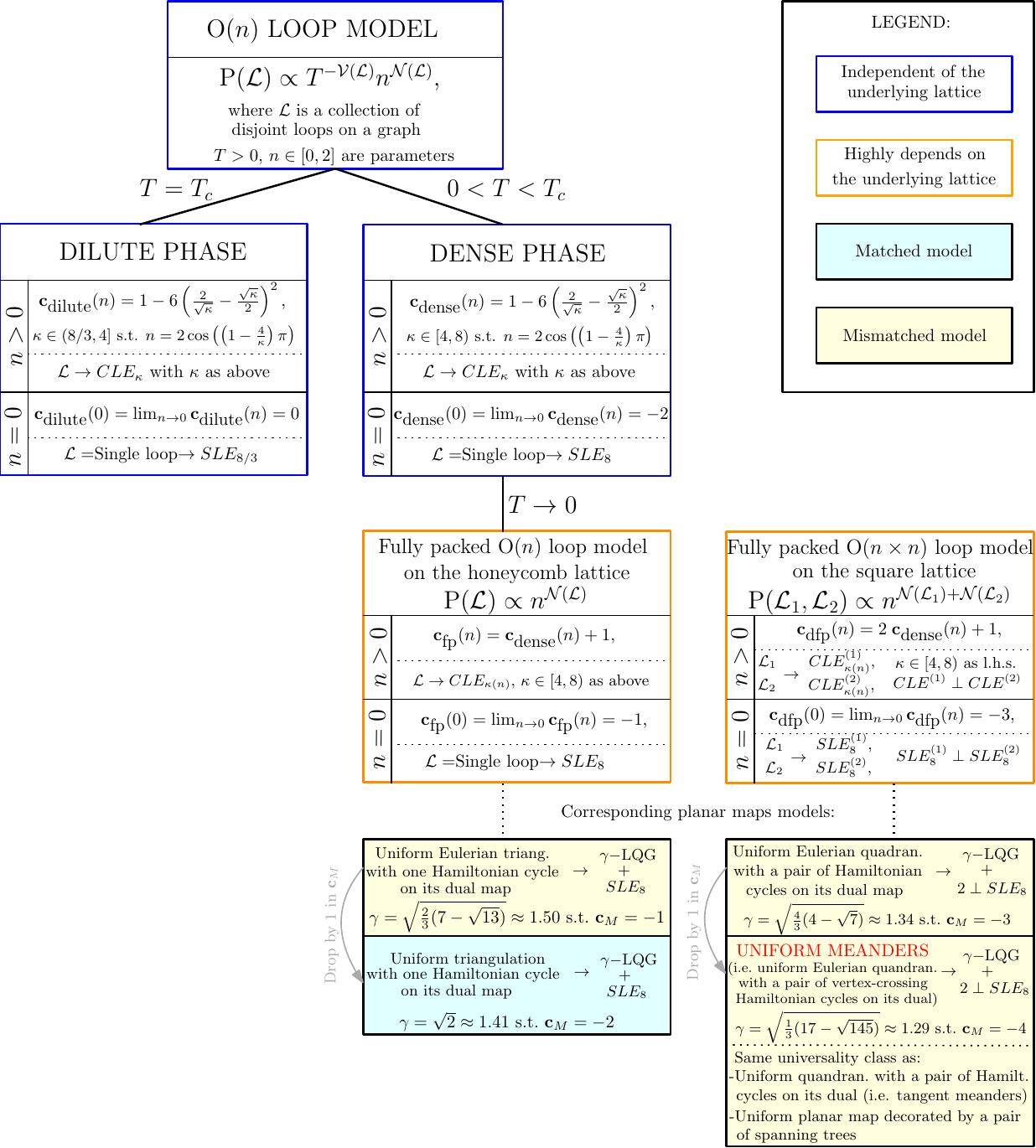}
 		\vspace{-0.01\textheight}
 		\caption{A table summarizing the discussions on the O$(n)$ and related models in Section~\ref{sec-meander-conjecture}.  For each box (i.e.\ for each model), above the dotted line there is the  expression for the central charge of the associated CFT; below the dotted line there is the corresponding scaling limit. The two bottom boxes summarize the conjectures explained in Sections \ref{sec-fully-packed-maps} and \ref{subsec:DGG} for the corresponding planar maps models. We recall that the relation between $\gamma$ and $\cc_M$ is given in \eqref{eqn:lqg-ccM}. A legend for the use of the colors is given on the top right: recall that the definitions of \emph{matched} and \emph{mismatched} models are given in Remark \ref{remark-mismatched}.}
 		\label{fig-table-On-models}
 	\end{center}
 	\vspace{-1em}
\end{figure}

\subsection{Positive temperature O($n$) loop model and its scaling limit}\label{subsec:pos_meas}

For a fixed $n>0$, the O($n$) loop model  on a graph is  the statistical physics model where a configuration $\mathcal L$ is a collection of disjoint loops (i.e.\ simple cycles). Given the  temperature  $T>0$,  the weight of each configuration is $T^{-\mathcal V(\mathcal L)} n^{\mathcal N(\mathcal L)}$, where $\mathcal V(\mathcal L)$ is the total number of vertices occupied by all loops in $\mathcal L$ and $\mathcal N(\mathcal L)$ is the total number of loops in $\mathcal L$. When $n\in (0,2]$, for the O($n$) loop model on a planar  lattice, there exists a critical temperature $T_c>0$ such that the following is true. For $T = T_c$, the scaling limit of the O$(n)$ loop model is conjecturally described by a CFT of central charge 
\eqb
\cc_{\textrm{dilute}} (n) := 1- 6\left(\frac{2}{\sqrt\kappa} -\frac{\sqrt \kappa}{2} \right)^2  , \quad 
\text{where $\kappa \in (8/3,4]$ solves $n=2\cos\left(\left(1-\frac{4}{\kappa}\right)\pi\right)$}.
\eqe
For $0<T<T_c$,   the scaling limit is instead conjecturally described by a CFT with central charge 
\eqb \label{eq:cdense}
\cc_{\textrm{dense}} (n) =1- 6\left(\frac{2}{\sqrt\kappa} -\frac{\sqrt \kappa}{2} \right)^2  ,\quad 
\text{where $\kappa \in [4,8)$ solves $n=2\cos\left(\left(1-\frac{4}{\kappa}\right)\pi\right)$}.
\eqe
These two regimes  are called the \textbf{dilute} and \textbf{dense} phases, respectively. See the lecture notes~\cite{Jacobsen-CFT} for more background for the relation between CFT and the O$(n)$ loop model.

After the discovery of SLE~\cite{schramm0},  it was further conjectured that the collection of loops in the dilute and the dense O($n$) loop models converge to the conformal loop ensemble whose parameter $\kappa\in (\frac{8}{3},8)$ is related to $n$ as above (a conformal loop ensemble (CLE) is a conformally invariant collection of SLE$_\kappa$-type loops~\cite{shef-cle}).
For $n=0$, the O$(n)$ loop model is supported on configurations with $\mathcal N(\mathcal L)=1$, whose weight is $T^{-\mathcal V(\mathcal L)}$. For $T=T_c$,  the model becomes the (critical) self-avoiding polygon and the conjectural scaling limit of the single loop is SLE$_{8/3}$. For $T\in (0, T_c)$ the conjectural limit is SLE$_8$, which coincides with the scaling limit of the uniform spanning tree~\cite{lsw-lerw-ust}. The CFT central charges in the two cases are $\cc_{\textrm{dilute}} (0)=\lim_{n\to 0}c_{\textrm{dilute}} (n)=0$ and   $\cc_{\textrm{dense}} (0)=\lim_{n\to 0}c_{\textrm{dense}} (n)=-2$, respectively.  
All the aforementioned scaling limit conjectures are believed to be independent of the underlying lattice.  See~\cite{Peled-On-model} for more background on the O$(n)$ loop model. See~\cite[Section 2.2]{Smirnov-ICM},~\cite[Section 2.3]{shef-cle} and  references therein for the connection between $O(n)$ loop models and CLE.

\subsection{Fully packed O($n$) loop model and its scaling limit}\label{subsec:fully}

As $T\to 0$, the  O($n$) loop model is supported on configurations that maximize $\mathcal V(\mathcal L)$. This is the so-called \textbf{fully packed} O($n$) loop model. In contrast to the dilute and dense phases,  the scaling limit of the fully packed O($n$) loop model highly depends on the underlying lattice.  
In particular, on the  honeycomb lattice,  Blote and Nienhuis~\cite{BN-FPL} argued that on the one hand, the scaling limit of the loops in the fully packed phase is in the same universality class as those in the dense phase, i.e., it should be described by the same CLE. On the other hand, 
the CFT descibing the scaling limit has central charge 
\begin{equation}\label{eq:cfull}
\cc_{\textrm{fp}} (n)=\cc_{\textrm{dense}} (n)+1,
\end{equation}
where $\cc_{\textrm{dense}}$ is as in~\eqref{eq:cdense}.

The physics argument for~\eqref{eq:cfull} is the following. On the honeycomb lattice, a loop configuration $\mathcal L$ maximizes $\mathcal V(\mathcal L)$ when it covers all vertices. Therefore, the fully packed O$(1)$ loop model is the uniform measure on such configurations.
By~\cite{BN-FPL}, a fully packed configuration $\mathcal L$  on the honeycomb lattice can be mapped to an integer-valued function $h(\mathcal L)$ on the dual triangular lattice, which we call the height function. 
Under this mapping,  the fully packed O$(1)$ loop model on the honeycomb lattice is equivalent to a solid-on-solid (SOS) model which produces the same height function. (Explaining the height function construction would be too much of a digression; see~\cite[Figure 1]{GKN-Hamiltonian} and~\cite[Figure~2-4]{NHB-SOS-Ising} for demonstrations.) 
On the other hand, thanks to a connection to the zero-temperature triangular anti-ferromagnetic Ising model, it was already known from ~\cite{BH-SOS-Ising,BN-Ising} that the SOS model has central charge $\cc=1$, and  the height function $h(\mathcal L)$ is believed to converge to the Gaussian free field. This means that $\cc_{\textrm{fp}} (n=1)=1$. Since $\cc_{\textrm{dense}}(n=1)=0$ by~\eqref{eq:cdense}, we get~\eqref{eq:cfull} when $n=1$. 
For $n\neq 1$, although $h(\mathcal L)$ is determined by $\mathcal L$, it is conjectured in~\cite{BN-FPL}  that $\mathcal L$ and $h(\mathcal L)$ are asymptotically decoupled and their joint scaling limit is equivalent to that of the the superposition of the dense O$(n)$ loop model and the $\cc=1$ SOS model.
Since the central charge is additive for independent models,~\eqref{eq:cfull} holds for all $n\in [0,2]$. This relation is supported  by some numerical evidence in~\cite{BN-FPL}. It was put on a solid ground at the physics level of rigor via the exact Bethe ansatz techniques in~\cite{BSY-FPL}.

For $n=0$,  the fully packed loop model on the  honeycomb lattice is supported on configurations where a single loop visits each vertex, namely, it produces a uniformly sampled Hamiltonian cycle on the  honeycomb lattice. According to~\eqref{eq:cfull}, the central charge of the corresponding CFT is $\cc_{\textrm{fp}}(0)=\lim_{n\to 0}\cc_{\textrm{dense}}(n)+1=-1$. Moreover,  the scaling limit of the Hamiltonian cycle should be an SLE$_8$ curve (which is the limit of CLE$_\kappa$ as $\kappa \nearrow 8$).

\subsection{Fully packed O($n$) loop model on random planar maps}
\label{sec-fully-packed-maps}

It was observed by	Guitter, Kristjansen, and Nielsen~\cite{GKN-Hamiltonian}  that the height function  $h(\mathcal L)$ for a fully packed loop configuration $\mathcal L$ does not necessarily make sense if the honeycomb lattice is replaced by an arbitrary planar map whose dual map is a triangulation. It makes sense if and only if the triangulation is Eulerian.\footnote{We recall for later convenience the following equivalent characterizations of a Eulerian triangulation: (1) a planar map such that all its faces are triangles and the degree of each vertex is even; (2) a planar map such that all its faces are triangles, colored in black and white, so that no two adjacent faces have the same color; (3) a planar map such that its dual map is a planar bicubic map, i.e.\ all its vertices have degree three and these vertices are colored in black and white so that any two adjacent vertices have different colors.} 
%, namely the degree of each vertex is even. 
Recalling the conjectural connection between LQG and random planar maps (Remark~\ref{remark-lqg-physics}), this leads to the conjecture that the scaling limit of a uniform random Eulerian triangulation decorated by a Hamiltonian cycle (i.e.\ by a fully packed O($n=0$) loop model) on its dual map is described by $\gamma$-LQG  with matter central charge 
$$\ccM =\cc_{\textrm{dense}} (n=0)+1= -1$$ 
(equivalently, $\gamma\stackrel{\eqref{eqn:lqg-ccM}}{=} \sqrt{\frac23(7-\sqrt{13})}=\frac{\sqrt{13}-1}{\sqrt{3}}$), decorated by an SLE$_8$ curve. Via the KPZ relation~\cite{kpz-scaling}, it was conjectured and numerically tested in~\cite{GKN-Hamiltonian} that the number of (rooted) Hamiltonian-cycle decorated Eulerian triangulations of size $n$ behaves asymptotically like $A^n n^{-\alpha}$ with $\alpha \stackrel{\eqref{eq:exp}}{=}1 + 4/\gamma^2 = (13+\sqrt{13})/6$.

If we consider an ordinary (not necessarily Eulerian) triangulation decorated by a Hamiltonian cycle (i.e.\ by a fully packed O($n=0$) loop model) on its dual map,  Guitter, Kristjansen, and Nielsen~\cite{GKN-Hamiltonian}  argued that since the $+1$ increment in~\eqref{eq:cfull} disappears, the scaling limit should be  $\gamma$-LQG  with matter central charge $\ccM = \cc_{\textrm{dense}} (n=0)=-2$ (equivalently, $\gamma\stackrel{\eqref{eqn:lqg-ccM}}{=}\sqrt{2}$). In addition to the matrix model justification in~\cite{EGK-Hamiltonian}, we sketch a justification of this assertion using mating of trees~\cite{wedges}. The dual map of a triangulation with a Hamiltonian cycle can be visualized as a  picture similar to Figure~\ref{fig-meander}, where the Hamiltonian cycle is a straight line and other edges on the dual map are on the two sides of the straight line (in this setting, the other edges do not necessarily form a loop).  
This produces a pair of non-crossing matchings that can be viewed as two planar trees. They converge to a pair of independent continuum random trees.
By~\cite[Theorem 1.9]{wedges}, the mating of these trees gives LQG with $\ccM =-2$ (equivalently, $\gamma\stackrel{\eqref{eqn:lqg-ccM}}{=}\sqrt{2}$) decorated by an SLE$_8$ curve which describes the scaling limit of the Hamiltonian cycle (i.e., the interface between the two trees). 
{In other words, a uniform triangulation decorated by a Hamiltonian cycle converges to LQG with $\ccM=-2$ decorated by SLE$_8$ in the so-called \emph{peanosphere sense}~\cite[Section 2.1.1]{wedges}. This type of convergence is not sufficient to imply convergence, e.g., in the Gromov-Haudorff sense or under an appropriate embedding.}

\begin{remark} \label{remark:ddgg}
Although it is not needed for the justification of Conjecture~\ref{conj-meander}, we note that, as explained in Section 3.3 of the recent paper~\cite{ddgg-hamiltonian}, the conjectures in this section can be extended to a triangulation decorated by the fully packed O($n$) loop model on its dual map for $n > 0$. 
In particular, for an Eulerian triangulation decorated by a fully packed O($n$) loop model on its dual map for $n\in(0,1]$,\footnote{As explained in \cite[Section 3.3]{ddgg-hamiltonian} in the case of a uniform random Eulerian triangulation one can consider the fully packed O($n$) loop model only in the range $n\in[0,1]$. Indeed, the case $n\in(1,2]$ should degenerate in the scaling limit to a tree-like structure.} the scaling limit should be described by $\gamma(n)$-LQG  with matter central charge 
\[
\ccM (n) = \cc_{\textrm{dense}} (n)+1\stackrel{\eqref{eq:cdense}}{=}2- 6\left(\frac{2}{\sqrt\kappa} -\frac{\sqrt \kappa}{2} \right)^2, \quad\text{where $\kappa \in [6,8)$ solves } n=2\cos\left(\left(1-\frac{4}{\kappa}\right)\pi\right)
\]
(equivalently, $\gamma(n)\stackrel{\eqref{eqn:lqg-ccM}}{=} \frac{1}{\sqrt{12}}\left(\sqrt{3\left(\kappa+\frac{16}{\kappa}\right)+22}-\sqrt{3\left(\kappa+\frac{16}{\kappa}\right)-26}\right)\in\left(\frac{\sqrt{13}-1}{\sqrt{3}},2\right]$), decorated by a CLE$_{\kappa(n)}$. (Recall that the limit of CLE$_\kappa$ as $\kappa \nearrow 8$ is SLE$_8$, making this conjecture consistent with the $n=0$ case described above.)
For an ordinary (not necessarily Eulerian) triangulation decorated by a fully packed O($n$) loop model on its dual map for $n\in(0,2]$, the scaling limit should be described by $\gamma(n)$-LQG  with matter central charge 
\[
\ccM (n) = \cc_{\textrm{dense}} (n)\stackrel{\eqref{eq:cdense}}{=}1- 6\left(\frac{2}{\sqrt\kappa} -\frac{\sqrt \kappa}{2} \right)^2, \quad\text{where $\kappa \in [4,8)$ solves } n=2\cos\left(\left(1-\frac{4}{\kappa}\right)\pi\right)
\]
(equivalently, $\gamma(n)\stackrel{\eqref{eqn:lqg-ccM}}{=} \frac{4}{\sqrt{k}}\in(\sqrt{2},2]$), decorated by a CLE$_{\kappa(n)}$.
Note that $\gamma$ and $\kappa$ are mismatched (in the sense of Remark~\ref{remark-mismatched}) for random Eulerian triangulations decorated by a fully packed O($n$) loop model on their dual map for $n\in[0,1]$, and are matched for ordinary triangulations decorated by a fully packed O($n$) loop model on their dual map for $n\in[0,2]$.
%\td{In the conjectures above, should SLE$_{\kappa(n)}$ be CLE$_{\kappa(n)}$? Also, I don't understand why there is a reference to (6.3) above the first equality in each of the two new indented equations. Isn't the relationship between the central charges part of the conjecture, not a consequence of (6.3)?}
\end{remark}

\subsection{Meanders and the fully packed O$(n\times m)$ loop model}\label{subsec:DGG}

Di Francesco, Golinelli, and Guitter~\cite{dgg-meander-asymptotics} arrived at their meander conjecture following a similar reasoning as in~\cite{GKN-Hamiltonian}. They first considered  the \textbf{(double) fully packed O}$(n\times m)$ \textbf{loop model} on the square lattice $\BB Z^2$  
with $n,m\in [0,2]^2$. In this model, each configuration is a pair of  loop collections $(\mathcal L_1,\mathcal L_2)$ such that  every edge is visited by either a loop in  $\mathcal L_1$ or a loop in $\mathcal L_2$ but not both, while every vertex  is visited by both types of loops. For $n,m\in (0,2]^2$, the weight of each configuration is $n^{\mathcal N(\mathcal L_1)}m^{\mathcal N(\mathcal L_2)}$, where, as above, $\mathcal N(\mathcal L_i)$ is the number of loops in $\mathcal L_i$ for $i=1,2$. For $n=m=0$, the model is supported on configurations with $\mathcal N(\mathcal L_1)=\mathcal N(\mathcal L_2)=1$ and assigns equal weight to them.  

Again using a mapping to a height function model, Jacobsen and Kondev~\cite{JK-FLP2,JK-DPL} made a similar conjecture as~\eqref{eq:cfull} on the central charge $\cc_{\textrm{dfp}}(n,m)$ of the scaling limit of the fully packed O$(n\times m)$ loop model on $\BB Z^2$. 
In order to define such a mapping to a height function model, Jacobsen and Kondev first considered the uniform  measure on configurations $(\vec{\mathcal L_1},\vec{\mathcal L_2})$ on  $\BB Z^2$, where $(\vec{\mathcal L_1},\vec{\mathcal L_2})$ are two collections of oriented loops satisfying the same conditions above. 
As explained in~\cite[Section 3]{dgg-meander-asymptotics}, Jacobsen and Kondev~\cite{JK-FLP2,JK-DPL} mapped $(\vec{\mathcal L_1},\vec{\mathcal L_2})$ on $\mathbb Z^2$ to a 3D height function $(h_1(\vec{\mathcal L_1},\vec{\mathcal L_2}),h_2(\vec{\mathcal L_1},\vec{\mathcal L_2}),h_3(\vec{\mathcal L_1},\vec{\mathcal L_2}))$, which they conjectured to converge to three independent Gaussian free fields, hence a conformal theory
with central charge $c = 3$. 
Introducing some specific weights for the oriented configurations $(\vec{\mathcal L_1},\vec{\mathcal L_2})$, Jacobsen and Kondev also determined a general expression for the central charge of the scaling limit of this new weighted oriented model (see \cite[Eq.\ (6.4)]{JK-FLP2}). Using a simple relation (given in \cite[Eq.\ (4.1)]{JK-FLP2}) between the weights of the oriented configurations $(\vec{\mathcal L_1},\vec{\mathcal L_2})$ and the weights $(n,m)$ of the fully packed O$(n\times m)$ loop model, i.e.\ of the unoriented configurations $(\mathcal L_1,\mathcal L_2)$, they finally conjectured that the central charge of the scaling limit of the fully packed O$(n\times m)$ loop model on $\BB Z^2$ is given by:
\begin{equation}\label{eq:FPL2}
	\cc_{\textrm{dfp}}(n,m)=\cc_{\textrm{dense}} (n)+\cc_{\textrm{dense}} (m) +1.
\end{equation}
They also conjectured that $\mathcal L_1$ and $\mathcal L_2$ asymptotically decouple, i.e., their scaling limits are independent.
The relation in \eqref{eq:FPL2} has also been justified at a physics level of rigor via the exact Bethe ansatz techniques~\cite{DCN-FPL-exact,JZJ-FPL-exact}. For $n=m=0$, the condition $\mathcal N(\mathcal L_1)=\mathcal N(\mathcal L_2)=1$ means that we have two Hamiltonian cycles on the square lattice whose edge sets are disjoint.
According to~\eqref{eq:FPL2}, the central charge of the limiting CFT is $\cc_{\textrm{dfp}}(0,0)=2\,\cc_{\textrm{dense}} (0) +1=-3$. Moreover, the two Hamiltonian cycles should converge to two independent SLE$_8$ curves (following the same reasoning in the discussion after \eqref{eq:cdense}). 

Recall from Section~\ref{sec-meander} that a meander can be viewed as a triplet $(M,P^1,P^2)$, where $P^1,P^2$ are two Hamiltonian cycles with disjoint edge sets on a 4-regular planar map $M$, i.e.\ a planar map with all vertices of degree four. Moreover, at each vertex the two cycles $P^1$ and $P^2$ cross each other. This implies that the edges on the boundary of each face of the dual map of $M$ alternate between dual edges which cross edges of $P^1$ and dual edges which cross edges of $P^2$. Hence, each face of $M$ has even degree, and so the dual of $M$ is an Eulerian quadrangulation.

As explained in~\cite[Section 3]{dgg-meander-asymptotics}, the height function mapping of Jacobsen and Kondev described in the previous section makes sense on a general planar map $G$ (instead of $\mathbb Z^2$) if and only if the dual map of $G$ is an Eulerian quadrangulation. Since the dual of $M$ is an Eulerian quadrangulation, one can associate to $(M,P^1,P^2)$ a height function $(h_1,h_2,h_3)$ as above.  However,  the additional constraint that the two cycles must cross at each vertex  yields that $(h_1,h_2,h_3)$ must lie on a 2D subspace. This effect exactly cancels the $+1$ in~\eqref{eq:FPL2}. (See \cite[Figure 2]{dgg-meander-asymptotics} and the discussion below that figure.)
Therefore, the scaling limit should be $\gamma$-LQG with $\ccM = 2\,\cc_{\op{dense}}(0) = -4$ (equivalently, $\gamma \stackrel{\eqref{eqn:lqg-ccM}}{=} \sqrt{\frac13(17-\sqrt{145})}$), decorated by two independent SLE$_8$ curves as in Conjecture~\ref{conj-meander}. This leads to several enumeration conjectures in~\cite{dgg-meander-asymptotics} on meanders which are supported by numerical evidence  in the  follow-up works~\cite{DFGJ-numerical,Jensen-num-meanders}. In particular, they conjectured that the number of meanders grows as $A^n n^{-\alpha}$, where $A \approx 12.26$ and $\alpha \stackrel{\eqref{eq:exp}}{=} 1+4/\gamma^2 =  (29 + \sqrt{145})/12$.

Although not mentioned explicitly in~\cite{dgg-meander-asymptotics}, their argument naturally leads to the following scaling limit conjectures in addition to Conjecture~\ref{conj-meander}.
Consider a uniform 4-regular planar map decorated by a pair of edge-disjoint Hamiltonian cycles, without the constraint that the cycles cross at each vertex. Such configurations are called \emph{tangent meanders} in \cite{DFGJ-numerical}. The scaling limit should be described by $\gamma$-LQG with $\ccM=-4$ (equivalently, $\gamma \stackrel{\eqref{eqn:lqg-ccM}}{=} \sqrt{\frac13(17-\sqrt{145})}$).  If we further condition on the event that the dual quandrangulation is Eulerian, then we should instead get $\gamma$-LQG with $\ccM=-3$ (equivalently, $\gamma \stackrel{\eqref{eqn:lqg-ccM}}{=} \sqrt{\frac 4 3 (4-\sqrt 7)}$). 
%\td{I think there is a problem here: the dual of a quandrangulation is always Eulerian.}
Finally, consider a uniformly sampled pair of Hamiltonian cycles on $\mathbb Z^2$ with the condition that they form a meander. Then the central charge for the corresponding CFT should be $\cc=-4$. In all cases, the two Hamiltonian cycles  
should converge to two independent SLE$_8$ curves.

As in Remark \ref{remark:ddgg}, the same reasoning as in the conjecture for meanders can be extended to get a scaling limit conjecture for a 4-regular planar map decorated by the fully packed O$(n \times m)$ loop model for $(n, m) \in [0, 2]^2$, subject to the constraint that the loops cross at each vertex.  Recall that $\cc_{\textrm{dense}} (n) \stackrel{\eqref{eq:cdense}}{=}1- 6\left(\frac{2}{\sqrt{\kappa(n)}} -\frac{\sqrt{\kappa(n)}}{2} \right)^2$, where $\kappa(n) \in [4,8]$ solves $n=2\cos\left(\left(1-\frac{4}{\kappa(n)}\right)\pi\right)$. The following conjecture follows from exactly the same reasoning as the meander conjecture above.

\begin{conj} \label{conj-nm}
	For $k\in\mathbb N$, let $\mathcal C_{k}$ be the set of triples $(M ,\mathcal L_1, \mathcal L_2)$, where $M $ is a 4-regular planar map with $2k$ vertices; and $(\mathcal L_1, \mathcal L_2)$ is a pair of collections of loops on $M $ with the following properties:
	\begin{itemize} 
		\item every edge of $M $
		is visited by exactly one loop in $\mathcal L_1 \cup \mathcal L_2$; 
		\item every vertex of $M$ is visited by exactly one loop in $\mathcal L_1$
		and exactly one loop in $\mathcal L_2$;
		\item every pair of loops $(\ell_1, \ell_2) \in \mathcal L_1 \times \mathcal L_2$ cross each other at every shared vertex (i.e., the edges of $\ell_1$ and $\ell_2$ in cyclic order around the vertex alternate).
	\end{itemize}
	Fix $(n, m) \in [0, 2]^2$ such that\footnote{If $\cc_{\textrm{dense}} (n)+\cc_{\textrm{dense}} (m)> 1$, the planar map should degenerate in the scaling limit to a tree-like structure.} $\cc_{\textrm{dense}} (n)+\cc_{\textrm{dense}} (m)\leq 1$.
	Let $(M ,\mathcal L_1, \mathcal L_2)$ be sampled from the uniform measure on $\mathcal C_k$ weighted by $n^{\mathcal N(\mathcal L_1)}m^{\mathcal N(\mathcal L_2)}$.
	Then $(M ,\mathcal L_1, \mathcal L_2)$ converges under an appropriate scaling limit to a Liouville quantum gravity sphere with matter central charge
	$$\ccM (n,m) = \cc_{\mathrm{dense}} (n)+\cc_{\mathrm{dense}} (m),$$
	equivalently, with coupling constant $\gamma(n,m)$ satisfying $\ccM(n,m) \stackrel{\eqref{eqn:lqg-ccM}}{=} 25-6\,\left(2/\gamma(n,m) + \gamma(n,m)/2\right)^2$,
	%$$\frac{\sqrt{48 (\kappa(m) + \kappa(n)) + \kappa(n)\kappa(m) \big( 3 (\kappa(n)+\kappa(m))-2\big)}-\sqrt{48 (\kappa(m) + \kappa(n)) + \kappa(n)\kappa(m) \big( 3 (\kappa(n)+\kappa(m))-50\big)}}{\sqrt{12 \, \kappa(n) \, \kappa(m)}} ,$$ 
	decorated by a whole-plane CLE$_{\kappa(n)}$ and an independent whole-plane CLE$_{\kappa(m)}$. If $\kappa(n)=8$ or $\kappa(m)=8$ then CLE$_8$ is interpreted as SLE$_8$.
\end{conj}

\subsection{Convergence of the cyclic meandric permutation to the identity}
\label{sec-cyclic-conj}

Recall Conjecture~\ref{conj-cyclic-meander}, which asserts that the cyclic meandric permutation $\tau_{\ell_n}$ associated with a uniform meander $\ell_n$ of size $n$ converges to the identity permuton. 
In this subsection we will give a heuristic justification of Conjecture~\ref{conj-cyclic-meander} using Conjecture~\ref{conj-meander}. To prove Conjecture~\ref{conj-cyclic-meander}, one would need to show that with high probability when $n$ is large, $| \tau_{\ell_n}(j) - j|$ is of order $o(n)$ for all but all but a negligible fraction of the values of $j\in [1,2n]\cap\BB Z$. 

To see why this should be true, consider the setting of Section~\ref{sec-permuton-def} with $\eta_1,\eta_2$ independent SLE$_8$ curves and $\gamma = \sqrt{\frac13(17-\sqrt{145})}$, so that $\perm$ is the meandric permuton. Let $  \psi  , \phi :[0,1] \rta [0,1]$ be two Lebesgue measurable function chosen so that $\eta_2(t) = \eta_1(\psi(t))$ and $\eta_2(\phi(t)) = \eta_1(t)$. We claim that a.s.\ for Lebesgue-a.e.\ $t\in [0,1]$, 
\eqb \label{eqn-composition-to-zero}
\lim_{\ep\rta 0} \sup_{s\in [-\ep , \ep]} |\psi(\phi(t) + s) - t| = 0 .
\eqe
If we think of $\eta_1$ and $\eta_2$ as the continuum limits of the loop and the line, respectively, for the meander $\ell_n$, then we can think of $\psi$ and $\phi$ as the continuum analogs of $\sigma_{\ell_n}^{-1}$ and $\sigma_{\ell_n}$, respectively. Hence, if we assume Conjecture~\ref{conj-meander}, then~\eqref{eqn-composition-to-zero} should imply that for a typical $t\in [0,1]$, 
\eqb \label{eqn-cyclic-to-zero}
 \lim_{\ep\rta 0} \lim_{n\rta\infty} \sup_{k \in [-2\ep n , 2\ep n] \cap\BB Z} \frac{1}{2n} \left|\sigma_{\ell_n}^{-1}(\sigma_{\ell_n}(\lfloor 2t n \rfloor) + k) - 2 t n \right| = 0 . 
\eqe 
By the formula $\tau_{\ell_n}(j) = \sigma_{\ell_n}^{-1}(\sigma_{\ell_n}(j)+1)$ from~\eqref{eqn-permutation-relation}, the relation~\eqref{eqn-cyclic-to-zero} for Lebesgue-a.e.\ $t\in [0,1]$ implies the desired convergence of $\tau_{\ell_n}$ to the identity permuton. 

We now explain why~\eqref{eqn-composition-to-zero} is true. For $\mu_h$-a.e.\ $z\in\BB C$, each of $\eta_1$ and $\eta_2$ hits the point $z$ exactly once (c.f.\ the proof of Lemma~\ref{lem-permuton-defined}). Equivalently, for Lebesgue-a.e.\ $t\in [0,1]$, then each of $\eta_1$ and $\eta_2$ hits $\eta_1(t)$ exactly once. For such a time $t$, the point $\eta_1(t)$ is not on the boundary of $\eta_1([t-\delta,t+\delta])$ for any $\delta > 0$ (otherwise, $\eta_1$ would have to hit $\eta_1(t)$ at a time not in $[t-\delta,t+\delta]$). Hence, $\eta_1([t-\delta,t+\delta])$ contains a neighborhood of $\eta_1(t)$. Since $\eta_2$ is a continuous curve and $\eta_2$ hits $\eta_1(t)$ exactly once (namely, at the time $\phi(t)$), there exists $\ep > 0$ such that 
\eqbn
\eta_2([\phi(t)-\ep,\phi(t)+\ep]) \subset \eta_1([t-\delta,t+\delta])  \quad \text{which implies} \quad \sup_{s\in [-\ep , \ep]} |\psi(\phi(t) + s) - t| \leq \delta .
\eqen 
Since $\delta>0$ is arbitrary, this gives~\eqref{eqn-composition-to-zero}.

\section{Open problems}
\label{sec-open-problems}

{We state several additional open problems related to this work, besides proving the conjectures stated in Section~\ref{sec-intro}. At the end of the section, we comment briefly on how hard we expect the problems to be.}

It is difficult to imagine how one would make a direct connection between a subsequential limit of meandric permutons and SLE-decorated LQG. Consequently, in order to prove Conjecture~\ref{conj-meander-permuton} one would likely first need to solve the following problem. 

\begin{prob} \label{prob-char}
Is there a description of the meandric permuton which does not use SLE or LQG?
\end{prob}

Possible solutions to Problem~\ref{prob-char} might include an axiomatic characterization of the meandric permuton or a direct construction of this permuton in terms of Brownian motion (similar to the construction of the skew Brownian permutons in~\cite{borga-skew-permuton}). We emphasize that Conjecture~\ref{conj-meander-permuton} is likely to be difficult to prove even if one finds a nice solution to Problem~\ref{prob-char}, since meanders are notoriously difficult to analyze. 

\medskip

It seems quite difficult to get a lower bound for the length of the longest increasing subsequence for classes of random permutations which are known or expected to converge to the random permutons considered in this paper. Currently, the only available lower bounds are the essentially trivial ones, namely $\op{LIS}(\sigma_n) \geq \op{const} \times |\sigma_n|^{1/2}$. We recall that the Erd\"os-Szekeres theorem~\cite{erdos-szekeres} gives $\op{LIS}(\sigma) \times \op{LDS}(\sigma ) \geq n$ for every $\sigma\in \mcl S_n$. 

\begin{prob} \label{prob-lis-lower}
For $n\in\BB N$, let $\sigma_n$ be a uniform Baxter, semi-Baxter, strong-Baxter, or meandric permutation of size $n$.
Show that there exists $\alpha >1/2$ (depending only on the type of permutation, not on $n$) such that $\op{LIS}(\sigma_n) \geq n^\alpha$ with high probability when $n$ is large.
\end{prob}

{We note that the analog of Problem~\ref{prob-lis-lower} in the case of permutations sampled from the Brownian separable permuton (the $\rho \to 1$ limit of skew Brownian permutons) is solved in~\cite{bdg-separable-lis}. }
Problem~\ref{prob-lis-lower} is closely related to the following SLE/LQG problem.

\begin{prob} \label{prob-lis-sle}
Assume that we are in the setting of Section~\ref{sec-permuton-def}. What is the maximal Hausdorff dimension (with respect to the LQG metric $D_h$) of a set $X\subset \BB C$ which is hit in the same order by $\eta_1$ and $\eta_2$? 
\end{prob}

If $\Delta$ is the maximal Hausdorff dimension in Problem~\ref{prob-lis-lower} and $d_\gamma$ is the Hausdorff dimension of $\BB C$ w.r.t.\ $D_h$, then it is natural to conjecture that $\op{LIS}(\sigma_n) \approx |\sigma_n|^{\Delta/d_\gamma}$ for sufficiently ``nice'' sequences of permutations $\{\sigma_n\}_{n\in\BB N}$ which converge in law to $\perm$. 

\medskip

The following question is motivated by the formula for pattern densities in Theorem~\ref{thm:density}. 
We are especially interested in the case when $\kappa = 8$, since this case corresponds to the meandric permuton.

\begin{prob} \label{prob-sle-order}
Let $\kappa > 4$ and let $\eta$ be a whole-plane space-filling SLE$_\kappa$ from $\infty$ to $\infty$. Also let $z_1,\dots,z_n\in\BB C$ be distinct points. Is there an exact formula for the function 
\eqb
P_\kappa(z_1,\dots,z_n) = \BB P\left[\text{$\eta$ hits $z_j$ before $z_{j+1}$, $\forall j =1,\dots,n-1$} \right] ?
\eqe
\end{prob}

\medskip

Recall Conjecture~\ref{conj-cyclic-meander}, which asserts that a uniform \emph{cyclic} meandric permutation converges to the identity permuton. We explained in Section~\ref{sec-cyclic-conj} why this conjecture should follow from Conjecture~\ref{conj-meander}. 

\begin{prob} \label{prob-cyclic}
Is there an unconditional proof of Conjecture~\ref{conj-cyclic-meander} which does not use Conjecture~\ref{conj-meander}?
\end{prob} 

\medskip
 
The techniques and many of the results of this paper apply to a wide range of different random permutons which can be described in terms of SLE and LQG. It is therefore natural to ask the following. 

\begin{prob} \label{prob-other}
Are there any other interesting random permutons, besides the meandric permuton and the skew Brownian permutons, which can be described by SLE-decorated LQG in the manner of Section~\ref{sec-permuton-def}?
\end{prob}

\medskip

The physics background provided in Section~\ref{sec-meander-conjecture} motivates several conjectures about scaling limits. We highlight the following problem, which is motivated by the arguments of Sections~\ref{sec-fully-packed-maps}~and~\ref{subsec:DGG}.

\begin{prob} \label{prob-phy}
Prove that	the scaling limit of a uniform Eulerian triangulation decorated by a Hamiltonian cycle on its dual  is described by LQG  with matter central charge $\ccM = -1$ (i.e., $\gamma = \sqrt{\frac23(7-\sqrt{13})}$) decorated by an SLE$_8$ curve. Moreover, prove the other scaling limit conjectures from Section~\ref{subsec:DGG} (see also Conjecture~\ref{conj-nm} and the two bottom boxes in the table in Figure~\ref{fig-table-On-models}).
\end{prob}

{All of the conjectures and open problems stated in this paper seem to be quite difficult: we have spent some time on each of them and were unable to solve them. However, we expect that some are likely to be more accessible than others. 

Probably the easiest of the open problems in this paper is Problem~\ref{prob-lis-lower}, but it still seems to require some non-trivial new ideas. Of course, computing the exact exponent for the LIS (or the exact optimal dimension in Problem~\ref{prob-lis-sle}) is likely to be much harder than just proving a non-trivial lower bound. Problem~\ref{prob-sle-order} seems like it could also conceivably be accessible with current techniques. 

Typically, proving that a discrete model converges to SLE-decorated LQG  requires some sort of combinatorial miracle to get a handle on the discrete model, plus a substantial amount of additional work to deduce convergence from such a combinatorial miracle. Consequently, we expect that Conjectures~\ref{conj-meander}, \ref{conj-meander-permuton}, and~\ref{conj-nm} and Problem~\ref{prob-phy} are the most difficult of the conjectures and open problems stated in this paper. Proving convergence of meandric permutations (Conjecture~\ref{conj-meander-permuton}) might be easier than proving convergence of meanders to SLE-decorated LQG in, e.g., the Gromov-Hausdorff sense since one gets tightness for free, so one just has to identify the scaling limit. This would likely require solving Problem~\ref{prob-char}, which seems hard, but which is a purely continuum problem so could conceivably be done without finding any combinatorial miracle for meanders. One would also need some new tools for analyzing uniform meanders in order to check the hypotheses of any hypothetical characterization theorem for the subsequential limits.

Conjecture~\ref{conj-cyclic-meander} and Problem~\ref{prob-cyclic} appear to require some new insights about meanders, but could likely be done with much less ``integrability'' than what would be needed to get convergence to SLE and LQG.  }

\bibliography{cibib,cibib2}
\bibliographystyle{hmralphaabbrv}

\end{document}